%% file: eig_clusters_hal.tex
\newif \ifPdfTex
\PdfTextrue

\newif \ifSIAM
\SIAMfalse

\ifSIAM\documentclass[final]{siamltex} \else
\documentclass[10pt,a4paper]{article}
\fi

\usepackage{pifont}
\usepackage{graphicx}
\usepackage{amssymb}
\usepackage{amsmath}
\usepackage{booktabs}
\usepackage{paralist}
\usepackage{bm}
\usepackage{url}
\usepackage[usenames]{color}
\usepackage[usenames,dvipsnames]{xcolor}
\usepackage{mathtools}
\usepackage{physymb}
\usepackage{bbding}
\usepackage{multirow}
\usepackage{enumitem}
\usepackage{authblk}
\usepackage{cases}
\usepackage{array}
\usepackage[normalem]{ulem} 
\usepackage{cancel} 
\usepackage[textsize=tiny,color=blue!50!white]{todonotes}
\usepackage{amsthm}
\renewcommand{\S}{Section~}

\numberwithin{equation}{section}

\newif \iffract
\fracttrue

\newif \ifPDF
\PDFtrue
\newcommand{\be}{\begin{equation}}
\newcommand{\ee}{\end{equation}}
\def\ba#1\ea{\begin{align}#1\end{align}}
\def\ban#1\ean{\begin{align*}#1\end{align*}}
\def\bat#1\eat{\begin{alignat}#1\end{alignat}}
\def\batn#1\eatn{\begin{alignat*}#1\end{alignat*}}
\def\bs#1\es{\begin{split}#1\end{split}}
\newcommand{\bse}{\begin{subequations}}
\newcommand{\ese}{\end{subequations}}
\newcommand{\bt}{\begin{theorem}}
\newcommand{\et}{\end{theorem}}
\newcommand{\bl}{\begin{lemma}}
\newcommand{\el}{\end{lemma}}
\newcommand{\bc}{\begin{corollary}}
\newcommand{\ec}{\end{corollary}}
\newcommand{\bp}{\begin{proof}}
\newcommand{\ep}{\end{proof}}
\newcommand{\bd}{\begin{definition}}
\newcommand{\ed}{\end{definition}}
\newcommand{\br}{\begin{remark}}
\newcommand{\er}{\end{remark}}
\newcommand{\bas}{\begin{assumption}}
\newcommand{\eas}{\end{assumption}}
\newcommand{\bex}{\begin{example}}
\newcommand{\eex}{\end{example}}
\newcommand{\bqo}{\begin{quote}}
\newcommand{\eqo}{\end{quote}}
\newcommand{\bdc}{\begin{description}}
\newcommand{\edc}{\end{description}}
\newcommand{\bi}{\begin{itemize}}
\newcommand{\ei}{\end{itemize}}
\newcommand{\ben}{\begin{enumerate}}
\newcommand{\een}{\end{enumerate}}
\def\ba#1\ea{\begin{align}#1\end{align}}

\def\ch{{c_h}}
\def\chb{{\bar c_h}}
\def\cht{{\tilde c_h}}

\def\chn{{c}_h}
\def\chtn{\tilde {c}_h}

\newtheorem{lemma}{Lemma}[section]
\newtheorem{corollary}[lemma]{Corollary}
\newtheorem{assumption}[lemma]{Assumption}
\newtheorem{theorem}[lemma]{Theorem}
\newtheorem{definition}[lemma]{Definition}
\newtheorem{remark}[lemma]{Remark}
\newtheorem{example}[lemma]{Example}

\newif \ifwhere

\newcommand\ie{i.e.}
\newcommand\eg{e.g.}
\newcommand\cf{cf.}
\newcommand\eal{{\em et al.}}
\newcommand\eq{:=}
\newcommand\qe{=:}

\newcommand\pt{\partial}

\DeclarePairedDelimiter\newnorm{\|}{\|}
\renewcommand{\norm}{\newnorm}
\newcommand{\VpV}[2]{\langle {#1},{#2}\rangle_{D(A^{1/2})', D(A^{1/2})}}
\newcommand{\VpVS}[2]{\langle {#1},{#2}\rangle_{H^{-1}_\#(\Omega), H^1_\#(\Omega)}}

\newcommand\VN{{\mathcal{V}_N}}

\newcommand\eres{\eta_{\mathrm{res}}}
\newcommand\etaltwo{\eta_{{L^2}}}
\newcommand\scp{{\cdot}}
\newcommand\Gr{\nabla}

\newcommand\Dv{\nabla {\cdot}}
\newcommand\dv{\mathrm{div}}
\newcommand\Lap{\Delta}
\newcommand\Th{\mathcal{T}_h}
\newcommand\Vh{\mathcal{V}_h}
\newcommand\Vhint{\mathcal{V}^{\mathrm{int}}_h}
\newcommand\Vhext{\mathcal{V}^{\mathrm{ext}}_h}

\newcommand\Ta{{\mathcal{T}_{\ver}}}
\newcommand\Om{\Omega}
\newcommand\oma{{\omega_\ver}}
\newcommand{\elm}{{K}}

\newcommand{\ver}{{\ta}}

\newcommand\li{\lambda_i}
\newcommand\lk{\lambda_k}
\newcommand{\la}[1]{\lambda_{#1}}
\newcommand\ola[1]{\overline\lambda_{#1}}
\newcommand{\ula}[1]{{\underline\lambda_{#1}}}

\newcommand\lih{\lambda_{ih}}
\newcommand\lkh{\lambda_{kh}}
\newcommand{\lah}[1]{\lambda_{#1 h}}
\newcommand{\lan}[1]{\lambda_{#1 N}}

\newcommand{\ue}[1]{\varphi_{#1}^0}

\newcommand{\uh}[1]{\varphi_{#1 h}}
\newcommand{\uho}[1]{\varphi_{#1 h}^0}
\newcommand{\un}[1]{\varphi_{#1 N}}
\newcommand\uih{\varphi_{ih}}
\newcommand\ujh{\varphi_{jh}}
\newcommand\ukh{\varphi_{kh}}

\newcommand{\Phie}{\vec{\Phi}^0}
\newcommand{\Phih}{\vec{\Phi}_h}
\newcommand{\Phiho}{\vec{\Phi}_h^0}

\newcommand{\Res}{{\rm Res}}
\newcommand\res{\scriptr_{(ih)}}

\newcommand\Resg{{{\rm Res}_{m,M}^h}}
\newcommand\Resc{{\rm Res}(\gamh)}
\newcommand\Tr{{{\rm Tr}}}
\newcommand\Span{{{\rm Span}}}
\newcommand\frh[1]{{\bm \sigma}_{#1 h}}
\newcommand\Dt{{H^{-1}(\Omega)}}

\newcommand\kk{{\bf k}}
\newcommand\xx{{\vec{x}}}

\newcommand\ta{{\bf a}}
\newcommand\tn{{\bf n}}

\newcommand\tv{{\bf v}}
\newcommand\tx{{\bf x}}

\newcommand\tV{{\bf V}}
\newcommand\T{\mathcal{T}}

\newcommand\PP{{\mathbb P}}

\newcommand{\N}{{\mathbb N}}
\def\R{{\mathbb R}}
\def\Z{{\mathbb Z}}

\renewcommand{\real}{\mathbb R}
\newcommand\RR{{\mathcal{R}}}
\newcommand\RRs{{\mathcal{R}^\ast}}

\newcommand\Hoi[1]{H^1(#1)}
\newcommand\Hoo{H^1_0(\Om)}
\newcommand\Hoor{H^1_{0}(\Om)}
\newcommand\Lt{L^2(\Om)}

\newcommand\Hdv{{\mathbf H}(\dv,\Om)}

\newcommand\Hsa{H^1_*(\oma)}

\newcommand\W{{\mathcal V}}

\newcommand{\HH}{\mathcal{H}}

\iffract

\else

\fi

\newcommand\Ttab{\rule{0pt}{2.6ex}}       
\newcommand\Btab{\rule[-1.2ex]{0pt}{0pt}} 

\renewcommand{\vec}[1]{\mathsf{#1}}
\newcommand{\matr}[1]{\textsf{\textbf{#1}}}

\newcommand{\m}{m}
\newcommand{\M}{M}
\newcommand{\gam}{{\gamma^0}}
\newcommand{\gamh}{\gamma_h}
\newcommand{\Lh}[1]{{{\bm \Lambda}^{\!h}_{#1}}}
\newcommand{\La}[1]{{{\bm \Lambda}_{#1}}}
\def\HS{{\gS_2(\HH)}}
\def\Sone{\gS_1(\HH)}
\def\LH{\mathcal{L}(\HH)}
\def\gS{{\mathfrak S}}
\newcommand{\J}{{J}}
\newcommand{\LL}{{L}}
\newcommand{\I}{{I}}
\newcommand{\A}{{A}}
\newcommand{\Over}{{\matr{M}_{\Phie,\Phiho}}}
\newcommand{\sHH}{{}}

\newcolumntype{L}[1]{>{\raggedright\let\newline\\\arraybackslash\hspace{0pt}}m{#1}}
\newcolumntype{C}[1]{>{\centering\let\newline\\\arraybackslash\hspace{0pt}}m{#1}}
\newcolumntype{R}[1]{>{\raggedleft\let\newline\\\arraybackslash\hspace{0pt}}m{#1}}

\title{Guaranteed {\it a posteriori} bounds for eigenvalues and eigenvectors: multiplicities and clusters \thanks{Part of this work has been supported from French state funds managed by the CalSimLab LABEX and the ANR within the Investissements d'Avenir program (reference ANR-11-LABX-0037-01). The last author has also received funding from the European Research Council (ERC) under the European Union’s Horizon 2020 research and innovation program (Grant Agreement No. 647134 GATIPOR). Part of this work was supported by the French ``Investissements d'Avenir" program, project ISITE-BFC (contract ANR-15-IDEX-0003). 
YM and EC acknowledge funding from PICS- CNRS, PHC PROCOPE 2017 (Project No. 37855ZK), and  European Research Council (ERC) under the European Union's Horizon 2020 research and innovation programme (grant agreement No 810367).}}

\author[1,6]{Eric Canc\`es}
\author[2]{Genevi\`eve Dusson}
\author[3,4]{Yvon Maday}
\author[5]{Benjamin Stamm}
\author[6,1]{Martin Vohral\'ik}
\affil[1]{CERMICS, Ecole des Ponts ParisTech, 6 \& 8 Av. Pascal, 77455 Marne-la-Vall\'ee, France}
\affil[2]{Universit\'e Bourgogne Franche-Comt\'e, Laboratoire de Math\'ematiques de Besan\c{c}on, UMR CNRS 6623, 16 route de Gray,
25030 Besan\c{c}on, France}
\affil[3]{Sorbonne Universit\'e and Universit\'e de Paris, CNRS, Laboratoire Jacques-Louis Lions (LJLL), 75005 Paris, France}
\affil[4]{Institut Universitaire de France, 75005 Paris, France}
\affil[5]{Center for Computational Engineering Science, RWTH Aachen University, Aachen, Germany}
\affil[6]{Inria, 2 Rue Simone Iff, 75589 Paris, France}
\ifPdfTex \usepackage[breaklinks,bookmarks=false]{hyperref}
\else
\ifPDF \usepackage[dvips,breaklinks,bookmarks=false]{hyperref} \fi \fi

\ifPDF
\hypersetup{colorlinks, linkcolor=blue, citecolor=blue,
urlcolor=blue, plainpages=false, pdfwindowui=false,
pdfstartview={FitH}, pdftitle={Guaranteed {\it a posteriori} bounds for eigenvalues and eigenvectors: multiplicities and clusters}}
\fi

\begin{document}

\renewcommand{\thefootnote}{\arabic{footnote}}

\maketitle

\begin{abstract}
   \noindent This  paper presents {\it a posteriori} error estimates for conforming numerical approximations of eigenvalue clusters of second-order self-adjoint elliptic linear operators with compact resolvent. Given a cluster of eigenvalues, we estimate the error in the sum of the eigenvalues, as well as the error in the eigenvectors represented through the density matrix, \ie, the orthogonal projector on the associated eigenspace. This allows us to deal with degenerate (multiple) eigenvalues within the framework. All the bounds are valid under the only assumption that the cluster is separated from the surrounding smaller and larger eigenvalues; we show how this assumption can be numerically checked. Our bounds are guaranteed and converge with the same speed as the exact errors. They can be turned into fully computable bounds as soon as an estimate on the dual norm of the residual is available, which is presented in two particular cases: the Laplace eigenvalue problem discretized with conforming finite elements, and a Schr\"odinger operator with periodic boundary conditions of the form $-\Delta + V$ discretized with planewaves. For these two cases, numerical illustrations are provided on a set of test problems.

   \end{abstract}

\tableofcontents

\section{Introduction} \label{sec_intr}

Elliptic eigenvalue problems arise in many mathematical models used in science and engineering; often, precise approximations of eigenvalues and eigenvectors are crucial. To guarantee the quality of the approximations at stake, one needs to estimate the size of the errors for the computed quantities, namely the eigenvalues and eigenvectors. {\it A posteriori} error bounds aim at providing such estimates.

Already very good {\it a posteriori} estimates have been proposed for elliptic source problems, based for example on the theory of equilibrated fluxes for the Laplace source problem following Prager and Synge~\cite{Prag_Syng_47}, see~\cite{Lad_Leg_83, Dest_Met_expl_err_CFE_99, Brae_Pill_Sch_p_rob_09, Ern_Voh_p_rob_15} and the references therein.
Nonetheless, the error estimation for eigenvalue problems seems more complex in comparison. Following Kato~\cite{Kat_bounds_eigs_49},
Forsythe~\cite{Fors_eigs_55}, Weinberger~\cite{Wein_eigs_56}, and Bazley and
Fox~\cite{Baz_Fox_eigs_Sch_61},
recent works have been presented for the estimation of simple eigenvalues, possibly only the lowest one, see \eg~\cite{Luo_Lin_Xie_eigs_bounds_12,Hu_Huang_Lin_bounds_eigs_14, Hu_Huang_Shen_eigs_NC_14,Cars_Ged_LB_eigs_14, Liu_fram_eigs_15}, see also the references therein.
A thorough {\it a posteriori} analysis of errors in both simple eigenvectors and eigenvalues for the Laplace eigenvalue problem can be found in~\cite{Canc_Dus_Mad_Stam_Voh_eigs_conf_17} (conforming discretization methods) and~\cite{Canc_Dus_Mad_Stam_Voh_eigs_nonconf_18} (a unified framework including nonconforming discretization methods).

The above results however only hold for simple eigenvalues, whereas degenerate or near degenerate eigenvalues often appear in practice. This can dramatically deteriorate the estimates, especially when the latter depend on the gap between the estimated eigenvalue and the surrounding ones, as it is the case in~\cite{Canc_Dus_Mad_Stam_Voh_eigs_conf_17, Canc_Dus_Mad_Stam_Voh_eigs_nonconf_18}.
Only few results have been presented so far for the {\it a posteriori} estimation of multiple eigenvalues or clusters of eigenvalues.
In~\cite{Liu_fram_eigs_15,Wang_Cham_Lad_Zhong_eigvecs_16}, guaranteed error estimates are presented for the eigenvalue error, but not the eigenvector error.
The derivation of optimal eigenvalue convergence rates for adaptive finite element methods can be found in~\cite{Gallistl2015-yf,Dai2015-jq} for conforming finite elements, in~\cite{Gallistl2014-nq} and~\cite{Boffi2017-fj} for nonconforming and mixed finite elements, and in~\cite{Bonito2016-sv} for conforming and nonconforming finite elements of higher order.
{\it A posteriori} error estimation for clusters of eigenvalues have been presented in the case of the discontinuous Galerkin method in~\cite{Gia_Hall_a_post_eig_DG_12} and for Crouzeix--Raviart nonconforming finite elements in~\cite{Boffi2017-ny}. Also, {\it a posteriori} bounds have been established in~\cite{Grubisic2009-bg} and
\cite{Bank2013-zj} for $\mathbb{P}_1$ finite elements with triangular meshes, where the bounds are directly derived on the eigenspace, and are therefore independent of the choice of the eigenvectors, as is the case in this work. Finally, a recent contribution deriving upper bounds on eigenvectors associated with multiple eigenvalues is~\cite{Liu_Vejch_eigs_clusters_19}. Though the methodology also works on eigenspaces and the bounds are guaranteed, it does not seem to extend to a general cluster (the precision is limited by the difference of the largest and smallest eigenvalue in the cluster).

In this article, we extend the {\it a posteriori} error estimates presented in~\cite{Canc_Dus_Mad_Stam_Voh_eigs_conf_17} to clusters of eigenvalues, which  includes the possibility of degenerate eigenvalues. The estimators are derived for a generic second-order elliptic self-adjoint operator with compact resolvent denoted by $A$. More precisely,
let $(\la{i},\ue{i})_{i\ge 1}$ be the eigenvalues and associated eigenvectors of the operator $A$. We are interested in the cluster of eigenvalues $\la{\m},\ldots,\la{\M},$ with $\m,\M\in\N\backslash\{0\}, \m\le \M$. We first derive guaranteed bounds for the error in the sum of the eigenvalues. To derive these bounds, the only necessary assumption is that the cluster is separated from the surrounding lower and higher eigenvalues, as stated in Assumption~\ref{as:gap}, and a continuous--discrete gap condition summarized in Assumption~\ref{as:gap_cont_disc}. The problem is described in Section~\ref{sec_eig_pb}, and we consider a conforming discretization presented in Section~\ref{sec:conf_disc}.

In order to account for all the exact (respectively approximate) eigenvalues and eigenvectors of the cluster as a whole, the estimates rely on the use of  density matrices, which are the orthogonal projectors on the exact (resp. approximate) eigenspaces spanned by the eigenvectors of the cluster.
This allows to handle the nonuniqueness of the eigenvectors.
Indeed, under the gap Assumption~\ref{as:gap}, the exact density matrix $\gam$ is uniquely defined. The error estimates are therefore presented on the density matrix error in Hilbert--Schmidt and energy norms. All the definitions necessary to introduce the density matrix framework are presented in Section~\ref{sec:functional_setting}.

As we are aware that error estimates are usually provided in terms of eigenvectors instead of density matrix, we show in Section~\ref{sec:fct_density_matrix_equivalence} that for a well-chosen norm, the error on the density matrix is in fact equivalent to the error on the eigenvectors for a particular choice of approximate eigenvectors, which essentially guaranties that they are aligned with the reference exact eigenvectors. We then in Section~\ref{sec_res} introduce the residual, defined in this framework as an operator and not as a functional as in usual eigenvalue problems.

Generic error equivalences are presented in Section~\ref{sec_equiv}. More precisely, we provide estimates on density matrix errors and on the sum of eigenvalues error in terms of dual norms of the residual. The bounds are guaranteed, containing no unknown constant, but are not directly computable. Indeed, they depend on the dual norm of the residual, which is not always computable and can be difficult to estimate.

In Section~\ref{sec:apost_estimates}, we transform these equivalences into fully computable error bounds in two cases. We first in Section~\ref{sec:Laplace_apost} treat the case of the Laplace eigenvalue problem on an open Lipschitz polygon or polyhedron $\Omega \subset \R^d$ with Dirichlet boundary conditions discretized on simplicial meshes by conforming finite elements of degree $p$, based on~\cite{Canc_Dus_Mad_Stam_Voh_eigs_conf_17} for the estimate of the dual norm of the residual. This estimate relies on the construction of an equilibrated flux requiring to solve mixed finite element local residual problems. The error bound for the sum of the eigenvalues in the considered cluster is given in Theorem~\ref{thm_bound_eigs} and reads
\be \label{eq_est_vals_intr}
      0
      \le   \sum_{i=\m}^\M (\lah{i} - \la{i})
      \le \eta^2,
\ee
where $\lah{i}$ is the $i^{\rm th}$ approximate eigenvalue (counting multiplicities). Further, error bounds on the density matrix error are provided. In particular, Theorem~\ref{thm_bound_eigvecs} shows that
\be \label{eq_est_vecs_intr}
    \norm{|\Gr| (\gam - \gamh)}_\HS \leq \eta,
\ee
$\gamma_h$ being the approximate density matrix and $\norm{\cdot}_\HS$ the Hilbert--Schmidt norm associated with the $L^2(\Omega)$ Hilbert space. Moreover, Theorem~\ref{thm_bound_eigvecs} also shows that these bounds are efficient in the sense that
\be \label{eff_est_vecs_intr}
    \eta \leq C \norm{|\Gr| (\gam - \gamh)}_\HS,
\ee
where $C$ is a generic constant independent of the mesh size $h$ and the polynomial degree $p$.

We distinguish two cases. In Case I, no assumption other than the gap Assumptions~\ref{as:gap} and~\ref{as:gap_cont_disc} are needed, and we give sufficient conditions to check them in practice, \cf\ Remark~\ref{rem_pract}.
There, the bound $\eta^2$  only depends on the flux reconstruction and on a lower bound for the relative gaps between the cluster of approximate eigenvalues and the surrounding exact eigenvalues, that is lower bounds of the quantities
$\displaystyle\left( \frac{\lah{\m}}{\la{\m-1}} -1 \right)$ and
$\displaystyle\left( 1- \frac{\lah{\M}}{\la{\M+1}} \right)$.
In Case II, which holds under an additional elliptic regularity assumption on the corresponding source problem as described in~\eqref{eq_adj}, the pre-factor in $\eta$ can be brought to the optimal value of 1.

In Section~\ref{sec:PW_theory}, we then provide bounds for Schr\"odinger-type operators of the form $-\Delta+V$ on a cubic box with periodic boundary conditions, discretized with a planewave basis, in which case the dual norm of the residual is explicitly computable as the Laplace operator is diagonal in this basis. This allows to straightforwardly apply the bounds obtained in Section~\ref{sec_equiv} and derive error estimates both for the sum of the eigenvalues error and the error on the density matrix built from the eigenvectors in the form similar to~\eqref{eq_est_vals_intr}--\eqref{eff_est_vecs_intr}.
For clarity, the different assumptions used throughout this article are collected in Table~\ref{tab:assumptions}, and the main results and corresponding assumptions are listed in Table~\ref{tab:results_assumptions}.

\begin{table}
   \centering
   \begin{tabular}{|c|C{6cm}|C{3.5cm}|}
      \hline
      Assumption & Name & Implications \\
      \hline
      Assumption~\ref{as:gap} & Continuous gap conditions
      &  \\
      \hline
      Assumption~\ref{as:nonorth} & Non-orthogonality of the exact and approximate eigenspaces&  \\
      \hline
      Assumption~\ref{as:gap_cont_disc} & Continuous--discrete gap conditions & \\
      \hline
      Assumption~\ref{as:lowerbounds}
      & 
      Availability of accurate enough lower bounds for $\la{\m}$ and $\la{\M+1}$
      & 
      implies 
      Assumptions~\ref{as:gap} and~\ref{as:gap_cont_disc} 
      \\
      \hline
   \end{tabular}
   \caption{Assumptions and implications between assumptions}
   \label{tab:assumptions}
\end{table}

\begin{table}
   \begin{center}
      \begin{tabular}{|c|c|l|}
         \hline 
         \multicolumn{2}{|c|}{Assumptions} & Reference of the results \\
          \hline
         \multicolumn{2}{|c|}{\multirow{3}{*}{
            \ref{as:gap}}}
          & Lemma~\ref{lem:2.6} [Link between eigenvalue and eigenvector errors] \\
          \multicolumn{2}{|c|}{} & Theorem~\ref{th:eigenvalue_bounds} [Eigenvalue bounds] \\
          \multicolumn{2}{|c|}{} & Theorem~\ref{th:DM_bounds}, \eqref{eq:3.5.1} 
          [Upper bounds for the density matrix error]\\
         \cline{2-3}
         \multirow{2}{*}{\phantom{a}} & \multirow{2}{*}{
            and~\ref{as:nonorth}}
         & Lemma~\ref{lem:norm_eq2} [Link between density matrix and eigenvector errors] \\
         &  & Theorem~\ref{th:lower_bound_DM} [Lower bound for the density matrix error] \\
          \hline
          \hline
          \multicolumn{2}{|c|}{\multirow{5}{*}{
             \ref{as:gap} and~\ref{as:gap_cont_disc}}}
          & Lemma~\ref{lem:DM_bounds} [Bounds on the density matrix error] \\
          \multicolumn{2}{|c|}{} &    Theorem~\ref{thm_bound_eigs}, Case II [Guaranteed bounds for the sum of eigenvalues]\\
          \multicolumn{2}{|c|}{} &  Theorem~\ref{thm_bound_eigvecs}, Case II [Guaranteed bound for the density matrix error] \\ 
          \multicolumn{2}{|c|}{} & Theorem~\ref{thm_bound_eigs_PW} [Guaranteed bounds for the sum of eigenvalues]  \\
          \multicolumn{2}{|c|}{} & Theorem~\ref{thm_bound_eigvecs_PW} [Guaranteed bound for the density matrix errors]\\
          \cline{2-3}
         \multirow{2}{*}{\phantom{a}} & \multirow{3}{*}{
            and~\ref{as:nonorth}}
         &    Theorem~\ref{th:DM_bounds}, \eqref{eq:3.5.1bis} [Upper bounds for the density matrix error]\\
          &  &  Theorem~\ref{thm_bound_eigs}, Case I [Guaranteed bounds for the sum of eigenvalues]\\
          &  &  Theorem~\ref{thm_bound_eigvecs}, Case I [Guaranteed bound for the density matrix error] \\ 
          \hline
      \end{tabular}
   \end{center}

   \caption{Summary of the main results with the employed assumptions.}
   \label{tab:results_assumptions}
\end{table}

We present in Section~\ref{sec_num} numerical results for (i) the Laplace operator discretized with conforming finite elements in a 2D setting, and (ii) a Schr\"odinger operator $-\Delta+V$ on a cubic box with periodic boundary conditions, discretized in a planewave basis, in a 1D and 2D setting.
The error bounds fulfill the expectations, and in particular, the necessary assumptions already hold for coarse bases.
Finally, some conclusions are drawn in Section~\ref{sec:concl}, and Appendix~\ref{app:proof_lemma} details the proof of a technical result.

\section{Setting} \label{sec_setting}

We introduce here the considered eigenvalue problem, its generic conforming discretization, and the functional analysis setting that we adopt.

\subsection{The eigenvalue problem}\label{sec_eig_pb}

Let $\HH$ be a real {\em separable Hilbert space} endowed with an inner product
denoted by $(\cdot, \cdot)_\sHH$, and a corresponding norm denoted by $\|\cdot\|_\sHH$.
We consider a {\em self-adjoint operator} $A$ on $\HH$ with domain $D(A)$, {\em bounded-below}, and
with {\em compact resolvent}.
For such an operator, there exists a non-decreasing sequence of real numbers
$(\la{k})_{k \geq 1}$ such that $\la{k} \rightarrow +\infty$ and an
orthonormal basis $(\ue{k})_{k \geq 1}$ of $\HH$ consisting of vectors
of $D(A)$ such that
\begin{equation} \label{eq:eig_pb}
  \forall k \geq 1, \quad \A \ue{k} = \la{k} \ue{k}.
\end{equation}
In the following, we will often employ the Parseval identity, which states that for any $v \in
\HH$,
\begin{equation}
  \label{eq_Pars}
     \norm{v}_\sHH^2 = \sum_{k \geq 1} |(v, \ue{k})_\sHH|^2.
\end{equation}
Up to shifting the operator $\A$ by a constant $c\in\mathbb{R}_+$, we can assume without loss of generality that $\A$ is a
positive definite operator, in which case $(\la{k})_{k \geq 1}$ is a
sequence of positive numbers.
This enables to define the operators $\A^{s}$, $s \in \R$,
by their domains
\bse\label{eq_As} \begin{equation} \label{eq_As_dom_norm}
    D(\A^{s}) \eq \left\lbrace v\in\mathcal{H}; \quad
    \norm{\A^{s} v}_\sHH^2 \eq \sum_{k \geq 1} \la{k}^{2s} |(v,\ue{k})_\sHH|^2 < +\infty  \right\rbrace
\end{equation}
and expressions
\begin{equation} \label{eq_As_expr}
	\A^{s}: v\in D(\A^{s}) \mapsto \sum_{k \geq 1} \la{k}^{s} (v, \ue{k})_\sHH \ue{k} \in \HH.
\end{equation}\ese
In particular, the norm $\|A^{1/2}\bullet\|$ is referred to as the {\it energy norm}.
We also remark that $\A^{0}$ is the identity operator on $\HH$, that $\A^s \A^t = \A^{s+t}$ for all $s,t \in \R$, and that $D(A^s)=\HH$ for all $s \le 0$.
Also, $D(\A^s) \subset D(\A^t)$ for $s \geq t$, so that for all $k \geq 1$, $\ue{k}$ from~\eqref{eq:eig_pb} belongs to $D(\A^s)$ for all $s\in\R$  and
\begin{equation}
  \label{eq:eig_pb_weak}
    (\A^{1/2}\ue{k}, \A^{1/2} v)_\sHH = \la{k} (\ue{k}, v)_\sHH \qquad \forall v \in D(\A^{1/2}), \, \forall k \geq 1.
\end{equation}
This in particular implies, as $\norm{\ue{k}}=1$,
\begin{equation}
  \label{eq:eig_pb_norm}
  \| A^{1/2} \ue{k}\|_\sHH^2 = \la{k} \qquad \forall k \geq 1.
\end{equation}

\subsection{Functional analysis setting}
\label{sec:functional_setting}

In this article, we focus on the error estimation of clusters of eigenvalues and their corresponding eigenvectors. More precisely, given  $\m,M\in\N \setminus \{0\}$, $m\le M$, we consider the {\em eigenvalue cluster} composed of the $\J \eq M-\m+1$ eigenvalues $(\la{m},\ldots,\la{M})$ from~\eqref{eq:eig_pb}, counted with their multiplicities. For our {\it a posteriori} analysis, we will need to assume that the considered cluster is separated from the rest of the spectrum:
\begin{assumption}[Continuous gap conditions]
   \label{as:gap}
   There holds
  $\la{m-1}<\la{m}$ if $m>1$ and $\la{\M}<\la{\M+1}$.
\end{assumption}
We show in Remark~\ref{rem_pract} below how this condition can be verified practically.

We denote an orthonormal set of corresponding {\em eigenvectors} by
\begin{equation}
   \label{eq:ex_eigenvectors}
    \Phie \eq (\ue{m},\ldots,\ue{M}).
\end{equation}
Note that estimating the error between $\Phie$ and given approximate eigenvectors $\Phih = (\uh{m},\ldots,\uh{M})$ cannot in general be done without further assumptions on the choice of the eigenvectors. Indeed, in particular for multiple eigenvalues $\la{m} = \ldots = \la{M}$, for any matrix $\matr{U} \in O(\J)= \left\{  \matr{U}\in \R^{\J\times \J}\middle| \, \matr{U}^T \matr{U} = \matr{1}_\J     \right\}$, the group of orthogonal matrices of order~$\J$, $\Phie \matr{U}$ also form an orthonormal set of eigenvectors associated with $(\la{m},\ldots,\la{M})$.

To get rid of the above problematic nonuniqueness, we will measure and estimate
the errors not on the eigenvectors directly, but in the spaces spanned by these eigenvectors, which are uniquely determined, even in the case of degenerate (multiple) eigenvalues, as long as the gap Assumption~\ref{as:gap} is satisfied.
In this case, the orthogonal projector for the inner product~$(\cdot, \cdot)_\sHH$ onto $\text{Span} \{\ue{m},\ldots,\ue{M} \}$, denoted by $\gam$ and called {\em density matrix}, is also unique.
It is the rank-$\J$ operator on $\HH$ defined by
\begin{equation}
  \label{def:gamma0}
 \forall v \in \HH, \quad \gam v \eq \sum_{i = \m}^\M (v, \ue{i})_\sHH \ue{i}.
\end{equation}
The exact and approximate eigenspaces can therefore be compared through their density matrices.
In fact, we will introduce below a norm to measure the error on the density matrices which is equivalent to the energy norm of the error on the eigenvectors for the particular choice of eigenvectors for which the approximate eigenvectors are as much aligned as possible with the corresponding exact eigenvectors. Note that in particular $\gam v = v$ if $v \in \text{Span} \{\ue{m},\ldots,\ue{M}\}$ and $\gam v = 0$ if $v$ is in the orthogonal complement of $\text{Span} \{\ue{m},\ldots,\ue{M}\}$, which will often be used below.

The functional setting of trace-class and Hilbert--Schmidt operators used to define this norm is presented in detail in~\cite[Chapter VI]{Reed1978-li} and can also be found in~\cite{Cances_undated-fd}. We only briefly recall here the properties used in this article.
We denote by $\LH$ the space of bounded linear operators on $\HH$.
If $B \in \LH$ is a {\em positive} operator, (\ie, $(v,B v) \ge 0$ for any $v \in \HH$), then the value of the sum
\begin{equation} \label{eq_Tr}
\Tr(B)\eq\sum_{k \geq 1} (e_{k}, B e_{k} )_\sHH \in \R_+ \cup \left\{+\infty\right\}
\end{equation}
is independent of the choice of the orthonormal basis $(e_{k})_{k \geq 1}$ of $\HH$. Let now $B \in \LH$ be arbitrary and let $B^\dag$ denote the adjoint of $B$, \ie, $B^\dag \in \LH$ such that $(B^\dag v, w)_\sHH = (v, B w)_\sHH$ for all $v,w \in \HH$. We define $|B|:=\sqrt{B^\dag B}$,
\begin{equation} \label{eq_Tr_1}
    \|B\|_{\Sone} \eq \Tr (|B|) = \sum_{k \geq 1} (e_{k}, |B|e_{k} )_\sHH,
\end{equation}
and
\begin{equation}
    \|B\|_{\HS}  \eq \Tr(B^\dag B)^{1/2} = \left( \sum_{k \geq 1} \| B e_{k}\|_{\sHH}^2\right)^{1/2}. \label{eq_Tr_2_norm}
\end{equation}
Then the Banach space $\Sone$ of {\em trace-class} operators on $\HH$ is the space of all $B \in \LH$ with $\|B\|_{\Sone} < +\infty$. In particular, if $B  \in \LH$ is positive and {\em self-adjoint},
then $B \in \Sone$ if and only if $\Tr(B)<+\infty$.
The Hilbert space $\HS$ of {\em Hilbert--Schmidt} operators on $\HH$ is the space of all $B \in \LH$ with $\|B\|_{\HS} <+ \infty$ endowed with the scalar product
\begin{equation}
    (B,C)_{\HS}  \eq \Tr(B^\dag C) \label{eq_Tr_2_scp}.
\end{equation}
Recall that $\Sone \subset \HS \subset {\mathfrak S}_\infty(\HH) \subset \LH$, where ${\mathfrak S}_\infty(\HH)$ is the vector space of compact operators on $\HH$, and that for any compact self-adjoint operator $B$ on $\HH$, we have
\[
\|B\|_{\LH} = \max_{i \ge 1} |\mu_i|, \quad  \|B\|_\HS = \left( \sum_{i \ge 1} |\mu_i|^2 \right)^{1/2}, \quad \|B\|_{\Sone}= \sum_{i \ge 1} |\mu_i|,
\]
where the $\mu_i$'s are the eigenvalues of $B$ counting multiplicities, so that
\[
\|B\|_{\LH} \le  \|B\|_\HS \le \|B\|_{\Sone} \quad  \mbox{and} \quad \|B\|_\HS  \le \|B\|_{\LH}^{1/2} \|B\|_{\Sone}^{1/2}.
\]
For all $B \in \HS$ and $C \in \HS$, it follows that $BC \in \Sone$, $CB \in \Sone$, and
\begin{equation}
\label{eq:point5}
	\Tr(BC)=\Tr(CB) \le \|B\|_{\HS} \|C\|_{\HS}.
\end{equation}

Note that $\gam \in \LH$ and $\gam$ is positive, self-adjoint, and an (orthogonal) projector since $(\gam)^2 = \gam$. Furthermore, $\gam \in \Sone$ and its trace is equal to $\J = M-\m+1$, the dimension of $\text{Span}\{\ue{m},\ldots,\ue{M}\}$. Indeed,
\begin{equation} \label{eq:Trgam}
  \Tr(\gam) = \sum_{k\ge 1}
   (\ue{k}, \gam \ue{k})_\sHH
  = \sum_{k= \m}^\M |(\ue{k}, \ue{k})_\sHH|^2
  = \J.
\end{equation}
Moreover, we have
\begin{equation} \label{eq:Trgam2}
    \|\gam\|_{\HS}^2 = \Tr\Big((\gam)^\dag \gam\Big) = \Tr((\gam)^2) = \Tr(\gam) = \J.
\end{equation}
For the following, we set, for all $s \in \R$,
\begin{equation} \label{eq:vecs}
    \forall 	\vec{\Psi} = (\psi_\m, \ldots, \psi_\M) \in {[D(\A^s)]}^\J, \, \| \A^{s} \vec{\Psi} \|_{\sHH} \eq \left( 		\sum_{i=\m}^\M \|\A^{s} \psi_i\|_{\sHH}^2	\right)^{1/2}.
\end{equation}
A particular consequence of the definitions presented above is:

\begin{lemma}[Difference of orthogonal projectors] \label{lem:expan_gamerr}
   Let $\gamma_\I$ and $\gamma_\LL$ be two finite-rank orthogonal projectors of the same rank. There holds
   \[
      \|\gamma_\LL-\gamma_\I\|_{\HS}^2 =  2 \Tr(\gamma_\LL(1-\gamma_\I)). \]
\end{lemma}

\begin{proof}
   Note that the traces of $\gamma_\I$ and $\gamma_\LL$ are equal (to their rank). Therefore,
   \begin{align}
      \|\gamma_\LL-\gamma_\I\|_{\HS}^2 ={} & \Tr \left((\gamma_\LL-\gamma_\I)(\gamma_\LL-\gamma_\I)\right) \nonumber \\
       ={} &  \Tr(\gamma_\LL^2)+\Tr(\gamma_\I^2)-2\Tr(\gamma_\LL\gamma_\I) \nonumber \\
       ={} & \Tr(\gamma_\LL)+\Tr(\gamma_\I)-2\Tr(\gamma_\LL\gamma_\I) \nonumber \\
       ={} & 2 \left(\Tr(\gamma_\LL) - \Tr(\gamma_\LL\gamma_\I)\right) \nonumber \\
       ={} & 2 \Tr(\gamma_\LL(1-\gamma_\I)). \nonumber
   \end{align}
\end{proof}

\subsection{Conforming discretizations}
\label{sec:conf_disc}

We consider conforming approximations of problem~\eqref{eq:eig_pb} in a space $V_h\subset D(\A^{1/2})$.
The approximate $k$-th eigenpair $(\uh{k}, \lah{k}) \in
V_h \times \R_+$ is such that $(\uh{k},\uh{j}) = \delta_{kj}$, $1 \leq k,j
\leq \dim V_h$, and satisfies
\begin{equation} \label{eq:discretization}
	(\A^{1/2} \ukh, \A^{1/2} v_h) = \lkh (\ukh, v_h) \qquad \forall v_h \in V_h.
\end{equation}
We number the approximate eigenvalues in increasing order, that is
$0 < \lambda_{1h} \leq \lambda_{2h} \leq \ldots \leq \lambda_{\dim V_h h}$, while counting multiplicities.
Again, an immediate consequence of~\eqref{eq:discretization} is
\begin{equation}
  \label{eq:eig_pb_norm_disc}
  \| A^{1/2} \ukh \|_\sHH^2 = \lkh, \qquad \forall 1 \leq k \leq \dim V_h,
\end{equation}
and, since the approximation is conforming, there holds
\begin{equation}
    \label{eq:conf_approx}
    \la{k} \le \lah{k} \qquad \forall 1 \leq k
    \leq \dim V_h.
\end{equation}

The {\em approximate eigenvalues} we are interested in are denoted by $(\lah{\m},\ldots,\lah{\M})$, where of course we suppose $\dim V_h \geq \M$, and a
corresponding set of orthonormal {\em approximate eigenvectors} by  \be
    \Phih \eq (\uh{m},\ldots,\uh{M}).
\ee
The {\em approximate density matrix} is then defined by
\begin{equation}
  \label{def:gamh}
\forall v \in \HH, \quad  \gamh v \eq \sum_{i = \m}^\M (v,\uh{i})_\sHH \uh{i}.
\end{equation}
\begin{remark}[Discrete gap condition] If the discrete gap condition
   \label{as:gap_disc}
  $\lah{(m-1)}<\lah{m}$ if $m>1$ and $\lah{\M}<\lah{(\M+1)}$ is fulfilled, then the approximate density matrix $\gamma_h$ is uniquely defined. Note that this uniqueness is not needed at this point, but is a consequence of
  Assumption~\ref{as:gap_cont_disc} that is required 
  starting from the upcoming Theorem~\ref{th:DM_bounds} onward.
\end{remark}
Like the density matrix $\gam$, the approximate density matrix $\gamh$ is an orthogonal projector, hence $\gamh^2 = \gamh$, and there also holds $\Tr(\gamh) = \Tr(\gamh^2) = \J$. We will measure the error between the exact and approximate density matrices using the quantities
\[
\|\gam-\gamh\|_{\HS} \quad \mbox{and} \quad \|A^{1/2}(\gam-\gamh)\|_{\HS}.
\]
The second quantity is indeed justified as we have:
\begin{lemma}[Operators $A^{1/2}\gam$ and $A^{1/2}\gamh$] \label{lem_LH} Let the density matrices $\gam$ and $\gamh$ be respectively defined by~\eqref{def:gamma0} and~\eqref{def:gamh}. Then there holds $A^{1/2}\gam$, $A^{1/2}\gamh \in \HS$ and
\begin{equation}
    \|A^{1/2}\gam\|_{\HS}^2  = \sum_{i=\m}^\M \li, \quad \|A^{1/2}\gamh\|_{\HS}^2  = \sum_{i=\m}^\M \lih. \label{eq_A12_S}
\end{equation}
\end{lemma}

\begin{proof} Taking in~\eqref{eq_Tr_2_norm} $e_k = \ue{k}$, since $\gam \ue{k} = \ue{k}$ if $\m \leq k \leq M$ and $0$ otherwise, and employing~\eqref{eq:eig_pb_norm},
\[
    \|A^{1/2} \gam\|_{\HS}^2 =
      \sum_{k \geq 1} \| A^{1/2} \gam \ue{k}\|_\sHH^2 =
      \sum_{k=\m}^\M \| A^{1/2} \ue{k}\|_\sHH^2
      = \sum_{k=\m}^\M \lk.
\]
The result for $\|A^{1/2}\gamh\|_{\HS}$ is obtained similarly upon completing $\ukh$ to an orthonormal basis of $\HH$ used in~\eqref{eq_Tr_2_norm} and employing~\eqref{eq:eig_pb_norm_disc}.
\end{proof}

The following equalities will be useful in the upcoming analysis:

\begin{lemma}[Orthogonal projector and Hilbert--Schmidt norm] \label{lem_orth_proj}
   Let $\gamma_\J$ be the orthogonal projector of rank $\J = \M-\m+1$ onto $\Span \{ \varphi^\J_{\m},  \ldots,  \varphi^\J_{\M}\}$, where $\varphi^\J_{\m},  \ldots,  \varphi^\J_{\M} \in \HH$ are orthonormal. Completing $\{ \varphi^\J_{i} \}_{i=\m,\ldots,\M}$ to an orthonormal basis of $\HH$ denoted by $\{ \varphi^\J_{i} \}_{i\ge 1}$, there holds
   \begin{equation}
      \label{eq:gamma_v}
     \forall v\in\HH, \quad  \|\gamma_\J v\|_\sHH^2 = \sum_{i \geq 1} \left|\left(\gamma_\J v, \varphi^\J_{i}\right)_\sHH\right|^2 = \sum_{i=\m}^\M \left| \left(v,\varphi^\J_{i}\right)_\sHH\right|^2.
   \end{equation}
   If in addition $\varphi^\J_{\m},  \ldots,  \varphi^\J_{\M} \in D(A^{1/2})$, then there holds
   \begin{align}
      \|A^{1/2} \gamma_\J\|_{\HS}^2 & =
      \sum_{i \geq 1} \| A^{1/2} \gamma_\J \varphi^\J_i\|_\sHH^2 =
      \sum_{i=\m}^\M \| A^{1/2} \varphi^\J_i\|_\sHH^2 =
      \sum_{k\ge 1} \sum_{i=\m}^\M \left| \left(A^{1/2} \varphi^\J_{i},\ue{k}\right) \right|^2 
      \label{eq:gammaK_A}
      \\ & 
      =\sum_{k\ge 1} \sum_{i=\m}^\M \la{k} \left| \left(\varphi^\J_{i},\ue{k}\right) \right|^2.
      \nonumber
   \end{align}
\end{lemma}

\begin{proof} Result~\eqref{eq:gamma_v} follows as in~\eqref{eq_Pars} and since $\gamma_\J$ is a projector. For the other claim, let us first note that $A^{1/2} \gamma_\J \in \LH$, since $\mbox{Ran}(\gamma_J) \subset D(A^{1/2})$. Then, \eqref{eq:gammaK_A} is a consequence of~\eqref{eq_Tr_2_norm}, \eqref{eq_Pars}, and~\eqref{eq_As_dom_norm}.
\end{proof}

For some of the results presented in the following, we will need to assume that the approximate eigenvectors are not orthogonal to the exact ones:  \begin{assumption}[Non-orthogonality of the exact and approximate eigenspaces]
   \label{as:nonorth} There holds
   \[
   \forall v \in \Span\{\ue{\m},\ldots,\ue{\M} \}\backslash \{0\}, \quad
   \| \gamh v\|_\sHH \neq 0.
   \]
\end{assumption}
This assumption guarantees that every exact eigenvector is not orthogonal to the whole space spanned by the approximate eigenvectors. Note that this assumption, which in practice cannot be easily checked, is not needed for the first upper bound~\eqref{eq:3.5.1} below, which is used in Section~\ref{sec:Laplace_apost} for finite element discretizations in Case~II and in the planewave discretization in Section~\ref{sec:PW_theory}. However, in Case~I in the finite element discretization, we prefer to use the improved bound based on~\eqref{eq:3.5.1bis}, which requires this assumption. Assumption~\ref{as:nonorth} is also needed to show an equivalence between eigenvectors and density matrix errors (Lemma~\ref{lem:norm_eq2}), as well as to derive a lower bound for the density matrix error (Theorem~\ref{th:lower_bound_DM}).

\section{Density matrix error and residuals}

We develop in this section the links between the eigenvector errors and the density matrix errors. We will also define the residual and its dual norm, both for single and for cluster eigenpairs.

\subsection{Eigenvector error and density matrix error equivalence}
\label{sec:fct_density_matrix_equivalence}

Since there is a choice in the (exact and approximate) eigenvectors  of $\A$, in particular for multiple eigenvalues, the approximate
eigenvectors $\Phih$ might be far from the exact ones $\Phie$ individually, which is measured in the energy norm $\| \A^{1/2} (\Phie - \Phih) \|_{\sHH}$, while the density matrices and the eigenvalues are very close, or even equal. Traditionally, though, the error estimates are presented for the eigenvectors. Therefore, we first show that, given the exact eigenvectors $\Phie$, there exists a choice of approximate eigenvectors $\Phiho = (\uho{\m},\ldots,\uho{\M})$ constructed from $\Phih$ for which the error in the energy norm $\| \A^{1/2} (\Phie - \Phiho) \|_{\sHH}$ is equivalent to $\| \A^{1/2} (\gam - \gamh) \|_{\HS}$.
This is valid under the sole assumption that the two eigenspaces are not orthogonal with respect to the $\HH$ scalar product, as presented in Assumption~\ref{as:nonorth}.
We also show that the density matrix error $\| \A^{1/2} (\gam - \gamh) \|_{\HS}$ can be easily expressed in terms of the eigenvectors.

Let us define the following unitary-transformed approximate eigenvectors by
\begin{equation}
   \label{eq:unitary_min}
    \Phiho \eq (\uho{\m},\ldots, \uho{\M} )\eq \text{argmin}_{\matr{U}\in O(\J)} \|\matr{U}\Phih - \Phie\|_{\sHH},
\end{equation}
where we recall that $O(\J)$ denotes the group of orthogonal matrices of order $\J$. From~\cite[Lemma 4.3]{Cances2012-bo}, the minimization problem~\eqref{eq:unitary_min} has a unique solution and therefore $\Phiho$ is well defined as soon as Assumption~\ref{as:nonorth} is satisfied. Note that from this definition and the fact that the approximate eigenvectors are orthonormal, the rotated approximate eigenvectors are also orthonormal. Also, the approximate density matrix $\gamh$ given by~\eqref{def:gamh} can be equivalently written in terms of the rotated eigenvectors $\Phiho$ as
\begin{equation}
  \label{eq:rotated_DM}
  \forall v \in \HH, \quad \gamh v =   \sum_{i = \m}^\M (v, \uho{i})_\sHH \uho{i}.
\end{equation}

To relate $\| \A^{1/2} (\gam - \gamh) \|_{\HS}$ to the energy norm of the eigenvector errors $\norm{\A^{1/2}(\Phie-\Phiho)}_{\sHH}$, we first need a preliminary lemma, expressing these two quantities in terms of the exact and approximate eigenvalues and the eigenvectors of the operator $A$.

\begin{lemma}[Link between eigenvalue and eigenvector errors]
   \label{lem:2.6} 
   Let Assumption~\ref{as:gap} hold and 
   let the density matrices $\gam$ and $\gamh$ be respectively defined by~\eqref{def:gamma0} and~\eqref{def:gamh} and the eigenvectors $\Phie$ and $\Phiho$ by~\eqref{eq:ex_eigenvectors} and~\eqref{eq:unitary_min}. Then there holds
      \begin{equation}
      \|A^{1/2}(\gam - \gamh)\|_{\HS}^2 = \sum_{i=\m}^\M
      (\lah{i}-\la{i}) + 2\sum_{i=\m}^\M \la{i} \|(1-\gamh)\ue{i}\|_\sHH^2, \label{eq:Agamhgam}
   \end{equation}
and
   \begin{equation}
         \norm{\A^{1/2}(\Phie-\Phiho)}_{\sHH}^2 = \sum_{i=\m}^\M (\lah{i} - \la{i}) + \sum_{i=\m}^\M \la{i} \norm{\ue{i} - \uho{i}}_\sHH^2. \label{eq:A12Phi}
   \end{equation}

\end{lemma}

\begin{remark}
   Note that Assumption~\ref{as:gap} is needed for $\gam$ and $\Phie$ to be well-defined, but is not needed stricto sensu in the proof of Lemma~\ref{lem:2.6}, which is still valid in the case where there is no discrete or continuous gap between the eigenvalue cluster and the rest of the spectrum. 
   More precisely, this Lemma would hold for any orthogonal projectors $\gam,\gamh$ of rank $M-m+1$ and associated orthogonal vectors $\Phie$ and $\Phiho$  such that ${\rm Ran}(\gam) = {\rm Span}(\Phie)$, ${\rm Ran}(\gamh) = {\rm Span}(\Phih)$, and 
   satisfying for $i = m, \ldots, M$, $(A^{1/2} \ue{i}, A^{1/2} v) = \la{i} (\ue{i},v)$ for all $v\in D(A^{1/2})$, and $(A^{1/2} \uh{i}, A^{1/2} v_h) = \lah{i} (\uh{i},v_h)$ for all $v_h\in V_h$.
   For simplicity of presentation, Lemma~\ref{lem:2.6} is kept without these generalizations.
\end{remark}

\begin{proof}
   To show~\eqref{eq:Agamhgam}, we first expand $\norm{\A^{1/2}(\gam-\gamh)}_{\HS}^2$ and then use~\eqref{eq_A12_S} to obtain
   \begin{align}
      \norm{\A^{1/2}(\gam-\gamh)}_{\HS}^2 = {} & \norm{\A^{1/2}\gam}_{\HS}^2 -2 (\A^{1/2}\gam,\A^{1/2}\gamh)_{\HS} + \norm{\A^{1/2}\gamh}_{\HS}^2 \nonumber \\
      = {} & \sum_{i=\m}^\M \li - 2 \sum_{i \geq 1} (\A^{1/2}\gam \ue{i},\A^{1/2}\gamh \ue{i})  + \sum_{i=\m}^\M \lih \label{eq:xxx} \\
      = {} & \sum_{i=\m}^\M \li - 2 \sum_{i=\m}^\M (\A^{1/2} \ue{i},\A^{1/2}\gamh \ue{i})  + \sum_{i=\m}^\M \lih,
      \nonumber
   \end{align}
   where we have also used that $\gam \ue{i} = \ue{i}$ if $\m \leq i \leq M$ and $0$ otherwise.
   Employing~\eqref{eq:eig_pb_weak} leads to
   \begin{align} 
      \norm{\A^{1/2}(\gam-\gamh)}_{\HS}^2  &= \sum_{i=\m}^\M (\la{i} + \lah{i})
      - 2 \sum_{i=\m}^\M \la{i} \left(  \ue{i}, \gamh \ue{i}   \right)_\sHH  \nonumber
      \\& =  \sum_{i=\m}^\M (\lah{i} - \la{i})
      + 2 \sum_{i=\m}^\M \la{i} \left(
      1-
      \left(  \ue{i}, \gamh \ue{i}   \right)_\sHH
      \right). \label{eq_A12ggh_dev}
   \end{align}
   Writing $1 = (\ue{i},\ue{i})$ and observing that $(1-\gamh) = (1-\gamh)^2$, we obtain
   \begin{align}
      \norm{\A^{1/2}(\gam-\gamh)}_{\HS}^2  = & \sum_{i=\m}^\M (\lah{i} - \la{i})
      + 2 \sum_{i=\m}^\M \la{i}
      \left(  \ue{i}, (1-\gamh)^2 \ue{i}   \right)_\sHH, \nonumber
   \end{align}
   which leads to~\eqref{eq:Agamhgam}, using that $(1-\gamh)$ is self-adjoint.

   To show~\eqref{eq:A12Phi}, we complete $\Phiho$ to an orthonormal basis $(\varphi_{ih}^0)_{i \ge 1}$ of $\HH$, employ this basis in~\eqref{eq_Tr_2_norm}, and use~\eqref{eq_A12_S},
   \[
   \sum_{i=\m}^\M \| A^{1/2} \uho{i}\|_\sHH^2 = \sum_{i \geq 1} \| A^{1/2} \gamh \uho{i}\|_\sHH^2 = \|A^{1/2}\gamh\|_{\HS}^2 = \sum_{i=\m}^\M \lah{i}.
   \]
   We next use definition~\eqref{eq:vecs} together with~\eqref{eq:eig_pb_norm} and~\eqref{eq:eig_pb_weak} to see that
\begin{align*}
   \norm{\A^{1/2}(\Phie-\Phiho)}_{\sHH}^2 & = \sum_{i=\m}^\M \left[
      \norm{A^{1/2} \ue{i}}^2 - 2 (A^{1/2} \ue{i}, A^{1/2} \uho{i})  + \norm{A^{1/2} \uho{i}}^2 \right] \\ &
      =  \sum_{i=\m}^\M \big[\lah{i} + \la{i} - 2 \la{i}  (\ue{i},\uho{i})_\sHH\big].   
\end{align*}
Using that for all $i=\m,\ldots,\M$, $\uho{i}$ as well as $\ue{i}$ are of norm 1 leads to
\begin{align*}
   \norm{\A^{1/2}(\Phie-\Phiho)}_{\sHH}^2
  & =  \sum_{i=\m}^\M (\lah{i} - \la{i}) + 2 \sum_{i=\m}^\M \la{i} \left(1 -  (\ue{i},\uho{i})_\sHH
   \right) \\ &
   =  \sum_{i=\m}^\M (\lah{i} - \la{i}) + \sum_{i=\m}^\M \la{i} \norm{\ue{i} - \uho{i}}_\sHH^2,
\end{align*}
which concludes the proof.
\end{proof}

The following lemma relates the errors on the density matrix to the errors on the eigenvectors.

\begin{lemma}[Link between density matrix and eigenvector errors]
  \label{lem:norm_eq2} 
Let  the assumptions of Lemma~\ref{lem:2.6} hold, 
  together with Assumption~\ref{as:nonorth}. 
  Then
\begin{equation}
\label{eq:eqL2}
   \frac{1}{\sqrt{2}} \norm{\gam-\gamh}_{\HS} \le
		 \norm{\Phie-\Phiho}_{\sHH} \le \norm{\gam-\gamh}_{\HS}.
\end{equation}
Moreover,
\begin{multline}
\label{eq:eqH1}
   \frac{1}{\sqrt{2}} \norm{\A^{1/2}(\gam-\gamh)}_{\HS}
    \le \norm{\A^{1/2}(\Phie-\Phiho)}_{\sHH} 
   \\ 
   	\le
      \left( 1+ \frac{\la{\M}}{4\la{\m}}\|\gam - \gamh\|_{\HS}^2 \right)^{1/2}
   \norm{\A^{1/2}(\gam-\gamh)}_{\HS},
\end{multline}
and in particular
\begin{equation}
   \label{eq:eqH1_bonus}
   \norm{\A^{1/2}(\Phie-\Phiho)}_{\sHH}
   	\le
      \left( 1+ \frac{\J \la{\M}}{\la{\m}} \right)^{1/2}
   \norm{\A^{1/2}(\gam-\gamh)}_{\HS}.
\end{equation}
\end{lemma}

A proof for~\eqref{eq:eqL2} in the case $\m = 1$ can be found in Lemma 2.3 of~\cite{Cances_undated-fd}, whereas~\eqref{eq:eqH1} is proved in~\cite[Lemma 3.1]{GD} in a similar setting. For the sake of completeness, we present the proof of~\eqref{eq:eqH1} in Appendix~\ref{app:proof_lemma} in our specific setting. Using the (very) crude bound $\norm{\gam - \gamh}_{\HS}^2 \le 4 \J$, \cf~\eqref{eq:Trgam2}, we immediately deduce~\eqref{eq:eqH1_bonus} from~\eqref{eq:eqH1}.
Using Lemma~\ref{lem:norm_eq2}, it is possible to easily translate bounds expressed in terms of density matrices on bounds on the eigenvectors, as long as these eigenvectors are rotated correctly. In the rest of this paper, we will therefore focus on the estimation of density-matrix-based quantities only.

In terms of implementation, the natural outputs of an eigenvalue solver are often eigenvectors and not density matrices (note, however, that some algorithms directly compute density matrices using Cauchy's formula $\gam=\frac{1}{2i\pi}\oint_{\mathcal{C}}(z-A)^{-1} \, dz$, where ${\mathcal{C}}$ is a contour in the complex plane enclosing the eigenvalues $(\la{m},\ldots,\la{M})$). The practical computation of $\| A^{1/2}(\gamh-\gam)\|_{\HS}$ can easily be done in terms of the eigenvectors since, using~\eqref{eq:xxx} and~\eqref{def:gamh},
\[
   \norm{A^{1/2}(\gam-\gamh)}_{\HS}^2
   =
   \sum_{i=\m}^\M \left[ \| A^{1/2} \ue{i}\|_\sHH^2 - 2 \sum_{j=\m}^\M (\A^{1/2} \ue{i},\A^{1/2} \ujh)(\ujh,\ue{i}) + \| A^{1/2} \uih\|_\sHH^2 \right].
\]

To conclude this section, Table~\ref{tab:tab} presents a summary of the principal mathematical objects dealt with in the analysis, in the general case of a given self-adjoint operator $A$, as well as for the Laplace operator $-\Delta$ with Dirichlet boundary conditions and a Schr\"odinger operator $-\Delta+V$ on a cubic box with periodic boundary conditions, for which we present numerical simulations below in Section~\ref{sec_num}.

\begin{table}[!ht]
   \centering
   \caption{Mathematical objects used in this analysis. For the Laplace operator case, $\Omega$ is a bounded Lipschitz domain of $\R^d$, $d \geq 1$. For the periodic Schr\"odinger operator case, $\Omega=[0,L)^d$, $d \geq 1$,  $L^{2}_\#(\Omega)\eq\{v \in L^{2}_{\rm loc}(\R^d) |\, v \, L\Z^d\text{-periodic}\}$, $H^s_\#(\Omega)=\{v \in H^s_{\rm loc}(\R^d)|\,  v \, L\Z^d\text{-periodic}\}$, $s=1,2$, and $V \in L^\infty_{\#}(\Omega)$, $V \ge 1$, $\gamma \in \LH$ is a rank-$J$ orthogonal projector such that $\mbox{Ran}(\gamma)=
\Span \{ \varphi^\J_{\m},  \ldots,  \varphi^\J_{\M}\}$, where $\varphi^\J_{\m},  \ldots,  \varphi^\J_{\M} \in D(A^{1/2})$ are orthonormal in $\HH$.} \vspace{3mm}

   \label{tab:tab}
   \begin{tabular}{|c|c|c|c|}
      \cline{2-4}
      \multicolumn{1}{c|}{} & General framework & Laplace operator & Schr\"odinger operator \\
      \hline
      Hilbert space & $\HH$ & $L^2(\Omega)$ & $L^2_\#(\Omega)$ \\ \hline
      Operator  & $A$ & $-\Delta$ & $-\Delta +V$ \\ \hline
      \multirow{2}{*}{Domain} & \multirow{2}{*}{$D(A)$} & $\{ v \in H^1_0(\Omega)| \, $ & \multirow{2}{*}{$H^2_\#(\Omega)$} \\ 
       & & $\Delta v \in L^2(\Omega)\}$ &    \\ \hline
      Form domain       &  $D(\A^{1/2})$ & $H^1_0(\Omega)$ & $H^1_\#(\Omega)$ \\ \hline
      Norm of $v$ & $\|v\|_\sHH$ & $\left(\int_\Omega |v|^2\right)^{1/2}$ & $\left(\int_\Omega |v|^2\right)^{1/2}$ \\ \hline
      Energy norm  & \multirow{2}{*}{$\|A^{1/2}v\|_\sHH$} & \multirow{2}{*}{$\left(\int_\Omega |\Gr v|^2\right)^{1/2}$} & \multirow{2}{*}{$\left(\int_\Omega (|\Gr v|^2 + V |v|^2)\right)^{1/2}$} \\ 
      of $v$ & & & \\ \hline
      \begin{tabular}{@{}c@{}}Energy norm \\ of $\gamma$ \end{tabular}
        & \begin{tabular}{@{}c@{}} $\|A^{1/2}\gamma\|_\HS$ \\
        $=\big\{\sum_{i=\m}^\M \|A^{1/2} \varphi^\J_{i}\|^2\big\}^{1/2}$
      \end{tabular}
         & \begin{tabular}{@{}c@{}} $\big\{\sum_{i=\m}^\M \int_\Omega |\Gr \varphi^\J_i|^2\}^{1/2} $ \\
      \end{tabular}
         & \begin{tabular}{@{}c@{}} $\big\{\sum_{i=\m}^\M \int_\Omega (|\Gr \varphi^\J_i|^2 $ \\ $+ V |\varphi^\J_i|^2)\big\}^{1/2}$ \\ \end{tabular}
          \\ \hline
   \end{tabular}
\end{table}

\subsection{Residuals and their dual norms}\label{sec_res}

Classically, the derivation of {\em a posteriori} error estimates is based on the notion of the residual and its dual norm.
In our setting, we can define the residual for a single eigenpair as follows, where
 $D(A^{1/2})'$ stands for the dual of $D(A^{1/2})$.

\bd[Single eigenpair residual and its dual norm] For any eigenpair $(\uh{i},\lah{i}) \in V_h \times
\real_{+}$ of~\eqref{eq:discretization}, $\m \leq i \leq \M$, define the {\em residual} $\Res(\uih,\lih)\in D(A^{1/2})'$ by
\bse\begin{equation} \label{eq_res}
	\VpV{\Res(\uh{i},\lah{i})}{v} \eq \lah{i} (\uh{i}, v)_\sHH - \big(A^{1/2} \uh{i}, A^{1/2} v\big)_\sHH \qquad \forall v \in D(A^{1/2}).
\end{equation}
Its {\em dual norm} is then
\begin{equation}
  \label{eq_res_norm}
      \norm{\Res(\uh{i},\lah{i})}_{D(A^{1/2})'} \eq \sup_{\substack{v \in D(A^{1/2})\\ \norm{A^{1/2} v}_\sHH = 1}} \VpV{\Res(\uh{i},\lah{i})}{v}.
\end{equation}
\ese
\ed
To consider the error on the eigenvalue cluster in its globality,
we now define a cluster residual, which is an operator measuring the error with respect to the equation for the whole targeted eigenspace. Note that this operator depends on the approximate density matrix $\gamh$ only, and not on the exact density matrix $\gam$, exactly as the single eigenpair residuals depend on the approximate eigenpairs only.

\bd[Cluster residual] For $\gamh$ defined in \eqref{def:gamh}, define the {\em cluster residual} $\Resc \in \LH$ by
\begin{equation}
    \label{def:globalresidual}
    \Resc \eq A^{1/2} \gamh - A^{-1/2} \big(A^{1/2} \gamh\big)^\dag A^{1/2} \gamh.
\end{equation}

\ed

Note that $\Resc$ is a finite-rank operator of $\LH$, as $\gamh$ is finite-rank, $A^{1/2} \gamh$ is bounded by Lemma~\ref{lem_LH}, and $A^{-1/2} \in \LH$. 
The choice of this definition is motivated by the following remark.

\br[Strong form of the cluster residual] \label{rem_res_strong} When the approximation space in~\eqref{eq:discretization} satisfies $V_h\subset D(\A)$, which is the case for planewave discretizations of periodic Schr\"odinger operators, but not for Lagrange finite element discretizations of the Laplace operator, one could first define a cluster residual in $\LH$ by
\[
    \Resg \eq (1-\gamh) A \gamh.
\]
The corresponding operator for the exact density matrix $(1-\gam)A\gam$ is indeed zero, as ${\rm Ran}(A\gam) = {\rm Ran}(\gam) \subset {\rm Ker (1-\gam)}$. 

Then $A^{-1/2} \Resg = \Resc$ and
\[
    \|\Resc\|_{\HS}^2 = \|A^{-1/2} \Resg\|_{\HS}^2 =  \Tr\big(\Resg^\dag \A^{-1} \Resg\big).
\]
In the case where $V_h\subsetneq D(\A)$, the analog of $\Resg$ is $ (A^{1/2}(1-\gamh))^\dag A^{1/2} \gamh = (A^{1/2})^\dag A^{1/2} \gamh - (A^{1/2}\gamh)^\dag A^{1/2} \gamh$, which multiplied on the left by $A^{-1/2}$ is well-defined.

\er

We now show that the definitions of the single eigenpair and cluster residuals match in the sense that the sum of the dual norms of the single eigenpair residuals is equal to the Hilbert--Schmidt norm of the cluster residual. Therefore, we will be able to estimate the individual dual norms~\eqref{eq_res_norm} by existing tools in Section~\ref{sec:apost_estimates} below.

The following preliminary lemma relates the residual to the exact and approximate eigenpairs.
\bl[Residual expansion]
\label{lem:res_expansion}
   There holds
   \begin{equation}
      \forall s \leq 0, \quad
      \| \A^{s} \Resc \|_{\HS}^2
      =
      \sum_{k\ge 1} \sum_{i=\m}^\M
      \la{k}^{2s-1}(\la{k}-\lah{i})^2
      \left|\left(\uh{i}, \ue{k}\right)_\sHH  \right|^2.
      \label{eq:res_leftpart}
   \end{equation}
\el

\begin{proof} First note that $\A^{s} \Resc \in \LH$ for $s \leq 0$.
Using \eqref{def:gamh} and since $\A^{s}$ is self-adjoint, \eqref{eq_Tr_2_norm} yields
\ban
    \| \A^{s} \Resc \|_{\HS}^2 = {} & \sum_{i=\m}^{\M} \| \A^{s} \Resc \uih \|^2 = \sum_{i=\m}^{\M} \big(\Resc \uih, \A^{2s} \Resc \uih\big)_\sHH \\
    = & {} \sum_{i=\m}^{\M} \left[ \big(A^{1/2} \uih, \A^{2s} A^{1/2} \uih\big)_\sHH - 2 \big(A^{1/2} \uih, \A^{2s} A^{-1/2} \big(A^{1/2} \gamh\big)^\dag A^{1/2} \uih\big)_\sHH \right.\\
    {} & \left.+ \big(A^{-1/2} \big(A^{1/2} \gamh\big)^\dag A^{1/2} \uih, \A^{2s} A^{-1/2} \big(A^{1/2} \gamh\big)^\dag A^{1/2} \uih\big)_\sHH \right]\\
    = & {} \sum_{i=\m}^{\M} \left[ \big(A^{1/2} \uih, \A^{2s} A^{1/2} \uih\big)_\sHH - 2 \big(A^{1/2} \gamh A^{2s} \uih, A^{1/2} \uih\big)_\sHH \right.\\
    {} & \left.+ \big(\big(A^{1/2} \gamh\big)^\dag A^{1/2} \uih, \A^{2s-1} \big(A^{1/2} \gamh\big)^\dag A^{1/2} \uih\big)_\sHH \right]\\
    \qe \!\!\! {} & \sum_{i=\m}^{\M} [T_{1i} + T_{2i} + T_{3i}].
\ean
We now treat the three terms separately while expanding the operators $\A^{2s}$, $\A^{1/2}$, and $\A^{2s-1}$ on the eigenvectors using~\eqref{eq_As_expr}. This gives, noting that $A^{1/2}$ is self-adjoint,
\ban
    T_{1i} & = \left(A^{1/2} \uih, \sum_{k \geq 1} \la{k}^{2s} \big(A^{1/2} \uih, \ue{k}\big)\ue{k}\right) = \sum_{k \geq 1} \la{k}^{2s} \big|\big(A^{1/2} \uih, \ue{k}\big)\big|^2
    = \sum_{k \geq 1} \la{k}^{2s + 1} \left|\left(\uih, \ue{k}\right)\right|^2
\ean
and similarly, also using~\eqref{eq:discretization},
\ban
    T_{2i} & = - 2 \left(A^{1/2} \gamh \sum_{k \geq 1} \la{k}^{2s} \left(\uih, \ue{k}\right)\ue{k}, A^{1/2} \uih\right)_\sHH
    \\ &
    = - 2 \sum_{k \geq 1} \la{k}^{2s} \left(\uih, \ue{k}\right)\big(A^{1/2} \gamh \ue{k}, A^{1/2} \uih\big)_\sHH \\
    & = - 2 \sum_{k \geq 1} \la{k}^{2s} \left(\uih, \ue{k}\right) \lih \left(\uih, \gamh \ue{k}\right)_\sHH = - 2 \sum_{k \geq 1} \la{k}^{2s} \lih \left|\left(\uih, \ue{k}\right)\right|^2.
\ean
Finally, relying again on~\eqref{eq:discretization},
\ban
    T_{3i} & = \left(\big(A^{1/2} \gamh\big)^\dag A^{1/2} \uih,
    \sum_{k \geq 1} \la{k}^{2s-1} \big(\big(A^{1/2} \gamh\big)^\dag A^{1/2} \uih, \ue{k}\big)\ue{k} \right)_\sHH \\
    & = \sum_{k \geq 1} \la{k}^{2s-1} \big|\big(A^{1/2} \uih, A^{1/2} \gamh \ue{k}\big) \big|^2 = \sum_{k \geq 1} \la{k}^{2s-1} \lih^2 \left|\left(\uih, \ue{k}\right)\right|^2.
\ean
Developing the square in~\eqref{eq:res_leftpart} finishes the proof.
\end{proof}

We can now state the correspondence between the cluster residual and the single eigenpair residuals.

\bl[Relation between cluster and single eigenpair residuals]
\label{Lem:res} There holds
\[
    \|\Resc\|_{\HS}^2 = \sum_{i=\m}^\M \norm{\Res(\uih,\lih)}_{D(A^{1/2})'}^2.
\]
\el

\bp
For each $\m \leq i \leq \M$, define the Riesz representation of the residual $\res \in D(A^{1/2})$ such that
\be \label{eq_Riesz}
    \big(A^{1/2} \res, A^{1/2} v \big) = \VpV{\Res(\uh{i},\lah{i})}{v} \qquad \forall v \in D(A^{1/2}).
\ee
Consequently,
\[
    \norm{A^{1/2} \res} = \norm{\Res(\uh{i},\lah{i})}_{D(A^{1/2})'}.
\]
Moreover, using that $A^{-1/2}$ is self-adjoint and $A^{-1/2}A^{1/2}=1$, we see from~\eqref{eq_res} that
\[
    \big(A^{1/2} \res, A^{1/2} v \big) = \lah{i} \big(A^{-1/2} \uh{i}, A^{1/2} v\big)_\sHH - \big(A^{1/2} \uh{i}, A^{1/2} v\big)_\sHH \qquad \forall v \in D(A^{1/2}),
\]
so that
\[
    A^{1/2} \res = \lah{i} A^{-1/2} \uh{i} - A^{1/2} \uh{i}.
\]
Expressing the norms related to $\A^{-1/2}$, $\A^0$, and $\A^{1/2}$ via~\eqref{eq_As_dom_norm}, we conclude therefrom that
\be \label{eq_Riesz_dev} \bs
    \norm{A^{1/2} \res}^2 & = \lah{i}^2 \norm{A^{-1/2}\uh{i}}^2 - 2 \lah{i} \big(\uh{i}, \uh{i}\big) + \norm {A^{1/2} \uh{i}}^2 \\
    & = \lah{i}^2 \sum_{k \geq 1} \la{k}^{-1} |(\uh{i},\ue{k})_\sHH|^2 - 2 \lah{i} \sum_{k \geq 1} |(\uh{i},\ue{k})_\sHH|^2 + \sum_{k \geq 1} \la{k} |(\uh{i},\ue{k})_\sHH|^2,
\es \ee
and the assertion follows using~\eqref{eq:res_leftpart} with $s=0$.
\ep

\section{Error equivalences} \label{sec_equiv}

The framework is now ready to prove {\em a posteriori} estimates for $\|A^{1/2}(\gam-\gamh)\|_{\HS}$ and the sum of the eigenvalues errors in terms of the cluster residual $\Resc$. These results extend~\cite[Theorems 3.4, 3.5 and  Lemmas 3.1, 3.2]{Canc_Dus_Mad_Stam_Voh_eigs_conf_17} to the case of eigenvalue clusters, and especially cover the case of degenerate eigenvalues.

\subsection{Eigenvalue error equivalence}

We first show how to estimate the sum of the eigenvalues errors in terms of errors on the density matrix.

\bt[Eigenvalue bounds]
\label{th:eigenvalue_bounds}
Let Assumption~\ref{as:gap} hold and let the density matrices $\gam$ and $\gamh$ be respectively defined by~\eqref{def:gamma0} and~\eqref{def:gamh}. Then
\begin{equation}
    \label{eq:estimlambda}
    \| \A^{1/2} (\gam - \gamh) \|_\HS^2 - \la{\M}\| \gam - \gamh \|_{\HS}^2
    \le
    \sum_{i=\m}^\M (\lah{i} - \la{i})
    \le
    \| \A^{1/2} (\gam - \gamh) \|_\HS^2.
\end{equation}
\et

\bp
We start from~\eqref{eq:Agamhgam}, \ie,
\[
   \norm{\A^{1/2}(\gam-\gamh)}_{\HS}^2 =
   \sum_{i=\m}^\M (\lah{i} - \la{i})
   + 2 \sum_{i=\m}^\M \la{i}
   \| (1-\gamh) \ue{i} \|_\sHH^2.
\]
Noting that $ 2\sum_{i=\m}^\M \la{i}
\| (1-\gamh) \ue{i} \|_\sHH^2 \ge 0$ easily proves the right-hand side of~\eqref{eq:estimlambda}.
Moreover, bounding the eigenvalues by the largest in the sum, expressing the sum of the projected eigenvectors as a trace, and using $(1-\gamh)^2=1-\gamh$ yields
\begin{align*}
    2\sum_{i=\m}^\M \la{i}
    \| (1-\gamh) \ue{i} \|_\sHH^2 & \leq 2\la{\M} \sum_{i=\m}^\M \|(1 - \gamh) \ue{i}\|_\sHH^2
   = 2 \la{\M} \Tr(\gam (1-\gamh))
   = \la{\M} \|\gam-\gamh\|_{\HS}^2,
\end{align*}
where we have used Lemma~\ref{lem:expan_gamerr} with $\gam$ and $\gamh$ for the last equality.
The left-hand side of~\eqref{eq:estimlambda} follows.
\ep

\subsection{Eigenvector error equivalence}

We next estimate the energy norm of the density matrix error in terms of the Hilbert--Schmidt norm of the cluster residual $\Resc$. We henceforth often need the following assumption, in addition to Assumption~\ref{as:gap}:

\begin{assumption}[Continuous--discrete gap conditions]
   \label{as:gap_cont_disc}
   There exist $ \ula{\M+1}$,  and $\ola{\m-1}$ if $\m>1$, 
   which we take as $\ola{\m-1} = \lah{(\m-1)}$,
   such that there holds
\[
    \la{\m-1} \le \ola{\m-1} < \lah{\m} \text{ when } \m>1,
    \qquad
    \lah{\M} < \ula{\M+1} \le \la{\M+1}.
\]
\end{assumption}

For practical use, we usually proceed as follows.

\br[Verification of Assumptions~\ref{as:gap} and~\ref{as:gap_cont_disc}, choice of $ \ula{\M+1}$,
and uniqueness of the discrete projector $\gamma_h$] \label{rem_pract} 
Let $\ula{\m}$ be a guaranteed lower bound for $\la{\m}$ and $\ula{\M+1}$ a guaranteed lower bound for $\la{\M+1}$ (obtained by, \eg in a finite element discretization, employing the nonconforming finite element method on a coarse mesh and using the technique presented \cite[Theorem~3.2]{Cars_Ged_LB_eigs_14} or~\cite[formula (6)]{Liu_fram_eigs_15}).
If these bounds are accurate enough, it follows from~\eqref{eq:conf_approx} that we can request
\begin{align}
&\lambda_{m-1} \le \lambda_{(m-1)h} <  \underline{\lambda}_m \le \lambda_m \le \lambda_{mh}, \quad \mbox{when } m > 1,   \label{eq:LBm} \\
&\lambda_{M} \le \lambda_{Mh} <  \underline{\lambda}_{M+1} \le \lambda_{M+1} \le \lambda_{(M+1)h},  \label{eq:LBM+1}
\end{align}
so that: 1) Assumption~\ref{as:gap} is satisfied; 2) Assumption~\ref{as:gap_cont_disc} is satisfied with $\ola{\m-1} = \lah{(\m-1)}$ and hence the constants $\chn$ and $\chtn$ in~\eqref{def:ch} and~\eqref{eq:ch_tilde} will be well-defined; 3) the discrete gap condition of Remark~\ref{as:gap_disc} is satisfied and hence the discrete projector $\gamma_h$ is uniquely defined.
\er

In view of Remark~\ref{rem_pract}, we introduce the following assumption, which implies in particular Assumptions~\ref{as:gap} and~\ref{as:gap_cont_disc} as well as Remark~\ref{as:gap_disc}.

\begin{assumption}[Availability of accurate enough lower bounds for $\la{\m}$ and $\la{\M+1}$]
   \label{as:lowerbounds}
   We know  two real numbers $\ula{m}$ and $ \underline{\lambda}_{M+1}$ such that the inequalities \eqref{eq:LBm}--\eqref{eq:LBM+1} are satisfied.
\end{assumption}

\bt[Upper bounds for the density matrix error]
\label{th:DM_bounds} Let Assumption~\ref{as:gap} hold, let the density matrices $\gam$ and~$\gamh$ be respectively defined by~\eqref{def:gamma0} and~\eqref{def:gamh}, and let the cluster residual $\Resc$ be defined by~\eqref{def:globalresidual}. Then, there holds
\begin{equation}
    \label{eq:3.5.1}
    \| A^{1/2} (\gam - \gamh) \|_\HS^2
    \le
     \| \Resc \|_{\HS}^2 + (\la{\M}+\lah{M})\| \gam - \gamh \|_{\HS}^2.
\end{equation}
Let in addition Assumptions~\ref{as:nonorth} and~\ref{as:gap_cont_disc} hold and set
\begin{equation}
  \label{def:ch}
    \ch \eq \max \left[	\left(\frac{\lah{\m}}{\ola{\m-1}}-1\right)^{-1} , \left(1 - \frac{\lah{\M}}{\ula{\M+1}} \right)^{-1}
            \right],
\end{equation}
the first term in the {\rm max} being discarded for $\m=1$.
Then there also holds that
\begin{equation}
    \label{eq:3.5.1bis}
    \| A^{1/2} (\gam - \gamh) \|_\HS^2
    \le
     2 \ch^2\| \Resc \|_{\HS}^2 + \frac{\la{\M}}{2}\| \gam - \gamh \|_{\HS}^4.
\end{equation}

\et

\bp
To show~\eqref{eq:3.5.1}, let us decompose $\|A^{1/2}(\gam-\gamh)\|_{\HS}^2$ using~\eqref{eq_Tr_2_norm} as
\begin{align*}
	\|A^{1/2}(\gam-\gamh)\|_{\HS}^2 & = \Tr\big(\big(A^{1/2}(\gam-\gamh)\big)^\dag A^{1/2}(\gam-\gamh)\big)\\
  & =  \Tr\big(\big(A^{1/2}(\gam-\gamh)\big)^\dag A^{1/2}\gam\big)
  + \Tr\big(\big(A^{1/2}(\gamh-\gam)\big)^\dag A^{1/2}\gamh\big)\\
  & \qe T_1 + T_2.
\end{align*}
On the one hand, using definition~\eqref{def:gamma0} of $\gam$ and~\eqref{eq:eig_pb_weak},
\begin{align*}
	T_1 &= \sum_{i=\m}^\M ( A^{1/2} \ue{i},  A^{1/2} (\gam-\gamh) \ue{i})_\sHH
	=  \sum_{i=\m}^\M \la{i} ( \ue{i}, (\gam-\gamh) \ue{i})_\sHH.
\end{align*}
Since for $i=\m,\ldots,\M$, $( \ue{i}, (\gam-\gamh) \ue{i})_\sHH = \left(1  - \|\gamh \ue{i}\|_\sHH^2\right)\ge 0$, we can bound the above expression via Lemma~\ref{lem:expan_gamerr} as
\begin{align*}
	T_1 &\le \la{\M}  \sum_{i=\m}^\M  ( \ue{i}, (\gam-\gamh) \ue{i})_\sHH
   =   \la{\M} \Tr(\gam (1-\gamh))
	=  \frac{\la{\M}}{2} \|\gam-\gamh\|_{\HS}^2.
\end{align*}

On the other hand, writing $1 = \big(A^{1/2} \gamh  A^{-1/2}\big)^\dag + (1-\big(A^{1/2} \gamh  A^{-1/2}\big)^\dag)$, using the definition of the cluster residual~\eqref{def:globalresidual}, and employing~\eqref{eq:discretization}, we obtain
\ban
T_2  = {} & \Tr\big(\big(A^{1/2}(\gamh-\gam)\big)^\dag \big(A^{1/2} \gamh  A^{-1/2}\big)^\dag A^{1/2}\gamh\big) \\ & + \Tr\big(\big(A^{1/2}(\gamh-\gam)\big)^\dag (1-\big(A^{1/2} \gamh  A^{-1/2}\big)^\dag) A^{1/2}\gamh\big)\\
 = {} & \sum_{i=\m}^{\M} \big(A^{1/2}\gamh(\gamh-\gam)\uih, A^{1/2} \uih\big) + \Tr\big(\big(A^{1/2}(\gamh-\gam)\big)^\dag \Resc\big)\\
 = {} & \sum_{i=\m}^\M \lah{i}( \uh{i}, (\gamh -\gam) \uh{i})_\sHH + \Tr\big(\big(A^{1/2}(\gamh-\gam)\big)^\dag \Resc\big).
\ean
Using this time that, for all $i=\m,\ldots,\M$, $( \uh{i}, (\gamh-\gam) \uh{i})_\sHH = \left(1  - \|\gam \uh{i}\|_\sHH^2\right)\ge 0$, Lemma~\ref{lem:expan_gamerr}, and~\eqref{eq:point5}, Young's inequality leads to
\begin{align*}
	T_2
  & \le  \lah{\M} \Tr(\gamh(\gamh-\gam))
  + \norm{\A^{1/2}(\gam-\gamh)}_{\HS} \| \Resc \|_{\HS} \\
	&\le  \frac{\lah{\M}}{2} \|\gam-\gamh\|_{\HS}^2
   + \frac{1}{2} \norm{\A^{1/2}(\gam-\gamh)}_{\HS}^2
   + \frac{1}{2} \| \Resc \|_{\HS}^2.
\end{align*}
Putting these contributions together, we get
\begin{align*}
   \|A^{1/2}(\gam-\gamh)\|_{\HS}^2 \le {}  &   \frac{\la{\M}+\lah{\M}}{2}\|\gam-\gamh\|_{\HS}^2 
   \\ & + \frac{1}{2} \| \Resc \|_{\HS}^2 + \frac{1}{2} \norm{\A^{1/2}(\gam-\gamh)}_{\HS}^2,
\end{align*}
from which we deduce~\eqref{eq:3.5.1}.

To show~\eqref{eq:3.5.1bis}, we start from~\eqref{eq:res_leftpart} with $s=0$ which reads
\[
   \| \Resc \|_{\HS}^2
   =  \sum_{k\ge 1} \sum_{i=\m}^\M \la{k}
   \left(1- \frac{\lah{i}}{\la{k}}\right)^2
   \left|\left(\uh{i}, \ue{k}\right)_\sHH  \right|^2.
\]
Remark now that for $k\ge 1$, $k \notin \{\m, \ldots, \M \}$,
$(1- \frac{\lah{i}}{\la{k}})^2 \geq \ch^{-2}$ under Assumption~\ref{as:gap_cont_disc}, similarly as in \cite[proof of Lemma 3.1]{Canc_Dus_Mad_Stam_Voh_eigs_conf_17}. Thus, dropping some non-negative terms and introducing the density matrix $\gamh$, we obtain
\begin{align*}
    \| \Resc \|_{\HS}^2
    \ge {} & \sum_{\substack{k\ge 1 \\ k\notin \{\m, \ldots, \M \}}} \sum_{i=\m}^\M  \la{k} \left(1- \frac{\lah{i}}{\la{k}}\right)^2
    \left|\left(\uh{i}, \ue{k}\right)_\sHH  \right|^2 \\
    \ge {} & \ch^{-2}\sum_{\substack{k\ge 1 \\ k\notin \{\m, \ldots, \M \}}} \sum_{i=\m}^\M  \la{k}
    \left|\left(\uh{i}, \ue{k}\right)_\sHH  \right|^2 \\
    = {} & \ch^{-2}\sum_{\substack{k\ge 1 \\ k\notin \{\m, \ldots, \M \}}}  \la{k}
    \left( \ue{k} , \gamh \ue{k} \right)_\sHH.
\end{align*}
Now, introducing the rotated discrete eigenvectors $(\uho{i})_{i=1,\ldots,\M}$ defined by~\eqref{eq:unitary_min} through the expression of the density matrix~\eqref{eq:rotated_DM}, using the orthonormality of the eigenvectors $(\ue{k})$, and definition~\eqref{eq_As_dom_norm} of $\norm{\A^{1/2} v}_\sHH$, there holds
\begin{align*}
  \| \Resc \|_{\HS}^2
  \ge {} & \ch^{-2}\sum_{\substack{k\ge 1 \\ k\notin \{\m, \ldots, \M \}}} \sum_{i=\m}^\M  \la{k}
  \left|\left(\uho{i}-\ue{i}, \ue{k}\right)_\sHH  \right|^2 \\
  = {} & \ch^{-2} \left(\sum_{k\ge 1} \sum_{i=\m}^\M  \la{k}
  \left|\left(\uho{i}-\ue{i}, \ue{k}\right)_\sHH  \right|^2 - \sum_{k=\m}^\M \sum_{i=\m}^\M  \la{k}
  \left|\left(\uho{i}-\ue{i}, \ue{k}\right)_\sHH  \right|^2 \right) \\
  \ge {} & \ch^{-2} \left( \sum_{i=\m}^\M
  \|A^{1/2}(\uho{i}-\ue{i})\|_\sHH^2  - \la{\M}\sum_{k=\m}^\M \sum_{i=\m}^\M
  \left|\left(\ue{k}, \uho{i}-\ue{i}\right)_\sHH  \right|^2 \right).
\end{align*}

Since the eigenvectors composing $\Phiho$ are orthonormal and using Assumption~\ref{as:nonorth}, the $\J\times\J$  overlap matrix $\Over$ with entries $\left(\Over\right)_{i,k} = (\uho{i},\ue{k})_\sHH$ is symmetric (see~\cite[Lemma 4.3]{Cances2012-bo}).
Hence using once again that the eigenvectors are orthonormal, we obtain that for any $i,k=\m,\ldots,\M,$ $i\neq k$,
\[
   \left(\ue{k},\uho{i}-\ue{i}\right)_\sHH
   = \left(\ue{k},\uho{i}\right)_\sHH = \left(\ue{i},\uho{k}\right)_\sHH
   = \frac{1}{2} \left(
   \left(\uho{i},\ue{k}\right)_\sHH
   +\left(\uho{k},\ue{i}\right)_\sHH\right)
   = \frac{1}{2} \left(\uho{k}-\ue{k},\ue{i}-\uho{i}\right)_\sHH.
\]
Since for $i=\m,\ldots,\M,$ $\left(\ue{i},\uho{i}-\ue{i}\right)_\sHH = -\frac{1}{2} \|\ue{i}-\uho{i}\|^2_\sHH $, we obtain that for any $i,k=\m,\ldots,\M,$
\begin{equation}
   \label{eq:estimL2}
   \left(\ue{k},\uho{i}-\ue{i}\right)_\sHH = \frac{1}{2} \left(\uho{k}-\ue{k},\ue{i}-\uho{i}\right)_\sHH.
\end{equation}
From~\eqref{eq:estimL2}, definition~\eqref{eq:vecs}, and the Cauchy--Schwarz inequality, we deduce
\begin{align*}
  \| \Resc \|_{\HS}^2
  \ge {} & \ch^{-2} \left(
  \norm{\A^{1/2}(\Phie-\Phiho)}_\sHH^2  - \frac{\la{\M}}{4} \sum_{k=\m}^\M \sum_{i=\m}^\M
  \left|\left( \ue{k} - \uho{k} , \ue{i}-\uho{i}\right)_\sHH  \right|^2 \right) \\
  \ge {} & \ch^{-2} \left(
  \norm{\A^{1/2}(\Phie-\Phiho)}_\sHH^2  - \frac{\la{\M}}{4}
   \left[\sum_{i=\m}^\M
  \norm{\ue{i} - \uho{i}}^2_\sHH\right]^2
  \right) \\
  = {} & \ch^{-2} \left(
  \norm{\A^{1/2}(\Phie-\Phiho)}_\sHH^2  - \frac{\la{\M}}{4}
  \norm{\Phie-\Phiho}^4_\sHH
  \right).  \\
\end{align*}
Finally, using~\eqref{eq:eqL2} and~\eqref{eq:eqH1} finishes the proof of
\eqref{eq:3.5.1bis}.
\ep

\bt[Lower bound for the density matrix error]
\label{th:lower_bound_DM}
Let Assumption~\ref{as:gap} and~\ref{as:nonorth} hold, let the density matrices $\gam$ and~$\gamh$ be respectively defined by~\eqref{def:gamma0} and~\eqref{def:gamh}, and let the cluster residual $\Resc$ be defined by~\eqref{def:globalresidual}. Set
\begin{equation} \label{eq_chb}
  \chb \eq \max\left\{ \left(\frac{\lah{\M}}{\la{1}} -1  \right)^2,1\right\}.
\end{equation}
Then, there holds
\be \label{eq:3.5.2} \bs
  {} & \| \Resc \|_{\HS}^2 \\
  \le {} & \chb
  \| A^{1/2} (\gam - \gamh)\|^2_\HS +  \frac{3 (\la{\M} -\la{\m})^2}{4\la{\m}} \|\gam - \gamh\|_\HS^4 + \frac{3}{\la{\m}} \left( 1 + \frac{1}{4}\|\gam - \gamh\|_\HS^4 \right)\times \\
  {} &
  \left[2   \left( 1+ \frac{\la{\M}}{4\la{\m}}\|\gam - \gamh\|_{\HS}^2 \right)^2
  \norm{\A^{1/2}(\gam-\gamh)}_{\HS}^4
  + 2{(\la{\M})^2}
  \|\gam - \gamh\|_{\HS}^4\right].
\es \ee
\et

\bp First, let us define the Lagrange multiplier matrix of the orthonormality constraints for $\Phiho$ defined in~\eqref{eq:unitary_min} by
\begin{equation}
\label{eq:lambdad}
	\Lh{} = (\Lh{ij})_{\m\le i,j \le \M}
    \eq( (\A^{1/2}\uho{i}, \A^{1/2}\uho{j})_\sHH )_{\m\le i,j \le \M} \in \R^{\J \times \J}.
\end{equation}
 Note that the matrix $\Lh{}$ is not diagonal in general. However, the matrix of the Lagrange multipliers of the orthonormality constraints for $\Phie$ is diagonal, from~\eqref{eq:eig_pb_weak}. It is denoted by
 \begin{equation}
   \label{eq:lambdae_lagrange}
   \La{}
   \eq(\delta_{ij} \la{i})_{\m\le i,j \le \M} \in \R^{\J \times \J}.
 \end{equation}
Using $1 = \gam + (1-\gam)$ and $(1-\gam)\gam = 0$, there holds
\be \label{eq_dec}
   \| \Resc \|_{\HS}^2 =
   \|  (1-\gam) \Resc \|_{\HS}^2 +  \| \gam \Resc  \|_{\HS}^2.
\ee

To estimate the first term in~\eqref{eq_dec}, we note that the development performed in Lemma \ref{lem:res_expansion} can be done similarly in this case, leading to
\begin{align*}
   \|  (1-\gam) \Resc \|_{\HS}^2 {} = & \sum_{\substack{k\ge 1 \\ k\notin \{\m,\ldots,\M\}}} \sum_{i=\m}^\M
   \la{k}\left( 1-\frac{\lah{i}}{\la{k}}  \right)^2
\left|\left(\uh{i}, \ue{k}\right)_\sHH  \right|^2.
\end{align*}
Bounding the eigenvalue term by $\chb$, using the self-adjointness of $A^{1/2}$, and employing the expansion~\eqref{eq_As_expr}, the Parseval equality~\eqref{eq_Pars}, the Hilbert--Schmidt norm definition~\eqref{eq_Tr_2_norm}, and the definitions of the projectors~\eqref{def:gamma0} and~\eqref{def:gamh}, we obtain
\begin{align*}
   \|  (1-\gam) \Resc \|_{\HS}^2 \le {} & \max_{\substack{i \in \{\m,\ldots,\M\} \\ k\notin \{\m,\ldots,\M\}}}
\left( 1-\frac{\lah{i}}{\la{k}}  \right)^2
\sum_{\substack{k\ge 1 \\ k\notin \{\m,\ldots,\M\}}}  \sum_{i=\m}^\M \la{k}
\left|\left(  \uh{i} , \ue{k} \right)_\sHH \right|^2 \\
\le  {} & \chb
 \sum_{\substack{k\ge 1 \\ k\notin \{\m,\ldots,\M\}}}  \sum_{i=\m}^\M
\left|\left(  \uh{i} , A^{1/2} \ue{k} \right)_\sHH \right|^2
\\
=  {} & \chb
\|  (1-\gam) A^{1/2}\gamh \|^2_\HS \\
\leq {} & \chb \| A^{1/2} (\gam - \gamh)\|^2_\HS,
\end{align*}
where the last estimate follows by a Pythagorean equality as~\eqref{eq_dec} and the fact that $(1-\gam) A^{1/2}\gam = 0$.

To deal with the second term in~\eqref{eq_dec}, first note that from the definition of the residual~\eqref{def:globalresidual}
\begin{align*}
   \| \gam \Resc  \|_{\HS}^2 = \left\|\gam\left(A^{1/2} \gamh - A^{-1/2} \big(A^{1/2} \gamh\big)^\dag A^{1/2} \gamh\right)\right\|_\HS^2.
\end{align*}
Expanding $\gamh$ on the rotated eigenvector basis $(\uho{\m},\ldots,\uho{\M})$ defined in~\eqref{eq:unitary_min} and using that $\gam$ is self-adjoint leads to, as in Lemma \ref{lem:res_expansion},
\begin{align*}
   \| \gam \Resc  \|_{\HS}^2 = & {} \sum_{i=\m}^{\M} \left[ \big(A^{1/2} \uho{i}, \gam A^{1/2} \uho{i}\big)_\sHH - 2 \big( A^{1/2} \uho{i}, \gam A^{-1/2} \big(A^{1/2} \gamh\big)^\dag A^{1/2} \uho{i}\big)_\sHH \right.\\
   {} & \left.+ \big(\big(A^{1/2} \gamh\big)^\dag A^{1/2} \uho{i}, \A^{-1/2}\gam\A^{-1/2} \big(A^{1/2} \gamh\big)^\dag A^{1/2} \uho{i}\big)_\sHH \right]\\
   =: {} & \!\!  \sum_{i=\m}^\M [ T_{1i} +T_{2i} + T_{3i} ].
\end{align*}
First, expanding $\gam$ and using the self-adjointness of $A^{1/2}$ leads to
\begin{align*}
   T_{1i} {} & = \sum_{k=\m}^\M \left| (\ue{k}, A^{1/2}\uho{i}) \right|^2  = \sum_{k=\m}^\M \la{k} \left| (\ue{k}, \uho{i}) \right|^2.
\end{align*}
Second, we expand $A^{-1/2}$ on the eigenvectors $\ue{k}$ following~\eqref{eq_As_expr} and we use the self-adjointness of $A^{1/2}$ and the definition of the Lagrange multipliers \eqref{eq:lambdad}
to obtain
\begin{align*}
   T_{2i}  & =  - 2 \sum_{k=\m}^\M \frac{1}{\sqrt{\la{k}}}\big( A^{1/2} \uho{i}, \ue{k}\big)  \big( \ue{k}, \big(A^{1/2} \gamh\big)^\dag A^{1/2} \uho{i}\big)_\sHH \\
   & =  - 2 \sum_{k=\m}^\M \big( \uho{i}, \ue{k}\big)  \big( \big(A^{1/2} \gamh\big) \ue{k}, A^{1/2} \uho{i}\big)_\sHH \\
   & = - 2 \sum_{k=\m}^\M \sum_{j=\m}^\M \big( \uho{i}, \ue{k}\big)
   \big( A^{1/2} \uho{j},  A^{1/2} \uho{i}\big)_\sHH
   \big(  \uho{j}, \ue{k}\big) \\
   & = - 2 \sum_{k=\m}^\M \sum_{j=\m}^\M \Lh{ij} \big( \uho{i}, \ue{k}\big)
   \big(  \uho{j}, \ue{k}\big).
\end{align*}
Third, using the definition of $\gam$ and expanding $A^{-1/2}$ two-times on the eigenvectors $\ue{k}$ as well as $\gamh$ on the rotated eigenvector basis $(\uho{\m},\ldots,\uho{\M})$ leads to
\begin{align*}
   T_{3i} & = \sum_{k=\m}^\M \frac{1}{\la{k}}
   \big(\big(A^{1/2} \gamh\big)^\dag A^{1/2} \uho{i}, \ue{k} \big) \big( \ue{k}, \big(A^{1/2} \gamh\big)^\dag A^{1/2} \uho{i}\big)_\sHH \\
   & = \sum_{k=\m}^\M \frac{1}{\la{k}}
   \big( A^{1/2} \uho{i}, A^{1/2} \gamh \ue{k} \big) \big(A^{1/2} \gamh \ue{k}, A^{1/2} \uho{i}\big)_\sHH \\
   & = \sum_{k=\m}^\M \frac{1}{\la{k}} \sum_{j=\m}^\M \sum_{p=\m}^\M
   \big( A^{1/2} \uho{i}, A^{1/2} \uho{j}\big) \big( \uho{j}, \ue{k} \big) \big(A^{1/2} \uho{p}, A^{1/2} \uho{i}\big)
   \big( \uho{p}, \ue{k} \big) \\
   & = \sum_{k=\m}^\M \frac{1}{\la{k}} \sum_{j=\m}^\M \sum_{p=\m}^\M
   \Lh{ij} \Lh{ip} \big( \uho{j}, \ue{k} \big)
   \big( \uho{p}, \ue{k} \big).
\end{align*}
Putting $T_{1i}, T_{2i}, T_{3i}$ together, we can write
\begin{align*}
   \| \gam \Resc  \|_{\HS}^2
   = {} & \sum_{k=\m}^\M \sum_{i=\m}^\M  \frac{1}{\la{k}}
   \left( \la{k} (\uho{i},\ue{k}) - \sum_{j=\m}^\M
   \Lh{ij}(\uho{j},\ue{k})
   \right)^2
   \\
   = {} & \sum_{k=\m}^\M \sum_{i=\m}^\M  \frac{1}{\la{k}}
   \Bigg( (\delta_{ik}\la{k} - \Lh{ik})  (\uho{k},\ue{k})
   + (1 - \delta_{ik})(\la{k}- \La{ii}) (\uho{i},\ue{k})
   \\
   {} & - \sum_{\substack{j=\m \\ j\neq k}}^\M
   (\Lh{ij} - \delta_{ij}\La{ij})(\uho{j},\ue{k})
   \Bigg)^2.
\end{align*}

Using definition~\eqref{eq:lambdae_lagrange}, the inequality $(a+b+c)^2 \le 3(a^2+b^2+c^2)$, and using that the eigenvectors are orthonormal, we obtain
\begin{align*}
   \| \gam \Resc  \|_{\HS}^2
   = {} & \sum_{k=\m}^\M \sum_{i=\m}^\M  \frac{1}{\la{k}}
   \Bigg( (\La{ik} - \Lh{ik})  (\uho{k},\ue{k})
   + (1 - \delta_{ik})(\La{kk}- \La{ii}) (\uho{i},\ue{k})
   \\
   {} & - \sum_{\substack{j=\m \\ j\neq k}}^\M
   (\Lh{ij} - \La{ij})(\uho{j},\ue{k})
   \Bigg)^2
   \\
   \le {} & 3 \sum_{k=\m}^\M \sum_{i=\m}^\M
   \frac{1}{\la{k}} (\La{ik} - \Lh{ik})^2  (\uho{k},\ue{k})^2 \\
   {} & + 3 \sum_{k=\m}^\M \sum_{i=\m}^\M
   \frac{1}{\la{k}}
   (1 - \delta_{ik})((\La{kk}- \La{ii}))^2 (\uho{i}-\ue{i},\ue{k})^2
   \\
   {} & + 3 \sum_{k=\m}^\M \sum_{i=\m}^\M
   \frac{1}{\la{k}}
    \left( \sum_{\substack{j=\m \\ j\neq k}}^\M
   (\Lh{ij} - \La{ij})(\uho{j}-\ue{j},\ue{k}) \right)^2.
\end{align*}
Noting that $|(\uho{k},\ue{k})| \le 1$ for $k=\m,\ldots,\M$, using~\eqref{eq:estimL2}, the Cauchy--Schwarz inequality, and~\eqref{eq:eqL2}, we get
\begin{align*}
  \| \gam \Resc  \|_{\HS}^2
   \le {} & \frac{3}{\la{\m}} \| \La{} - \Lh{}\|_{\mathrm{F}}^2
   +  \frac{3 (\la{\M}-\la{\m})^2}{4\la{\m}} \sum_{k=\m}^\M \sum_{i=\m}^\M
   (\uho{i} - \ue{i},\ue{k}-\uho{k})^2
   \\
   {} & + \frac{3}{4\la{\m}} \sum_{k=\m}^\M \sum_{i=\m}^\M
    \left( \sum_{\substack{j=\m \\ j\neq k}}^\M
   (\Lh{ij} - \La{ij})(\uho{j}-\ue{j},\ue{k}-\uho{k}) \right)^2
   \\
   \le {} & \frac{3}{\la{\m}} \| \La{} - \Lh{}\|_{\mathrm{F}}^2
   +  \frac{3 (\la{\M}-\la{\m})^2}{4\la{\m}} \|\Phiho - \Phie\|_\sHH^4
   + \frac{3}{4\la{\m}} \| \La{} - \Lh{}\|_{\mathrm{F}}^2  \|\Phiho - \Phie\|_\sHH^4
   \\
   \le {} & \frac{3}{\la{\m}} \| \La{} - \Lh{}\|_{\mathrm{F}}^2 \left( 1 + \frac{1}{4}\|\gam - \gamh\|_\HS^4 \right)
 +  \frac{3 (\la{\M}-\la{\m})^2}{4\la{\m}} \|\gam - \gamh\|_\HS^4,
\end{align*}
where $\| \cdot \|_{\mathrm{F}}$ is the matrix Frobenius (or Hilbert--Schmidt) norm.
Combining the estimates for the two summands in~\eqref{eq_dec}, the dual norm of the residual can be bounded by
\begin{align}
   \| \Resc \|_{\HS}^2
   \le {} & \chb
   \| A^{1/2} (\gam - \gamh)\|^2_\HS +  \frac{3 (\la{\M}-\la{\m})^2}{4\la{\m}} \|\gam - \gamh\|_\HS^4
   \nonumber\\
   {} & +     \frac{3}{\la{\m}} \| \La{} - \Lh{}\|_{\mathrm{F}}^2 \left( 1 + \frac{1}{4}\|\gam - \gamh\|_\HS^4 \right). \label{eq:temp_estimate}
\end{align}

We are left with estimating the Lagrange multipliers error in the Frobenius norm. For $\m \le i,j \le \M,$ and using~\eqref{eq:estimL2}, there holds
\begin{align*}
   \Lh{ij} - \La{ij}
    = {} & (  \A^{1/2} \uho{i},  \A^{1/2} \uho{j} )_\sHH
    - ( \A^{1/2} \ue{i}, \A^{1/2} \ue{j} )_\sHH \\
     = {} &  (  \A^{1/2}(\uho{i}-\ue{i}),  \A^{1/2} (\uho{j}-\ue{j}) )_\sHH
    \\ & + (  \A^{1/2}\ue{i}, \A^{1/2}  (\uho{j} - \ue{j}) )_\sHH
    + (  \A^{1/2}(\uho{i} - \ue{i}),  \A^{1/2} \ue{j} )_\sHH \\
    ={} & (  \A^{1/2}(\uho{i}-\ue{i}),  \A^{1/2} (\uho{j}-\ue{j}) )_\sHH
     \\ & + \la{i} ( \ue{i},  \uho{j} - \ue{j} )_\sHH
     +  \la{j} (\ue{j}, \uho{i} - \ue{i})_\sHH \\
     ={} & (  \A^{1/2}(\uho{i}-\ue{i}),  \A^{1/2} (\uho{j}-\ue{j}) )_\sHH
     \\ & 
     + \frac{\la{i}}{2}( \ue{i}-\uho{i},  \uho{j} - \ue{j} )_\sHH
     +  \frac{\la{j}}{2} (\ue{j}-\uho{j}, \uho{i} - \ue{i})_\sHH.
\end{align*}
Using the Cauchy--Schwarz inequality,
\begin{align*}
   |\Lh{ij} - \La{ij}|  \le & 
   \| \A^{1/2} (\uho{j} - \ue{j}) \|_\sHH \| \A^{1/2} (\uho{i} - \ue{i}) \|_\sHH \\ &
   +\frac{\la{i}+\la{j}}{2} \| \uho{j} - \ue{j} \|_\sHH  \| \uho{i} - \ue{i} \|_\sHH,
\end{align*}
from which we deduce that
\begin{align*}
   |\Lh{ij} - \La{ij}|^2 \le {} &
   2 \| \A^{1/2} (\uho{j} - \ue{j}) \|_\sHH^2 \| \A^{1/2} (\uho{i} - \ue{i}) \|_\sHH^2 \\ &
   + 2\left(\frac{ \la{i}+\la{j}}{2}\right)^2  \| \uho{j} - \ue{j} \|_\sHH^2  \| \uho{i} - \ue{i} \|_\sHH^2.
\end{align*}
Finally, using~\eqref{eq:eqH1} and~\eqref{eq:eqL2}, the estimate for the Frobenius norm goes as
\begin{align}
   \|\Lh{} - \La{}\|_{\mathrm{F}}^2 = {} & \sum_{i,j = \m}^\M |\Lh{ij} - \La{ij}|^2
   \nonumber \\
   \le {} & 2 \left(  \sum_{i=\m}^\M  \| \A^{1/2} (\uho{i} - \ue{i}) \|_\sHH^2  \right)^2
   + 2 {(\la{\M})^2}
   \left(\sum_{i=m}^M \|\uho{i} - \ue{i}\|_\sHH^2 \right)^2
   \nonumber \\
   = {} & 2 \| \A^{1/2} (\Phiho - \Phie) \|_\sHH^4
   + 2(\la{\M})^2
   \|\Phiho - \Phie\|_\sHH^4 \nonumber \\
   \le {} & 2   \left( 1+ \frac{\la{\M}}{4\la{\m}}\|\gam - \gamh\|_{\HS}^2 \right)^2
\norm{\A^{1/2}(\gam-\gamh)}_{\HS}^4
   + 2(\la{\M})^2
   \|\gam - \gamh\|_\HS^4. \label{eq:temp_estimate2}
\end{align}
The result~\eqref{eq:3.5.2} follows from inserting~\eqref{eq:temp_estimate2} into~\eqref{eq:temp_estimate}.
\ep

\subsection{Bound on the $\HH$-norm of the density matrix error}

Finally, we provide two estimates for the Hilbert--Schmidt norm of the density matrix error. The second bound makes appear the Hilbert--Schmidt norm of the cluster residual $\Resc$, already present in the bounds above. The first bound measures the residual further scaled by $\A^{-1/2}$; it is typically sharper but can be less straightforward to estimate further.

\bl[Bounds on the density matrix error]
\label{lem:DM_bounds} Let  Assumptions~\ref{as:gap} and~\ref{as:gap_cont_disc} hold, let the density matrices $\gam$ and~$\gamh$ be respectively defined by~\eqref{def:gamma0} and~\eqref{def:gamh}, and let the cluster residual $\Resc$ be defined by~\eqref{def:globalresidual}. Set
\begin{equation}
   \label{eq:ch_tilde}
   \cht \eq \max \left[	(\ola{\m-1})^{-1/2}\left(\frac{\lah{\m}}{\ola{\m-1}}-1\right)^{-1} ,(\ula{\M+1})^{-1/2} \left(1 - \frac{\lah{\M}}{\ula{\M+1}} \right)^{-1}
      \right],
\end{equation}
the first term in the {\rm max} being discarded for $\m = 1$, and recall $c_h$ is defined in~\eqref{def:ch}.
Then there holds
\begin{equation}
   \label{eq:L2A-1}
   \| \gam - \gamh \|_{\HS} \le \sqrt{2} \ch \| \A^{-1/2}\Resc \|_\HS
\end{equation}
and
\begin{equation}
   \label{eq:L2A-1/2}
   \| \gam - \gamh \|_{\HS} \le \sqrt{2} \cht \| \Resc \|_\HS.
\end{equation}
\el

\bp
First, starting from~\eqref{eq:res_leftpart} with $s = 0$, neglecting again some positive terms in the sum, and bounding below the eigenvalue part with the help of $\cht$, we obtain
\begin{align*}
   \| \Resc \|_\HS^2 = {} & \sum_{k\ge 1} \sum_{i=\m}^\M
   \frac{(\la{k}-\lah{i})^2}{\la{k}}
   |\left(\uh{i}, \ue{k}\right)_\sHH|^2 \\
   \ge {} & \sum_{\substack{k\ge 1 \\ k\notin \{\m, \ldots, \M \}}} \sum_{i=\m}^\M
   \frac{(\la{k}-\lah{i})^2}{\la{k}}
   |\left(\uh{i}, \ue{k}\right)_\sHH|^2 \\
   \ge {} & \cht^{-2}
   \sum_{\substack{k\ge 1 \\ k\notin \{\m, \ldots, \M \}}} \sum_{i=\m}^\M
   |\left(\uh{i}, \ue{k}\right)_\sHH|^2.
\end{align*}
Similarly, for the $\A^{-1/2}$-scaled residual, there holds
\begin{align*}
   \| \A^{-1/2}\Resc \|_\HS^2 = {} & \sum_{k\ge 1} \sum_{i=\m}^\M
   \frac{(\la{k}-\lah{i})^2}{(\la{k})^2}
   |\left(\uh{i}, \ue{k}\right)_\sHH|^2 \\
   \ge {} & \sum_{\substack{k\ge 1 \\ k\notin \{\m, \ldots, \M \}}} \sum_{i=\m}^\M
   \frac{(\la{k}-\lah{i})^2}{(\la{k})^2}
   |\left(\uh{i}, \ue{k}\right)_\sHH|^2 \\
   \ge {} & \ch^{-2}
   \sum_{\substack{k\ge 1 \\ k\notin \{\m, \ldots, \M \}}} \sum_{i=\m}^\M
   |\left(\uh{i}, \ue{k}\right)_\sHH|^2.
\end{align*}
Moreover,
from Lemma~\ref{lem:expan_gamerr}, expanding the expression in terms of the eigenvectors,
\begin{align*}
   \|\gam-\gamh\|_{\HS}^2 & =
   2 \Tr(\gamh(1-\gam))
    = 2 \sum_{i=\m}^\M
   \left(\uh{i}, (1-\gam) \uh{i}\right)_\sHH 
   \\ & =
    2 \sum_{\substack{k\ge 1 \\ k\notin \{\m, \ldots, \M \}}} \sum_{i=\m}^\M
   |\left(\uh{i}, \ue{k}\right)_\sHH|^2,
\end{align*}
from which we easily deduce~\eqref{eq:L2A-1} and~\eqref{eq:L2A-1/2}.
\ep

By combining equations~\eqref{eq:L2A-1} or~\eqref{eq:L2A-1/2} with \eqref{eq:estimlambda} and~\eqref{eq:3.5.1} or~\eqref{eq:3.5.1bis}, it is possible to obtain estimates for the errors on the density matrix as well as on the sum of the eigenvalues which only depend on the dual norm of the cluster residual together with the exact eigenvalues $\la{\m-1},\la{\M}$, and $\la{\M+1}$ (or their corresponding lower and upper bound according to Assumption~\ref{as:gap_cont_disc}); the converse estimate~\eqref{eq:3.5.2}, not necessary in practice to guarantee upper bounds of the error and only used to theoretically assess the efficiency of the estimates, also employs $\la{1}$ and $\la{\m}$.
Note that practically computable bounds on these eigenvalues are obtained following Remark~\ref{rem_pract}. 
Using such bounds, computable estimates are obtained provided the dual norm of the residual $\|\Resc\|_{\HS}$ can be evaluated or estimated. This is possible for specific operators and numerical methods, as illustrated in the next section.

\section{Guaranteed and computable {\it a posteriori} error estimates}
\label{sec:apost_estimates}

In this section, we transform the estimates presented in Section~\ref{sec_equiv} into fully guaranteed and computable estimators in two particular cases. First, we focus on the Laplace operator $-\Delta$ with homogeneous Dirichlet conditions discretized with conforming finite elements, for which the dual norm of the residual was estimated in~\cite{Canc_Dus_Mad_Stam_Voh_eigs_conf_17}, based on~\cite{Prag_Syng_47, Dest_Met_expl_err_CFE_99, Brae_Pill_Sch_p_rob_09, Ern_Voh_p_rob_15}. We then present estimates for a Schr\"odinger operator of the form $-\Delta+V$ on a cubic box with periodic boundary conditions, where $V$ is a bounded-below periodic multiplicative potential, discretized with planewaves, in which case the dual norm of the residual can be easily computed.

\subsection{Finite element discretization of the Laplace operator}
\label{sec:Laplace_apost}

In this section, we consider the Laplace eigenvalue problem with Dirichlet boundary conditions.
Let $\Omega \subset \R^d$, $ d = 2,3$, be a polygonal/polyhedral domain with a Lipschitz boundary.
In this setting, $A=-\Delta$, $\HH = \Lt$, and $D(A^{1/2})=\W\eq \Hoo$. Let $\Hdv$ stand for the space of $[\Lt]^d$
functions with weak divergences in $\Lt$, and let $\Dt$ be the dual of $\Hoo$.
The problem reads: find eigenvector and eigenvalue pairs $(\ue{k},\la{k})$ such that $-\Delta \ue{k} = \la{k} \ue{k}$ in $\Omega$, subject to the orthonormality constraints $(\ue{k},\ue{j})_\sHH = \delta_{kj}$, $k,j \geq 1$. In weak form, this reads:
find $(\ue{k}, \la{k})_\sHH \in \W \times
\R_+$ with $(\ue{k},\ue{j})_\sHH = \delta_{kj}$ such that
\begin{equation}
  \label{eq_problem_weak}
	(\Gr \ue{k}, \Gr v)_\sHH = \la{k} (\ue{k}, v)_\sHH \qquad \forall v \in \W.
\end{equation}
Here, for $\omega \subset \Om$, $(\Gr u, \Gr v)_\omega$ stands for
$\int_\omega \Gr u \scp \Gr v$ and $(u, v)_\omega$ for
$\int_\omega u v$; we also denote $\norm{\Gr v}_\omega^2$ $\eq
\int_\omega |\Gr v|^2$ and $\norm{v}_\omega^2 \eq \int_\omega v^2$ and drop the index
whenever $\omega = \Om$.

We consider a conforming finite element discretization of this problem.
Let $\{\Th\}_h$ be a family of meshes, matching
simplicial partitions of the domain $\Om$. We suppose that it is
shape regular in the sense that there exists a constant $\kappa_\T>0$
such that the ratio of the element diameter and of the diameter of
its largest inscribed ball is uniformly bounded by $\kappa_\T$, \cf\
Ciarlet~\cite{Ciar_78}. A generic element of
$\Th$ is denoted by $\elm$. The set of vertices of $\Th$ is denoted by $\Vh$, the set of interior
vertices by $\Vhint$, the set of vertices located on the boundary by $\Vhext$, and a
generic vertex by $\ver$. We denote by $\Ta$ the patch of elements of $\Th$
which share the vertex $\ver \in \Vh$, by $\oma$ the corresponding
open subdomain, and by $\tn_\oma$ its outward unit normal. We will often
tacitly extend functions defined on $\oma$ by zero outside of $\oma$,
whereas $V_h(\oma)$ stands for the restriction of the space $V_h$ to
$\oma$. Next, $\psi_\ver$ for $\ver \in \Vh$ stands for the piecewise
affine ``hat'' function taking value $1$ at the vertex $\ver$ and
zero at the other vertices. Note that $(\psi_\ver)_{\ver \in \Vh}$
form a partition of unity since $\sum_{\ver \in \Vh} \psi_\ver =
1|_\Om$.

Let $\PP_{s}(\elm)$, $s \geq 0$, stand for the space polynomials on $\elm$ of total degree at most $s$, and $\PP_{s}(\Th)$ for the space of piecewise polynomials on $\Th$, without any continuity requirement at the element interfaces.
The approximation space is $V_h \eq \PP_{p}(\Th) \cap \W$ for a given polynomial degree $p \geq 1$.
Let also $\tV_h \times Q_h \subset \Hdv \times \Lt$ stand for the Raviart--Thomas--N\'ed\'elec (RTN) mixed finite element spaces of order $p+1$, \ie, $\tV_h \eq \{\tv_h \in \Hdv;
\tv_h|_\elm \in [\PP_{p+1}(\elm)]^d + \PP_{p+1}(\elm) \tx\}$ and $Q_h \eq \PP_{p+1}(\Th)$, see Brezzi and Fortin~\cite{Brez_For_91} or Roberts and Thomas~\cite{Ro_Tho_91}. We also denote by $\Pi_{Q_h}$ the
$\Lt$-orthogonal projection onto $Q_h$.

The discretized eigenvalue problem then reads in this case: find $(\uh{k}, \lah{k}) \in V_h \times \R_+$ with $(\uh{k},\uh{j})=\delta_{kj}$, $1 \leq k,j \leq \dim V_h$, such
that
\begin{equation} \label{eq_FE}
	(\Gr \uh{k}, \Gr v_h)_\sHH = \lah{k} (\uh{k}, v_h)_\sHH \qquad \forall v_h \in V_h.
\end{equation}

\subsubsection{Residual norm estimate}

In order to turn the error estimates obtained in~\S\ref{sec_equiv} into practical ones, we need to estimate the dual norm of the residual $\| \Resc \|_{\HS}$, which in turn requires an estimate of
$\norm{\Res(\uh{i},\lah{i})}_{\Dt}$ for $i=\m,\ldots,\M$.
The latter estimate relies on previous works~\cite{Prag_Syng_47, Dest_Met_expl_err_CFE_99,
Brae_Pill_Sch_p_rob_09, Ern_Voh_p_rob_15} and has been presented for the Laplace eigenvalue problem in~\cite{Canc_Dus_Mad_Stam_Voh_eigs_conf_17}. We recall the key points here for the sake of completeness.

From~\eqref{eq_problem_weak}, it is easy to see that for all $i\ge 1,$ there holds $-\Gr \ue{i} \in \Hdv$, with the weak divergence
equal to $\la{i} \ue{i}$. However, this does not hold at the discrete level, \ie, in general, $-\Gr \uh{i} \not \in \Hdv$, and a fortiori $\Dv(-\Gr \uh{i})
\neq \lah{i} \uh{i}$.
We therefore introduce an {\em equilibrated flux
reconstruction}, a vector field $\frh{i}$ constructed from $(\uh{i}, \lah{i})$,
satisfying
\bse \label{eq_flux}
\ba
  \frh{i} & \in \Hdv, \label{eq_flux_rec} \\
  \Dv \frh{i} & = \lah{i} \uh{i}. \label{eq_flux_rec_eq}
\ea
\ese
In the context of conforming finite elements, the flux reconstruction $\frh{i}$ for $i=\m,\ldots,\M$ can be constructed from the following {\em local
constrained minimizations}:

\bd[Equilibrated flux reconstruction] \label{def_fr} For a mesh vertex $\ver \in \Vh$, set
\batn{2}
    & \begin{array}{l}
        \tV_h^\ver \eq \{ \tv_h \in \tV_h(\oma); \tv_h
        \scp \tn_\oma = 0 \text{ on } \pt \oma\}, \\[1mm]
        Q_h^\ver \eq \{q_h \in Q_h(\oma); (q_h, 1)_\oma = 0 \},\\
    \end{array} \quad \qquad & & \ver \in \Vhint, \\[1mm]
    & \begin{array}{l} \tV_h^\ver \eq \{ \tv_h \in \tV_h(\oma); \tv_h
        \scp \tn_\oma = 0 \text{ on } \pt \oma \setminus \pt \Om\}, \\[1mm]
        Q_h^\ver \eq Q_h(\oma), \\
    \end{array} & & \ver \in \Vhext.
\eatn
Then define $\frh{i} \eq \sum_{\ver \in \Vh} \frh{i}^\ver \in \tV_h$, where
$\frh{i}^\ver \in \tV_h^\ver$ solve
\begin{equation}
  \label{eq_fl_equil_min}
     \frh{i}^\ver \eq \arg\min_{\substack{\tv_h \in \tV_h^\ver, \\ \Dv \tv_h =
     \Pi_{Q_h} (\lih \uih \psi_\ver - \Gr \uih \scp \Gr \psi_\ver )}} \norm{\psi_\ver \Gr \uih + \tv_h}_\oma \qquad \forall \ver \in
     \Vh.
\end{equation}
\ed

The Euler--Lagrange equations for~\eqref{eq_fl_equil_min} give the
standard {\em mixed finite element formulation},
\cf~\cite[Remark~3.7]{Ern_Voh_p_rob_15}: find $\frh{i}^\ver \in \tV_h^\ver$
and $p_h^\ver \in Q_h^\ver$ such that
\bse \label{eq_fl_equil} \bat{2}
    (\frh{i}^\ver, \tv_h)_\oma - (p_h^\ver, \Dv \tv_h)_\oma & = - (\psi_\ver  \Gr \uih, \tv_h)_\oma
    & \qquad & \forall \tv_h \in \tV_h^\ver, \label{eq_fl_equil_1}\\
    (\Dv \frh{i}^\ver,q_h)_\oma & = (\lih \uih \psi_\ver -  \Gr \uih \scp \Gr \psi_\ver, q_h)_\oma &
    \qquad &\forall q_h \in Q_h^\ver. \label{eq_fl_equil_2}
\eat\ese
Consequently, $\Dv \frh{i} = \lih \uih$, \cf, \eg,
\cite[Lemma~3.5]{Ern_Voh_p_rob_15}.

On each patch $\oma$ around the vertex $\ver \in \Vh$,
define
\bse \label{eq_Hsa} \bat{2}
    \Hsa & \eq \{v \in \Hoi{\oma}; \, (v, 1)_\oma = 0\}, \qquad \qquad & & \ver \in \Vhint, \\
    \Hsa & \eq \{v \in \Hoi{\oma}; \, v = 0 \text{ on } \pt \oma \cap \pt \Om\}, & & \ver \in
    \Vhext.
\eat \ese

Following Carstensen and
Funken~\cite[Theorem~3.1]{Cars_Funk_full_rel_FEM_00}, Braess~\eal\
\cite[\S3]{Brae_Pill_Sch_p_rob_09}, or~\cite[Lemma~3.12]{Ern_Voh_p_rob_15},
there exists a constant $C_{\rm cont, PF}$ only depending on the mesh
regularity parameter $\kappa_\T$ such that
\begin{equation}
  \label{eq_eff_cont_fr}
     \norm{\Gr(\psi_\ver v)}_\oma \leq C_{\rm cont, PF} \norm{\Gr
         v}_\oma \qquad \forall v \in \Hsa, \, \forall \ver \in \Vh.
\end{equation}
Moreover, the key result of Braess~\eal\
\cite[Theorem~7]{Brae_Pill_Sch_p_rob_09}, see~\cite[Corollaries~3.3
and~3.6]{Ern_Voh_p_rob_3D_16} for $d=3$, states that the
reconstructions of Definition~\ref{def_fr} satisfy the following {\em
stability} property,
\begin{equation}
  \label{eq_MFE_stab}
      \norm{\psi_\ver  \Gr \uih + \frh{i}^\ver}_\oma \leq C_{\rm st} \sup_{\substack{v \in \Hsa\\
          \norm{\Gr v}_\oma = 1}} \{ \langle \Res(\uih,\lih),\psi_\ver v\rangle_{\Dt,\Hoor} \}.
\end{equation}
The constant $C_{\rm st} > 0$ again only depends on $\kappa_\T$, and a
computable upper bound on $C_{\rm st}$ is given
in~\cite[Lemma~3.23]{Ern_Voh_p_rob_15}.

In this setting, the dual norm of the residual can be bounded as follows.

\bt[Residual equivalences] \label{thm_res_bounds} For $i=\m,\ldots,\M,$
let $(\uh{i},\lah{i}) \in
V_h \times \real$ be defined in~\eqref{eq_FE}. Then, for the reconstruction $\frh{i}$ from
Definition~\ref{def_fr},
\bse
\ba
       \norm{\Res(\uh{i},\lah{i})}_{\Dt}  & \leq \norm{\Gr \uh{i} + \frh{i}}_\sHH, \label{eq_res_bounds} \\
    \norm{\Gr \uh{i} + \frh{i}}_\sHH & \le (d+1)C_{\rm st} C_{\rm cont, PF}\norm{\Res(\uh{i},\lah{i})}_{\Dt}. \label{eq:efficiency_res}
\ea
\ese
Therefore, there holds
\bse
\ba
         \| \Resc \|_\HS^2 & \le \sum_{i=\m}^\M \norm{\Gr \uh{i} + \frh{i}}_\sHH^2, \label{eq_res_bounds2} \\
     \sum_{i=\m}^\M \norm{\Gr \uh{i} + \frh{i}}_\sHH^2 & \le  (d+1)^2 C_{\rm st}^2 C_{\rm cont, PF}^2  \| \Resc \|_\HS^2. \label{eq:efficiency_res2}
\ea
\ese
\et

\bp Fix $v \in \W$ with $\norm{\Gr v}_\sHH = 1$. Starting from~\eqref{eq_res},
adding and subtracting $(\frh{i}, \Gr v)_\sHH$, applying Green's theorem and
using~\eqref{eq_flux_rec_eq} yields
\[
	\langle\Res(\uh{i},\lah{i}),v\rangle_{\Dt,\Hoor} {} = \lih (\uh{i}, v)_\sHH - (\Gr \uh{i}, \Gr v)_\sHH
        = - (\Gr \uih + \frh{i}, \Gr v)_\sHH.
\]
Then, definition~\eqref{eq_res_norm} of the dual norm of the residual and the
Cauchy--Schwarz inequality yield~\eqref{eq_res_bounds}. This actually also holds when choosing for $\tV_h$ the
cheaper RTN space of order $p$ (instead of $p+1$), as~\eqref{eq_flux_rec_eq} still holds for Definition~\ref{def_fr} with this choice.
As in~\cite{Canc_Dus_Mad_Stam_Voh_eigs_conf_17}, the proof of~\eqref{eq:efficiency_res} relies on~\cite[Lemma 3.22]{Ern_Voh_p_rob_15}, where the weak norm $\norm{\Res(\uh{i},\lah{i})}_{\Dt}$ is treated as in~\cite[Theorems 3.3 and 4.8]{Ciar_Voh_chan_coef_18}.
Finally, the bounds~\eqref{eq_res_bounds2} and~\eqref{eq:efficiency_res2} directly follow from~\eqref{eq_res_bounds}, and~\eqref{eq:efficiency_res} combined with Lemma~\ref{Lem:res}.
\ep

\subsubsection{Final estimates} \label{sec_FE_fin}

We combine here the results of the previous sections to derive the
actual guaranteed and fully computable bounds. We will denote by
$\zeta_{(ih)}$ the solution of the Laplace {\em source problem} $-
\Lap \zeta_{(ih)} = \res$ in $\Om$, $\zeta_{(ih)} = 0$ on $\pt \Om$, \ie, $\zeta_{(ih)} \in \W$ such that
\begin{equation}
   \label{eq_adj}
      (\Gr \zeta_{(ih)}, \Gr v) = (\res, v) \qquad \forall v \in \W,
\end{equation}
where $\res \in \W$ is the {\em Riesz representation} of the residual defined by %
\bse \ba
    (\Gr \res, \Gr v) & = \langle\Res(\uih,\lih),v\rangle_{\Dt,\Hoor} \qquad \forall v \in \W, \label{eq_lift_res}\\
    \norm{\Gr \res}_\sHH & = \norm{\Res(\uih,\lih)}_{\Dt}, \label{eq_lift_res_norm}
\ea \ese
\cf~\eqref{eq_Riesz}.

\bt[Guaranteed bounds for the sum of eigenvalues] \label{thm_bound_eigs}
Let $\m,\M\in\N\backslash \{0\}, M \ge m$, and let Assumptions~\ref{as:gap} and~\ref{as:gap_cont_disc} hold.
For $i=\m,\ldots,\M$, let $(\uh{i},\lah{i})\in V_h \times \R_+$ be given by~\eqref{eq_FE}.
\noindent For $i=\m,\ldots,\M$, let next $\frh{i}$ be constructed following
Definition~\ref{def_fr} and define
\begin{equation}
	\label{eq:GlobInd}
    \eres^2 \eq \sum_{i=\m}^\M \norm{\Gr \uih + \frh{i}}^2_\sHH.
\end{equation}
Recall the notations~\eqref{def:ch} and~\eqref{eq:ch_tilde}.
Then 
\begin{equation}
   \label{eq_eig_bound}
    0 \le \sum_{i=\m}^\M (\lah{i} - \la{i}) \le \eta^2,
\end{equation}
where we distinguish the following two cases:

\noindent {\bf Case I} {\rm (General case)} Let Assumption~\ref{as:nonorth} hold. Then~\eqref{eq_eig_bound} holds with
\begin{equation}
   \eta^2 \eq (2 \chn^2+ 2\lah{\M} \chtn^4 \eres^2  ) \eres^2.
      \label{eq_eig_simpl_down}
\end{equation}
\noindent {\bf Case II} {\rm (Optimal estimates under elliptic regularity
assumption)} Assume that for $i=\m,\ldots,\M,$ the solutions
$\zeta_{(ih)}$ of problems~\eqref{eq_adj} belong to the space
$H^{1+\delta}(\Om)$, $0 < \delta \leq 1$, so that the approximation and
stability estimates
\bse\label{eq_el_reg} \ba
    \min_{v_h \in V_h} \norm{\Gr(\zeta_{(ih)} - v_h)} & \leq C_{\rm I} h^\delta
    |\zeta_{(ih)}|_{H^{1+\delta}(\Om)}, \label{eq_int}\\
    |\zeta_{(ih)}|_{H^{1+\delta}(\Om)} & \leq C_{\rm S} \norm{\res}_\sHH \label{eq_stab}
\ea \ese
are satisfied.
Then~\eqref{eq_eig_bound} holds with
\begin{equation}
   \label{eq_eig_down}
      \eta^2 \eq (1+ 4\lah{\M} \chn^2 C_I^2 C_S^2 h^{2\delta} ) \eres^2.
\end{equation}
\et

\bp{\bf (Case I)} Combining the estimates~\eqref{eq:estimlambda}, \eqref{eq:3.5.1bis}, \eqref{eq:L2A-1/2} together with \eqref{eq:conf_approx} and~\eqref{eq_res_bounds2} yield the result.

\noindent {\bf (Case II)}
The proof is as in Case~I, relying \eqref{eq:3.5.1} instead of~\eqref{eq:3.5.1bis} and on~\eqref{eq:L2A-1} instead of~\eqref{eq:L2A-1/2}. Using the characterization~\eqref{eq:res_leftpart} and similarly as in~\eqref{eq_Riesz_dev} in Lemma~\ref{Lem:res}, one can show that
\[
   \| \A^{-1/2}\Resc \|_\HS^2 = \sum_{i=\m}^\M \|\res\|^2_\sHH,
\]
where $\res$ is defined in~\eqref{eq_lift_res}.
Now, using an Aubin--Nitsche trick,~\eqref{eq_adj}, \eqref{eq_lift_res}, \eqref{eq_res} , and the discrete problem equation~\eqref{eq_FE},
we get
\[
    \norm{\res}^2 = (\Gr \zeta_{(ih)}, \Gr \res) = (\Gr (\zeta_{(ih)} - \zeta_{ih}), \Gr \res),
\]
where $\zeta_{ih} \in V_h$ is the minimizer in~\eqref{eq_int}.
Using the Cauchy--Schwarz inequality, estimates~\eqref{eq_el_reg},
and the characterization~\eqref{eq_lift_res_norm} altogether give
\[
    \norm{\res} \leq C_{\rm I} C_{\rm S} h^\delta
    \norm{\Res(\uih,\lih)}_{\Dt}.
    \ifSIAM \qquad \endproof \else \fi
\]
Therefore, also using Lemma~\ref{Lem:res},
\begin{equation}
   \label{eq:resA-1Laplacien}
      \| \A^{-1/2}\Resc \|_\HS^2 \le (C_{\rm I} C_{\rm S} h^\delta)^2 \sum_{i=\m}^\M \|\Res(\uih,\lih)\|^2_{\Dt} = (C_{\rm I} C_{\rm S} h^\delta)^2\|\Resc\|_{\HS}^2.
\end{equation}
Thus, estimates~\eqref{eq:estimlambda}, \eqref{eq:3.5.1}, \eqref{eq:L2A-1} together with \eqref{eq:conf_approx} and~\eqref{eq_res_bounds2} yields the result.

\ep

\br[Constants $C_{\rm I}$ and $C_{\rm S}$] \label{rem_C_I_C_st}
As discussed in~\cite{Canc_Dus_Mad_Stam_Voh_eigs_conf_17},
it is possible to obtain explicit bounds for the constants $C_{\rm I}$ and $C_{\rm S}$ in particular cases, \eg, when $\Om$ is
a convex polygon in $\real^2$.
In this case, the solution of the source problem
$\zeta_{(ih)}$ of~\eqref{eq_adj} belongs to $H^2(\Om)$ and
$|\zeta_{(ih)}|_{H^2(\Om)} = \norm{\Lap \zeta_{(ih)}} = \norm{\res}$, so it is possible to take
$\delta = 1$ and $C_{\rm S} = 1$, see~\cite[Theorem~4.3.1.4]{Gris_ell_nonsm_85}.
Computable bounds for $C_{\rm I}$ can be found in Liu and
Kikuchi~\cite{Liu_Kik_interp_10}, Carstensen~\eal\
\cite{Cars_Ged_Rim_expl_cnst_12}, and Liu and
Oishi~\cite[\S2]{Liu_Oish_bounds_eig_13}. Note that in the particular case of a mesh formed by isosceles rectangular triangles, there holds $C_{\rm I}
\leq \frac{0.493}{\sqrt 2}$. \er

\br[Improved guaranteed upper bounds for the eigenvalues]
 Similarly as in~\cite[Theorem 5.2]{Canc_Dus_Mad_Stam_Voh_eigs_conf_17}, it is possible to estimate $\norm{\Res(\uih,\lih)}_{\Dt}$ from below and combine this lower bound with~\eqref{eq:estimlambda} and~\eqref{eq:3.5.2} to obtain guaranteed improved upper bounds for the eigenvalues. For brevity, we do not state such results here.
\er

\bt[Guaranteed and polynomial-degree robust bound for the density matrix error]
\label{thm_bound_eigvecs} Let the assumptions of Theorem~\ref{thm_bound_eigs}
be verified. Then the energy density matrix error can be bounded via
\begin{equation}
       \norm{|\Gr| (\gam - \gamh)}_\HS \leq \eta, \label{eq_eigvec_rel}
\end{equation}
where $\eta$ is defined in the Case~I by~\eqref{eq_eig_simpl_down}
and in Case~II by~\eqref{eq_eig_down}.
Moreover, the
density matrix error can be bounded by
\begin{equation}
  \label{eq:gam_err}
  \norm{\gam - \gamh}_\HS \le
  \etaltwo,
\end{equation}
where
\begin{subnumcases}{\etaltwo \eq }
   \sqrt{2} \chtn \eres \quad & \text{(Case I)}, \label{etaltwo_I}
   \\
   \sqrt{2} \chn C_{\rm I} C_{\rm S} h^\delta \eres & \quad \text{(Case II)}. \label{etaltwo_II}
\end{subnumcases}
Recall finally the definition of $\chb$ by~\eqref{eq_chb}. Under Assumption~\ref{as:nonorth}, the estimator~$\eta$ is efficient as
\begin{equation} \label{eq_eigvec_eff} \bs
      \eres^2 \leq {} & (d  +  1)^2 C_{\rm st}^2 C_{\rm cont, PF}^2 \Bigg(
      \chb
      \| |\Gr| (\gam - \gamh)\|^2_\HS +  \frac{3 (\la{\M} -\la{\m})^2}{4\la{\m}} \|\gam - \gamh\|_\HS^4\\
      {} & + \frac{3}{\la{\m}} \left( 1 + \frac{1}{4}\|\gam - \gamh\|_\HS^4 \right)\times\\
      {} & \left[2   \left( 1+ \frac{\la{\M}}{4\la{\m}}\|\gam - \gamh\|_{\HS}^2 \right)^2
      \norm{|\Gr|(\gam-\gamh)}_{\HS}^4
      + 2(\la{\M})^2
      \|\gam - \gamh\|_\HS^4\right]
          \Bigg).
\es \end{equation}
This in particular implies that the bound~\eqref{eq_eigvec_rel} is efficient in the sense that
\be \label{eq_eff}
    \eta \leq C \norm{|\Gr| (\gam - \gamh)}_\HS,
\ee
where $C$ is a constant independent of the mesh size $h$ and the polynomial degree $p$.
\et

\bp
The proof of~\eqref{eq_eigvec_rel} is actually contained in the proof of~\eqref{eq_eig_bound} which relies on~\eqref{eq:estimlambda}. In Case~I, the estimate~\eqref{eq:gam_err} follows from~\eqref{eq:L2A-1/2} and~\eqref{eq_res_bounds2}, whereas in Case~II, the bound~\eqref{eq:gam_err} can be derived from~\eqref{eq:L2A-1} and~\eqref{eq:resA-1Laplacien} combined with~\eqref{eq_res_bounds2}.
The bound~\eqref{eq_eigvec_eff} is a consequence of~\eqref{eq:3.5.2} and~\eqref{eq:efficiency_res2}.
Finally, \eqref{eq_eff} follows from~\eqref{eq_eig_simpl_down} or~\eqref{eq_eig_down} in combination with~\eqref{eq_eigvec_eff}, the (crude) bound $\norm{\gam - \gamh}_{\HS}^2 \le 4 \J$, \cf~\eqref{eq:Trgam2}, the equivalence~\eqref{eq:eqL2}, the Poincar\'e inequality $\norm{\Phie-\Phiho}_{\sHH}^2 \leq \norm{\Gr (\Phie-\Phiho)}_{\sHH}^2 / \la{1}$, and the equivalence~\eqref{eq:eqH1}.
\ep

\subsection{Planewave discretization of a Schr\"odinger operator}
\label{sec:PW_theory}

In this section, we consider a Schr\"odinger-type operator of the form $-\Delta+V$, with periodic boundary conditions.
We denote by $\Omega \subset \R^d$, $d=1,2,3$, the periodic cell, by $\RR$ the periodic lattice, and by $\RRs$ the corresponding dual lattice. For simplicity, we assume that $\Omega = [0,L)^d$, $(L > 0)$, in which case $\RR$ is the cubic lattice $L\Z^d$,
and $\RRs = 2\pi\Z^d$. Our arguments can be easily extended to the general case.
The potential $V$ is multiplicative and satisfies $V\in L^\infty_\#(\Omega)$, where, for ${s}\ge 1$,
\[
  L^{s}_\#(\Omega) = \left\{ v \in L^{s}_{\rm loc}(\R^d), \quad v \quad \RR\text{-periodic}   \right\}.
\]
Up to shifting the operator $-\Delta+V$ by a positive constant, we can assume that $V \ge 1$.

For $\kk\in\RRs$, we denote by $e_\kk(x) = |\Omega|^{-1/2} e^{i \kk \cdot x}$ the planewave with wavevector $\kk$. The family $(e_\kk)_{\kk\in\RRs}$ forms an orthonormal basis of $L^2_\#(\Omega)$. Moreover, for all $v\in L^2_\#(\Omega)$,
\[
    v(x) = \sum_{\kk\in\RRs} \hat v_\kk e_\kk(x), \quad \text{where} \quad
    \hat v_\kk = (e_\kk,v)_{L^2_\#(\Omega)} = |\Omega|^{-1/2} \int_\Omega v(x) e^{-i \kk \cdot x} dx.
\]
Let us take in this case $\HH = L^2_\#(\Omega)$ and $D(A^{1/2})= \W \eq H^1_\#(\Omega)$ endowed with the norm
\[
    \| v \|_{D(A^{1/2})}^2 \eq \| \Gr v\|_\sHH^2 + (v,Vv)_\sHH \ge \|v\|_{H^1_\#(\Omega)}^2,
\]
where we endow the Sobolev spaces of real-valued $\RR$-periodic functions
\begin{align*}
   H^s_\#(\Omega) \eq \Bigg\{
      & v(x) = \sum_{\kk\in\RRs} \hat v_\kk e_\kk(x),
        \quad \text{where} \\ & \quad
      \|v\|_{H^1_\#(\Omega)}^2 \eq \sum_{\kk\in\RRs} (1+|\kk|^2)^s |\hat v_\kk|^2 < \infty,  \text{ and }
      \forall \kk, \, \hat v_{-\kk} = \hat v_\kk^\ast
      \Bigg\},  
\end{align*}
with the inner products
\[
    (v,w)_{H^s_\#(\Omega)} \eq \sum_{\kk\in\RRs} (1+|\kk|^2)^s \overline{\hat v_\kk} \hat w_\kk.
\]
Note that the constraints $\hat v_{-\kk} = \hat v_\kk^\ast$ imply that the functions are real-valued.

The eigenvalue problem reads in this case:
find eigenvector and eigenvalue pairs $(\ue{i},\la{i})$ subject to the orthonormality constraints $(\ue{i},\ue{j})_\sHH = \delta_{ij}$, $i,j \geq 1$, such that $(-\Delta +V) \ue{i} = \la{i} \ue{i}$ in $\Omega$. In weak form, this reads:
find $(\ue{i}, \la{i})_\sHH \in \W \times
\R_+$ with $(\ue{i},\ue{j})_\sHH = \delta_{ij}$ such that
\begin{equation}
  \label{eq_problem_weak_PW}
	(\Gr \ue{i}, \Gr v)_\sHH + (\ue{i}, V v) = \la{i} (\ue{i}, v)_\sHH \qquad \forall v \in \W.
\end{equation}

For $N\in\N\backslash\{0\},$ we consider the approximation space
\[
    \VN \eq \left\{
    \sum_{\substack{\kk\in\RRs \\ |\kk|\le \frac{2\pi}{L}N }} \hat v_\kk e_\kk(x),
    \quad \forall \kk, \, \hat v_{-\kk} = \hat v_\kk^\ast
    \right\}.
\]
The discrete problem then reads: find eigenpairs $(\un{i}, \lan{i}) \in
\VN \times \R_+$ with $(\un{i}, \un{j})_\sHH = \delta_{ij}$, $1 \leq i,j \leq N$, such
that
\begin{equation} \label{eq_PW}
	(\Gr \un{i}, \Gr v_N)_\sHH + (\un{i}, V v_N)_\sHH  = \lan{i} (\un{i}, v_N)_\sHH \qquad \forall v_N \in \VN.
\end{equation}
Given $\m,\M\in \N\backslash \{ 0 \}$, $\M\ge\m$, we focus on the eigenvalue cluster $(\lan{\m},\ldots,\lan{\M})$ and a set of the associated eigenvectors $(\un{\m},\ldots,\un{\M})$. Note that in this case, there actually holds $\VN \subset D(A)$ and not merely $\VN \subset D(A^{1/2})$ as supposed generally in~\eqref{eq:discretization}.

\subsubsection{Estimation of the dual norm of the residual}

In order to use the error estimates defined in Section~\ref{sec_equiv}, we need to estimate the Hilbert--Schmidt norm of the residual $\Resc$ defined in~\eqref{def:globalresidual}. As
\begin{equation} \label{eq_est_A}
    A \ge -\Delta+1 \ge 0, \quad A^{-1/2} \le (-\Delta+1)^{-1/2},
\end{equation}
the Hilbert--Schmidt norm of the residual can be estimated as follows, using the framework of Remark~\ref{rem_res_strong}.

\bc[Hilbert--Schmidt norm of the residual estimate] There holds
  \begin{equation} \label{eq:ineq_pot} \bs
      \|\Resc\|_{\HS}^2 & = \|A^{-1/2} \Resg\|_{\HS}^2 \le \|(-\Delta+1)^{-1/2} \Resg\|_{\HS}^2 \\
      & = \sum_{i=\m}^\M \norm{\Res(\un{i},\lan{i})}_{H^{-1}_\#(\Omega)}^2\\
      & = \sum_{i=\m}^\M \sup_{\substack{ v \in H^1_\#(\Omega)\\ \norm{v}_{H^1_\#(\Omega)} = 1}} \VpVS{\Res(\uh{i},\lah{i})}{v}.
  \es \end{equation}
\ec
Note that in the planewave setting, the Laplace operator is diagonal, so that in this case the quantity $\sum_{i=\m}^\M\norm{\Res(\un{i},\lan{i})}_{H^{-1}_\#(\Omega)}^2$ can actually be computed exactly at a negligible cost.

\br[Estimate~\eqref{eq:ineq_pot}]
Remark that inequality~\eqref{eq_est_A} is in fact independent of the choice of discretization. Therefore, \eqref{eq:ineq_pot} can also be used in the finite element setting, generalizing the estimates of Section~\ref{sec:Laplace_apost} for a Schr\"odinger operator on a torus.
\er

Actually, using the same argument, there holds
\begin{equation} \label{eq:reg_res_APW} \bs
    \| \A^{-1/2}\Resc \|_\HS^2 & =  \|A^{-1} \Resg\|_{\HS}^2 \le \|(-\Delta+1)^{-1} \Resg\|_{\HS}^2\\
     & = \sum_{i=\m}^\M\norm{\Res(\un{i},\lan{i})}_{H^{-2}_\#(\Omega)}^2.
\es \end{equation}
Since for $i=\m,\ldots,\M$, $\Res(\un{i},\lan{i}) \in \VN^\perp$, the orthogonal space of $\VN$ with respect to any $H^s_\#$ scalar product, there holds
\[
    \norm{\Res(\un{i},\lan{i})}_{H^{-s}_\#(\Omega)}^2
    = \sum_{\substack{\kk\in\RRs \\ |\kk| \ge \frac{2\pi}{L} N}} (1+|\kk|^2)^{-s}
    |\hat r^i_\kk|^2,
\]
where for $i=\m,\ldots,\M$, $(\hat r^i_\kk)_{\kk\in \RRs}$ are the planewave coefficients of $\Res(\un{i},\lan{i})$. Hence,
\begin{equation}
    \label{eq:reg_res_PW}
    \norm{\Res(\un{i},\lan{i})}_{H^{-2}_\#(\Omega)}
    \le \frac{L}{2\pi} \frac{1}{N} \norm{\Res(\un{i},\lan{i})}_{H^{-1}_\#(\Omega)}.
\end{equation}

\subsubsection{Final estimates} \label{sec_PW_fin}

We now state the guaranteed and fully computable error bounds for eigenvalues and density matrices of the operator $-\Delta+V$ discretized with planewaves.

\bt[Guaranteed bounds for the sum of eigenvalues] \label{thm_bound_eigs_PW}
Let $\m,\M\in\N\backslash \{0\}, M \ge m$, and let Assumption~\ref{as:gap} hold. For $i=\m,\ldots,\M$, let $(\un{i},\lan{i})\in \VN \times \R_+$ be defined in~\eqref{eq_PW}.
 Let  $ \ula{\M+1}$,  and $\ola{\m-1}$ if $\m>1$
 satisfying Assumption~\ref{as:gap_cont_disc} respectively with $\lan{\m}$ and $\lan{\M}$ in place of $\lah{\m}$ and $\lah{\M}$, i.e.
\begin{equation}
   \label{eq_set_pract_PW}
    \la{\m-1} \le \ola{\m-1} < \lan{\m} \text{ when } \m>1,
    \qquad
    \lan{\M} < \ula{\M+1} \le \la{\M+1},
\end{equation}
where we take $\ola{\m-1} = \lan{\m-1}$.
For $i=\m,\ldots,\M$, define
\[
    \eres^2 \eq \sum_{i=\m}^\M \norm{\Res(\un{i},\lan{i})}_{H^{-1}_\#(\Omega)}^2.
\]
Set
\begin{equation}
    c_{N}  \eq \max \left[	\left(\frac{\lan{\m}}{\ola{\m-1}}-1\right)^{-1} , \left(1 - \frac{\lan{\M}}{\ula{\M+1}} \right)^{-1}
              \right], \label{eq_cin}
\end{equation}
with the first term in the $\max$ discarded if $\m=1$. Then
\begin{equation}
   \label{eq_eig_bound_PW}
    0 \le \sum_{i=\m}^\M (\lah{i} - \la{i}) \le \eta^2,
\end{equation}
where
\begin{equation}
   \label{eq_eig_down_PW}
      \eta^2 \eq \left(1+  \frac{1}{N^2}  \frac{L^2\lan{\M}}{\pi^2}c_{N}^2 \right) \eres^2.
\end{equation}
\et

\bp Combining the estimates~\eqref{eq:estimlambda}, \eqref{eq:3.5.1}, \eqref{eq:L2A-1}, \eqref{eq:ineq_pot}, \eqref{eq:reg_res_APW}, and~\eqref{eq:reg_res_PW} yields the result.
\ep

Please note that in practice, condition~\eqref{eq_set_pract_PW} can be verified as in Remark~\ref{rem_pract}.

\bt[Guaranteed and robust bound for the density matrix errors]
\label{thm_bound_eigvecs_PW} Let the assumptions of Theorem~\ref{thm_bound_eigs_PW}
be verified. Then the energy density matrix error can be bounded via
\begin{equation}
       \norm{(-\Delta+V)^{1/2} (\gam - \gamh)}_\HS \leq \eta, \label{eq_eigvec_rel_PW}
\end{equation}
where $\eta$ is defined by~\eqref{eq_eig_down_PW}.
Moreover, the density matrix error can be bounded by
\begin{equation}
  \label{eq:gam_err_PW}
  \norm{\gam - \gamh}_\HS \le \eta_{L^2} \eq
  \sqrt{2} c_N \frac{L}{2\pi N} \eres.
\end{equation}
Recall finally the definition of $\chb$ by~\eqref{eq_chb}, with $\lan{\M}$ in place of $\lah{\M}$. Under Assumption~\ref{as:nonorth}, the
estimator~$\eta$ is efficient as
\begin{equation} \label{eq_eigvec_eff_PW} \bs
  \eres^2 \le {} & (\sup_{\Omega} V)
  \Bigg( \chb
  \| (-\Delta+V)^{1/2} (\gam - \gamh)\|^2_\HS +  \frac{3 (\la{\M} -\la{\m})^2}{4\la{\m}} \|\gam - \gamh\|_\HS^4 \\
  {} & + \frac{3}{\la{\m}} \left( 1 + \frac{1}{4}\|\gam - \gamh\|_\HS^4 \right)\times \\
  {} & \Big[2   \left( 1+ \frac{\la{\M}}{4\la{\m}}\|\gam - \gamh\|_{\HS}^2 \right)^2
  \norm{(-\Delta+V)^{1/2}(\gam-\gamh)}_{\HS}^4 
   \\ {} &
  + 2(\la{\M})^2
  \|\gam - \gamh\|_\HS^4\Big] \Bigg).
\es \end{equation}
\et

\bp
The proof of~\eqref{eq_eigvec_rel_PW} follows from the proof of~\eqref{eq_eig_bound_PW}. The estimate~\eqref{eq:gam_err_PW} follows from~\eqref{eq:L2A-1}, \eqref{eq:reg_res_APW}, and~\eqref{eq:reg_res_PW}.
Finally, the bound~\eqref{eq_eigvec_eff_PW} is a consequence of~\eqref{eq:3.5.2} and the inequality
\[
    \forall v\in H^1_\#(\Omega), \quad \|(-\Delta+V)^{1/2}v\|_\sHH^2
    = \|\Gr v \|_{L^2_\#}^2 + \int_\Omega V v^2
     \le (\sup_\Omega V) \|v\|_{H^1_\#(\Omega)}^2,
\]
which yields
\[
    \forall v\in H^{-1}_\#(\Omega), \quad \|(-\Delta+V)^{-1/2}v\|_\sHH^2 \ge \frac{1}{(\sup_\Omega V)} \|v\|_{H^{-1}_\#}^2.
\]
\ep

\section{Numerical experiments} \label{sec_num}

We now present some numerical results for two different examples. First, we perform simulations for the Laplace eigenvalue problem discretized with finite elements. Second, we show the estimates obtained for a Schr\"odinger operator on the torus discretized with planewaves.

\subsection{Laplace operator discretized with finite elements}
\label{sec:num_lap}
\begin{figure}[t]
	\centering
	\includegraphics[width=0.6\textwidth]{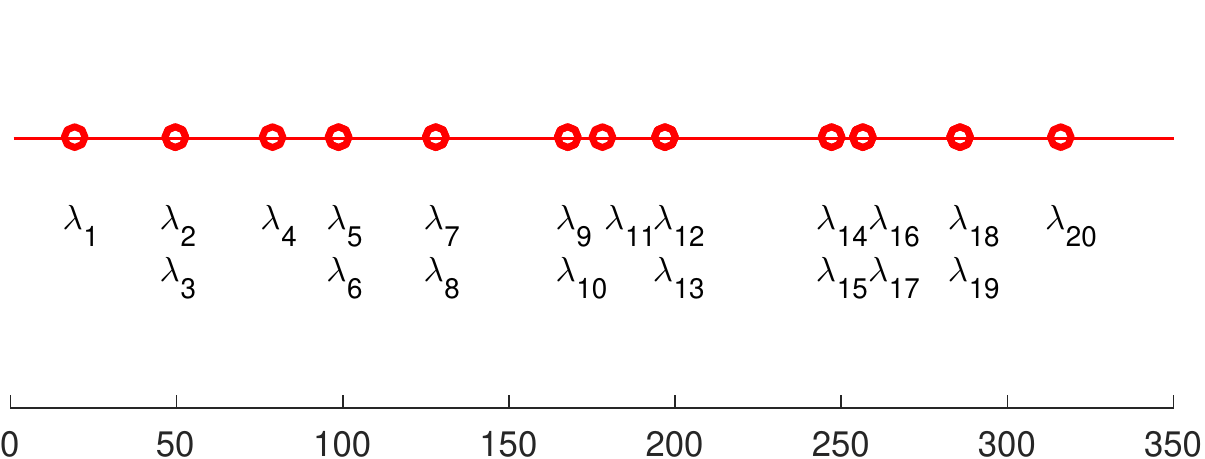}
	\caption{Plot of the first 20 eigenvalues of the Laplace operator with homogeneous Dirichlet boundary conditions on the unit square.}
	\label{fig:Spectrum}
\end{figure}
We start with a series of numerical examples using the conforming finite element method with piecewise linear polynomials, \ie, $p=1$, as presented in Section~\ref{sec:Laplace_apost} for the Laplace eigenvalue problem. We consider either the square domain $\Omega=(0,1)^2$ or an L-shaped domain with homogeneous Dirichlet conditions.
For the flux equilibration, we use the cheap Raviart--Thomas--N\'ed\'elec space of degree
$p=1$. This still provides guaranteed upper bounds, see the proof of Theorem~\ref{thm_res_bounds}, and we do not observe any asymptotic loss of the effectivity.
The numerical tests are performed with the FreeFem++ code~\cite{Hech_Pir_Hya_Oht_FreeFem++_12}.

Theorem~\ref{thm_bound_eigs} requires a lower bound $ \ula{\M+1}$, (recall that for the upper bound $\ola{\m-1}$ if $\m>1$, we simple use the numerically computed eigenvalue, \ie, $\ola{\m-1}=\lah{(\m-1)}\ge\la{\m-1}$, relying on the variational principle~\eqref{eq:conf_approx}).
A guaranteed lower bound $ \ula{\M+1}$ is obtained by employing the nonconforming finite element method on a coarse mesh $\mathcal{T}_{H}$ and using the technique presented in formula (6) of~\cite{Liu_fram_eigs_15}.

In the presentation of the results, we use the following notation:
\begin{align}
	{\tt Err}_{\lambda} \eq \sum_{i=\m}^\M (\lah{i} - \la{i}),
	\quad
	{\tt Err}_{H^1} \eq \norm{|\Gr| (\gam - \gamh)}_\HS,
	\quad
	{\tt Err}_{L^2} \eq \norm{\gam - \gamh}_\HS.
  \label{eq:notations}
\end{align}
The effectivity indices are then defined by
\[
	I^{{\tt eff}}_{\lambda} \eq \frac{\eta^2}{{\tt Err}_{\lambda}},
	\qquad
	I^{{\tt eff}}_{H^1} \eq \frac{\eta}{{\tt Err}_{H^1}},
	\qquad
	I^{{\tt eff}}_{L^2} \eq \frac{\etaltwo}{{\tt Err}_{L^2}},
\]
where $\eta$ and $\etaltwo$ are respectively defined in~\eqref{eq_eig_simpl_down} and~\eqref{etaltwo_I} for Case I and~\eqref{eq_eig_down} and~\eqref{etaltwo_II} for Case II.

\subsubsection{Unit square}
We first consider the unit square $\Omega= (0,1)^2$ where explicit eigenpairs are known.
Indeed, the sequence of eigenvalues is given by $\pi^2(k^2 + l^2)$, $k,l\in \mathbb N$, and the corresponding eigenvectors are $u_{k,l}=\sin(k\pi x)\sin(l \pi y)$.
The first few eigenvalues are therefore given by $\pi^2,5\pi^2,5\pi^2,8\pi^2,\ldots$ yielding a gap between the first and second, and the third and forth eigenvalues for example. Figure~\ref{fig:Spectrum} illustrates the first 20 eigenvalues and indicates the multiplicities.
For small eigenvalues, we use a coarse mesh $\mathcal{T}_{H,1}$ consisting of 121 triangles and 320 degrees of freedom and for larger eigenvalues, we use a second coarse mesh $\mathcal{T}_{H,2}$ consisting of 441 triangles and 1,240 degrees of freedom.
Since the domain is a convex polygon, we can apply {\bf Case~II} in Theorems~\ref{thm_bound_eigs} and~\ref{thm_bound_eigvecs} which exploits elliptic regularity results. We will here consider sequences of structured and uniformly refined meshes and use constants $C_{\rm I}
= \frac{0.493}{\sqrt 2}$, $C_{\rm S}=1$, and $\delta=1$ following Remark~\ref{rem_C_I_C_st}.

We first analyze the quality of the estimators for $m=2, M=3$.
The guaranteed lower bound is computed on the coarse mesh $\mathcal{T}_{H,1}$ yielding
$\ula{4} \approx 73.9444$.
Figure~\ref{fig:Conv} (top) illustrates the convergence of the error quantities ${\tt Err}_{\lambda}$, ${\tt Err}_{H^1}$, and ${\tt Err}_{L^2} $ as well as the corresponding upper bounds $\eta^2$, $\eta$, $\etaltwo$, whereas Table~\ref{tab:size2} (top) reports the effectivity indices.

\begin{figure}[t]
	\centering
   	\def\svgwidth{0.5\textwidth}
    \input{Square.pdf_tex}
	\hspace{0cm}
   	\def\svgwidth{0.5\textwidth}
    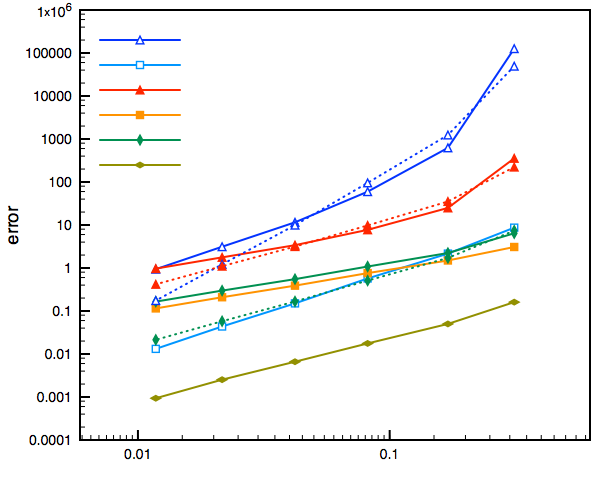
	\caption{Convergence of various measures of the error and their upper bounds for finite elements and the unit square with $m=2$, $M=3$ (top) and for the L-shaped domain with $m=3$, $M=5$ (bottom). 
	{
	The dotted lines provide the estimators using Case II with 
	$\delta=\frac23$ and the empirical choice $C_{\rm S}=C_{\rm I}=1$. }
	}
	\label{fig:Conv}
\end{figure}

\begin{table}[t]
	\footnotesize
	\setlength{\tabcolsep}{4pt}
	\centering
		\vspace{5pt}
      \hspace*{-1cm}\begin{tabular}{lrrrrrrrrrrrr}
		\hline
		\multicolumn{1}{c}{} &
		\multicolumn{1}{c}{$N$} &
		\multicolumn{1}{c}{$h$} &
		\multicolumn{1}{c}{{\tt ndof}} &
		\multicolumn{1}{c}{${\tt Err}_{\lambda}$} &
		\multicolumn{1}{c}{$\eta^2$} &
		\multicolumn{1}{c}{$I^{{\tt eff}}_{\lambda}$} &
		\multicolumn{1}{c}{${\tt Err}_{H^1}$} &
		\multicolumn{1}{c}{$\eta$} &
		\multicolumn{1}{c}{$I^{{\tt eff}}_{H^1}$} &
		\multicolumn{1}{c}{${\tt Err}_{L^2}$} &
		\multicolumn{1}{c}{$\etaltwo$} &
		\multicolumn{1}{c}{$I^{{\tt eff}}_{L^2}$}
		\Ttab\Btab
	 	\\
		\hline
$m=2$ 						&	40	 & 0.0354	 & 1681	 & 0.3351	 & 0.4661	 & 1.39	 & 0.5788	 & 0.6827	 & 1.18	 & 0.0041	 & 0.0183	 & 4.49 \\
$M=3$ 						&	80	 & 0.0177	 & 6561	 & 0.0837	 & 0.0972	 & 1.16	 & 0.2890	 & 0.3118	 & 1.08	 & 0.0010	 & 0.0046	 & 4.47 \\
$\mathcal T_{H,1}$ 	&	160	 & 0.0088	 & 25921	 & 0.0209	 & 0.0231	 & 1.10	 & 0.1445	 & 0.1521	 & 1.05	 & 0.0003	 & 0.0011	 & 4.49 \\
								&	320	 & 0.0044	 & 103041	 & 0.0052	 & 0.0057	 & 1.09	 & 0.0722	 & 0.0755	 & 1.05	 & 0.0001	 & 0.0003	 & 4.62 \\
		\hline
$m=9$ 						&	40	 & 0.0354	 & 1681	 & 3.2698	 & 3714.3421	 & 1135.96	 & 1.8235	 & 60.9454	 & 33.42	 & 0.0194	 & 0.3295	 & 17.01 \\
$M=10$					&	80	 & 0.0177	 & 6561	 & 0.8151	 & 76.6523	 & 94.04	 & 0.9037	 & 8.7551	 & 9.69	 & 0.0049	 & 0.0622	 & 12.81 \\
$\mathcal T_{H,2}$	& 160	 & 0.0088	 & 25921	 & 0.2036	 & 4.0755	 & 20.02	 & 0.4508	 & 2.0188	 & 4.48	 & 0.0012	 & 0.0148	 & 12.17 \\
								&	320	 & 0.0044	 & 103041	 & 0.0509	 & 0.2842	 & 5.58	 & 0.2253	 & 0.5331	 & 2.37	 & 0.0003	 & 0.0036	 & 12.03 \\
		\hline
$m=18$					&	40	 & 0.0354	 & 1681	 & 10.6565	 & 10777.4005	 & 1011.34	 & 3.4872	 & 103.8143	 & 29.77	 & 0.0729	 & 0.5069	 & 6.95 \\
$M=19$					&	80	 & 0.0177	 & 6561	 & 2.6465	 & 166.0018	 & 62.73	 & 1.6537	 & 12.8842	 & 7.79	 & 0.0183	 & 0.0887	 & 4.86 \\
$\mathcal T_{H,2}$	& 160	 & 0.0088	 & 25921	 & 0.6605	 & 8.7166	 & 13.20	 & 0.8152	 & 2.9524	 & 3.62	 & 0.0046	 & 0.0209	 & 4.57 \\
								&	320	 & 0.0044	 & 103041	 & 0.1651	 & 0.6511	 & 3.94	 & 0.4061	 & 0.8069	 & 1.99	 & 0.0011	 & 0.0051	 & 4.50 \\
		\hline
	\end{tabular}
	\caption{[Finite elements, unit square, Case~II]
	Errors, estimates, and effectivity indices for clusters of size 2 and increasing index of the eigenvalues.
	The values of $m$ and $M$ are indicated on the far left as well as the type of coarse mesh $\mathcal T_{H,i}$ used to obtain the auxiliary guaranteed lower bounds $\ula{\M+1}$.}
  \label{tab:size2}
\end{table}

We next analyze the effectivity indices of the estimator as we increase the index of the eigenvalues, still considering clusters of size 2. Table~\ref{tab:size2} (bottom) compares the results for the clusters corresponding to $\m=9,M=10$ and $m=18,M=19$.
Since higher eigenvalues are sought for, we considered the second coarse mesh $\mathcal T_{H,2}$ for computing $\ula{\M+1}$. This yields
\begin{align*}
	&
	\ula{11} \approx 	171.135, \quad
	\ula{20} \approx 	295.777.
\end{align*}
The results confirm the theoretical findings of Section~\ref{sec_FE_fin}. In particular, all the bounds are guaranteed, with effectivity indices taking values above one. Moreover, numerically, we observe asymptotic exactness of the estimators $\eta^2$ and $\eta$ of ${\tt Err}_{\lambda}$, respectively ${\tt Err}_{H^1}$, meaning that the corresponding effectivity indices $I^{{\tt eff}}_{\lambda}$ and $I^{{\tt eff}}_{H^1}$ tend to the optimal value of one. Additionally, we also numerically observe that the effectivity indices are robust with respect to increasing indices of the cluster of fixed size, which we could not cover in our theory. Indeed, for the efficiency bound~\eqref{eq_eigvec_eff} of Theorem~\ref{thm_bound_eigvecs}, the (exploding) factor $\chb$ appears.

Next, we consider clusters of increasing size. We consider the choices $m=1,M=4$ resp. $m=1,M=8$ and present the results in Table~\ref{tab:square_inc_size}.
We observe that the effectivity indices are also numerically robust when doubling the size of the cluster.
\begin{table}[t]
	\footnotesize
	\setlength{\tabcolsep}{3pt}
	\centering
		\vspace{5pt}
      \begin{tabular}{lrrrrrrrrrrrr}
		\hline
		\multicolumn{1}{c}{} &
		\multicolumn{1}{c}{$N$} &
		\multicolumn{1}{c}{$h$} &
		\multicolumn{1}{c}{{\tt ndof}} &
		\multicolumn{1}{c}{${\tt Err}_{\lambda}$} &
		\multicolumn{1}{c}{$\eta^2$} &
		\multicolumn{1}{c}{$I^{{\tt eff}}_{\lambda}$} &
		\multicolumn{1}{c}{${\tt Err}_{H^1}$} &
		\multicolumn{1}{c}{$\eta$} &
		\multicolumn{1}{c}{$I^{{\tt eff}}_{H^1}$} &
		\multicolumn{1}{c}{${\tt Err}_{L^2}$} &
		\multicolumn{1}{c}{$\etaltwo$} &
		\multicolumn{1}{c}{$I^{{\tt eff}}_{L^2}$}
		\Ttab\Btab
	 	\\
		\hline
$m=1$ 						&	10	 & 0.1414	 & 121	 & 13.5049	 & 21673.5051	 & 1604.86	 & 4.1325	 & 147.2192	 & 35.63	 & 0.2141	 & 1.7415	 & 8.13 \\
$M=4$ 						&	20	 & 0.0707	 & 441	 & 3.4018	 & 98.8430	 & 29.06	 & 1.9076	 & 9.9420	 & 5.21	 & 0.0554	 & 0.2274	 & 4.10 \\
$\mathcal T_{H,1}$ 	&	40	 & 0.0354	 & 1681	 & 0.8519	 & 5.0687	 & 5.95	 & 0.9297	 & 2.2514	 & 2.42	 & 0.0139	 & 0.0521	 & 3.75 \\
								&	80	 & 0.0177	 & 6561	 & 0.2131	 & 0.4708	 & 2.21	 & 0.4619	 & 0.6862	 & 1.49	 & 0.0035	 & 0.0128	 & 3.67 \\
								&	160	 & 0.0088	 & 25921	 & 0.0533	 & 0.0728	 & 1.37	 & 0.2306	 & 0.2698	 & 1.17	 & 0.0009	 & 0.0032	 & 3.67 \\
								&	320	 & 0.0044	 & 103041	 & 0.0133	 & 0.0155	 & 1.16	 & 0.1152	 & 0.1243	 & 1.08	 & 0.0002	 & 0.0008	 & 3.71 \\
		\hline
$m=1$ 						&	10	 & 0.1414	 & 121	 & 72.9222	 & 82403.2050	 & 1130.02	 & 9.3347	 & 287.0596	 & 30.75	 & 0.3359	 & 3.2521	 & 9.68 \\
$M=8$ 						&	20	 & 0.0707	 & 441	 & 18.0492	 & 281.4040	 & 15.59	 & 4.3588	 & 16.7751	 & 3.85	 & 0.0874	 & 0.3923	 & 4.49 \\
$\mathcal T_{H,2}$ 	&	40	 & 0.0354	 & 1681	 & 4.4994	 & 15.9735	 & 3.55	 & 2.1323	 & 3.9967	 & 1.87	 & 0.0221	 & 0.0893	 & 4.04 \\
								&	80	 & 0.0177	 & 6561	 & 1.1240	 & 1.8566	 & 1.65	 & 1.0603	 & 1.3626	 & 1.29	 & 0.0055	 & 0.0219	 & 3.94 \\
								&	160	 & 0.0088	 & 25921	 & 0.2810	 & 0.3445	 & 1.23	 & 0.5294	 & 0.5869	 & 1.11	 & 0.0014	 & 0.0054	 & 3.94 \\
								&	320	 & 0.0044	 & 103041	 & 0.0702	 & 0.0788	 & 1.12	 & 0.2646	 & 0.2808	 & 1.06	 & 0.0003	 & 0.0014	 & 4.00 \\
		\hline
	\end{tabular}
	\caption{[Finite elements, unit square, Case~II]
	Errors, estimates, and effectivity indices for clusters of increasing size.
	The values of $m$ and $M$ are indicated on the far left as well as the type of coarse mesh $\mathcal T_{H,i}$ used to obtain the auxiliary guaranteed lower bounds $\ula{\M+1}$.}
	\label{tab:square_inc_size}
\end{table}

\subsubsection{L-shaped domain}
We now address the case of an L-shaped domain $\Om \eq (-1,1)^2 \, \setminus \,
([0,1] \times [-1,0])$.
Note that in this setting, only {\bf Case I}, to our knowledge, is currently applicable in Theorems~\ref{thm_bound_eigs} and~\ref{thm_bound_eigvecs} to obtain guaranteed bounds (cf. Remark~\ref{rem_C_I_C_st}).
The first few eigenvalues are known to high accuracy~\cite{Tref_Betc_comp_eigs_plan_06}
\begin{align*}
		\la{1} \approx 9.6397238, \qquad
		\la{2} \approx 15.197252, \qquad
		\la{3} \approx 19.739209, \\
		\la{4} \approx 29.521481, \qquad
		\la{5} \approx 31.912636, \qquad
		\la{6} \approx 41.474510.
\end{align*}
We focus on the cluster from the third ($m=3$) to the fifth ($M=5$) eigenvalues.
A sequence of non-structured quasi-uniform meshes is considered first.
We test the estimator for the lower bounds of $\la{6}$ that are computed on two different coarse meshes $\mathcal T_{H,1}$ and $\mathcal T_{H,1}$, with 105 triangles resulting in 272 degrees of freedom resp. with 372 triangles resulting in 1033 degrees of freedom.
This yields a lower bound $\ula{6}\approx 34.0774$ resp. $ \ula{6}\approx 39.1209$.
The convergence plots are reported in Figure~\ref{fig:Conv} (bottom) for the latter case, and Table~\ref{tab:Lshaped} presents the effectivity indices in both cases.
We remark that the bound $\etaltwo$ for $\norm{\gam - \gamh}_\HS$ is guaranteed but of a much worse quality in this case, as the effectivity index $I^{{\tt eff}}_{L^2}$ increases with the number of degrees of freedom.

While Case I is always applicable and thus justified, Case II also applies theoretically with $\delta=\frac23$ for the L-shaped domain. 
The limiting issue is that the constants $C_{\rm S}$ and $C_I$ are unknown so that any empirical choice of these constants yields an error indicator but not a guaranteed estimator.
We have tested this indicator in Case II with  $\delta=\frac23$ and $C_{\rm S}=C_{\rm I}=1$ in order to obtain indicators that are no longer guaranteed but, on the other hand, have asymptotically the multiplicative pre-factor equal to 1, see also Remark~\ref{rem_C_I_C_st}.
Table~\ref{tab:Lshaped2} presents the effectivity indices that are now all decreasing (including the one of $\etaltwo$) and the dotted lines in Figure~\ref{fig:Conv} (bottom) present this indicator in Case II.
\color{black}

\begin{table}[t]
	\footnotesize
	\setlength{\tabcolsep}{3pt}
	\centering
		\vspace{5pt}
      \begin{tabular}{lrrrrrrrrrrrr}
		\hline
		\multicolumn{1}{c}{} &
		\multicolumn{1}{c}{$N$} &
		\multicolumn{1}{c}{$h$} &
		\multicolumn{1}{c}{{\tt ndof}} &
		\multicolumn{1}{c}{${\tt Err}_{\lambda}$} &
		\multicolumn{1}{c}{$\eta^2$} &
		\multicolumn{1}{c}{$I^{{\tt eff}}_{\lambda}$} &
		\multicolumn{1}{c}{${\tt Err}_{H^1}$} &
		\multicolumn{1}{c}{$\eta$} &
		\multicolumn{1}{c}{$I^{{\tt eff}}_{H^1}$} &
		\multicolumn{1}{c}{${\tt Err}_{L^2}$} &
		\multicolumn{1}{c}{$\etaltwo$} &
		\multicolumn{1}{c}{$I^{{\tt eff}}_{L^2}$}
		\Ttab\Btab
	 	\\
		\hline
$m=3$ 						&	20	 & 0.1703	 & 372	 & 2.1603	 & 320733.4214	 & 148468.40	 & 1.4948	 & 566.3333	 & 378.87	 & 0.0500	 & 5.1000	 & 101.92 \\
$M=5$ 						&	40	 & 0.0817	 & 1426	 & 0.5710	 & 3020.5208	 & 5289.65	 & 0.7607	 & 54.9593	 & 72.25	 & 0.0176	 & 2.0122	 & 114.26 \\
$\mathcal T_{H,1}$ 	&	80	 & 0.0421	 & 5734	 & 0.1503	 & 211.0547	 & 1403.82	 & 0.3886	 & 14.5277	 & 37.39	 & 0.0066	 & 0.9843	 & 148.78 \\
								&	160	 & 0.0216	 & 22001	 & 0.0436	 & 35.1498	 & 806.13	 & 0.2089	 & 5.9287	 & 28.38	 & 0.0025	 & 0.5277	 & 208.68 \\
								&	320	 & 0.0118	 & 86787	 & 0.0132	 & 8.7007	 & 661.24	 & 0.1149	 & 2.9497	 & 25.68	 & 0.0009	 & 0.2917	 & 311.83 \\
		\hline
$m=3$ 						&	10	 & 0.3124	 & 105	 & 8.6772	 & 126111.0898	 & 14533.55	 & 3.0801	 & 355.1212	 & 115.30	 & 0.1608	 & 6.4197	 & 39.93 \\
$M=5$ 						&	20	 & 0.1703	 & 372	 & 2.1603	 & 622.3367	 & 288.08	 & 1.4948	 & 24.9467	 & 16.69	 & 0.0500	 & 2.2311	 & 44.59 \\
$\mathcal T_{H,2}$ 	&	40	 & 0.0817	 & 1426	 & 0.5710	 & 59.5714	 & 104.32	 & 0.7607	 & 7.7182	 & 10.15	 & 0.0176	 & 1.0820	 & 61.44 \\
								&	80	 & 0.0421	 & 5734	 & 0.1503	 & 11.5424	 & 76.77	 & 0.3886	 & 3.3974	 & 8.74	 & 0.0066	 & 0.5505	 & 83.21 \\
								&	160	 & 0.0216	 & 22001	 & 0.0436	 & 3.1223	 & 71.61	 & 0.2089	 & 1.7670	 & 8.46	 & 0.0025	 & 0.2980	 & 117.86 \\
								&	320	 & 0.0118	 & 86787	 & 0.0132	 & 0.9370	 & 71.21	 & 0.1149	 & 0.9680	 & 8.43	 & 0.0009	 & 0.1652	 & 176.63 \\
		\hline
	\end{tabular}
	\caption{[Finite elements, L-shaped domain, Case~I]
	Errors, estimates, and effectivity indices for the cluster with $m=3$, $M=5$.
	The type of the coarse mesh $\mathcal T_{H,i}$ used to obtain the auxiliary guaranteed lower bounds $\ula{\M+1}$ is indicated on the far left.}
	\label{tab:Lshaped}
\end{table}

\begin{table}[t]
	\footnotesize
	\setlength{\tabcolsep}{3pt}
	\centering
		\vspace{5pt}
      \begin{tabular}{lrrrrrrrrrrrr}
		\hline
		\multicolumn{1}{c}{} &
		\multicolumn{1}{c}{$N$} &
		\multicolumn{1}{c}{$h$} &
		\multicolumn{1}{c}{{\tt ndof}} &
		\multicolumn{1}{c}{${\tt Err}_{\lambda}$} &
		\multicolumn{1}{c}{$\eta^2$} &
		\multicolumn{1}{c}{$I^{{\tt eff}}_{\lambda}$} &
		\multicolumn{1}{c}{${\tt Err}_{H^1}$} &
		\multicolumn{1}{c}{$\eta$} &
		\multicolumn{1}{c}{$I^{{\tt eff}}_{H^1}$} &
		\multicolumn{1}{c}{${\tt Err}_{L^2}$} &
		\multicolumn{1}{c}{$\etaltwo$} &
		\multicolumn{1}{c}{$I^{{\tt eff}}_{L^2}$}
		\Ttab\Btab
	 	\\
		\hline
$m=3$ 						&	20	 & 0.1703	 & 372	 & 2.1603	 & 29382.3983	 & 13601.19	 & 1.4948	 & 171.4129	 & 114.67	 & 0.0500	 & 3.7863	 & 75.67 \\
$M=5$ 						&	40	 & 0.0817	 & 1426	 & 0.5710	 & 987.9690	 & 1730.17	 & 0.7607	 & 31.4320	 & 41.32	 & 0.0176	 & 0.9157	 & 52.00 \\
$\mathcal T_{H,1}$ 	&	80	 & 0.0421	 & 5734	 & 0.1503	 & 86.9332	 & 578.23	 & 0.3886	 & 9.3238	 & 24.00	 & 0.0066	 & 0.2876	 & 43.47 \\
								&	160	 & 0.0216	 & 22001	 & 0.0436	 & 10.0040	 & 229.43	 & 0.2089	 & 3.1629	 & 15.14	 & 0.0025	 & 0.0989	 & 39.10 \\
								&	320	 & 0.0118	 & 86787	 & 0.0132	 & 1.3585	 & 103.25	 & 0.1149	 & 1.1656	 & 10.15	 & 0.0009	 & 0.0365	 & 38.99 \\
		\hline
$m=3$ 						&	10	 & 0.3124	 & 105	 & 8.6772	 & 49345.4041	 & 5686.76	 & 3.0801	 & 222.1383	 & 72.12	 & 0.1608	 & 7.3922	 & 45.98 \\
$M=5$ 						&  20	 & 0.1703	 & 372	 & 2.1603	 & 1237.8702	 & 573.01	 & 1.4948	 & 35.1834	 & 23.54	 & 0.0500	 & 1.7146	 & 34.27 \\
$\mathcal T_{H,2}$ 	&	40	 & 0.0817	 & 1426	 & 0.5710	 & 95.4292	 & 167.12	 & 0.7607	 & 9.7688	 & 12.84	 & 0.0176	 & 0.5097	 & 28.94 \\
								&	80	 & 0.0421	 & 5734	 & 0.1503	 & 9.9167	 & 65.96	 & 0.3886	 & 3.1491	 & 8.10	 & 0.0066	 & 0.1665	 & 25.17 \\
								&	160	 & 0.0216	 & 22001	 & 0.0436	 & 1.2135	 & 27.83	 & 0.2089	 & 1.1016	 & 5.27	 & 0.0025	 & 0.0578	 & 22.86 \\
								&	320	 & 0.0118	 & 86787	 & 0.0132	 & 0.1744	 & 13.25	 & 0.1149	 & 0.4176	 & 3.64	 & 0.0009	 & 0.0214	 & 22.86 \\
		\hline
	\end{tabular}
	\caption{[Finite elements, L-shaped domain, Case~II (with the empirical choice $C_{\rm S}=C_{\rm I}=1$)]
	Errors, estimates, and effectivity indices for the cluster with $m=3$, $M=5$.
	The type of the coarse mesh $\mathcal T_{H,i}$ used to obtain the auxiliary guaranteed lower bounds $\ula{\M+1}$ is indicated on the far left.}
	\label{tab:Lshaped2}
\end{table}

We finally test an adaptive refinement strategy using the local character of the density matrix estimator~\eqref{eq:GlobInd} (Case~I)
\[
	\eta^2 = \sum_{K \in \mathcal {T}_h} \eta_K^2
\]
with
\[
	\eta_K^2
	=
	(2 \chn^2+ 2\lah{\M} \chtn^4 \eres^2  ) \sum_{i=\m}^\M \norm{\Gr \uih + \frh{i}}^2_K.
\]
We employ the D\"orfler marking strategy~\cite{Dorf_cvg_FE_96} with $\theta=0.6$ and the newest vertex bisection mesh refinement.
The initial mesh is unstructured with 103 degrees of freedom.
The same lower bounds (using the mesh $\mathcal T_{H,2}$) as for the uniform refinement have been used and we note that Assumption~\ref{as:gap_cont_disc} is satisfied for the initial mesh. 

Figure~\ref{fig:LshapeAdaptivConv} illustrates the error in the eigenvalues and density matrix as well as their upper bounds (left) and the mesh at the 10th iteration of the adaptive mesh refinement procedure (right).
The optimal convergence rates are indicated by dashed lines. 
Table~\ref{tab:eff_adaptive_Lshaped} then presents more details including effectivity indices.
We observe quasi-optimal convergence with respect to the number of unknowns on the generated sequence of meshes and a loss of effectivity of $\etaltwo$ as in the case of uniform refinement.
\color{black}

\begin{table}[t]
	\footnotesize
	\setlength{\tabcolsep}{4pt}
	\centering
		\vspace{5pt}
	\begin{tabular}{rrrrrrrrrrrr}	
		\hline
		\multicolumn{1}{c}{Level} &
		\multicolumn{1}{c}{$h$} &
		\multicolumn{1}{c}{{\tt ndof}} &
		\multicolumn{1}{c}{${\tt Err}_{\lambda}$} &
		\multicolumn{1}{c}{$\eta^2$} &
		\multicolumn{1}{c}{$I^{{\tt eff}}_{\lambda}$} &
		\multicolumn{1}{c}{${\tt Err}_{H^1}$} &
		\multicolumn{1}{c}{$\eta$} &
		\multicolumn{1}{c}{$I^{{\tt eff}}_{H^1}$} &
		\multicolumn{1}{c}{${\tt Err}_{L^2}$} &
		\multicolumn{1}{c}{$\etaltwo$} &
		\multicolumn{1}{c}{$I^{{\tt eff}}_{L^2}$}
		\Ttab\Btab
	 	\\
		\hline
4	 & 0.2687	 & 311	 & 2.6872	 & 971.5901	 & 361.57	 & 1.6678	 & 31.1703	 & 18.69	 & 0.0569	 & 2.5161	 & 44.20 \\
8	 & 0.2000	 & 1503	 & 0.5247	 & 50.8697	 & 96.95	 & 0.7264	 & 7.1323	 & 9.82	 & 0.0114	 & 1.0228	 & 89.74 \\
12	 & 0.1000	 & 6944	 & 0.1086	 & 7.6657	 & 70.62	 & 0.3299	 & 2.7687	 & 8.39	 & 0.0024	 & 0.4568	 & 187.50 \\
16	 & 0.0500	 & 31608	 & 0.0233	 & 1.5273	 & 65.69	 & 0.1548	 & 1.2358	 & 7.98	 & 0.0005	 & 0.2103	 & 394.88 \\
		\hline
	\end{tabular}
	\caption{[Adaptive mesh refinement, L-shaped domain, Case~I] 
}
	\label{tab:eff_adaptive_Lshaped}
\end{table}

\begin{figure}[t]
	\centering
   	\def\svgwidth{0.5\textwidth}
    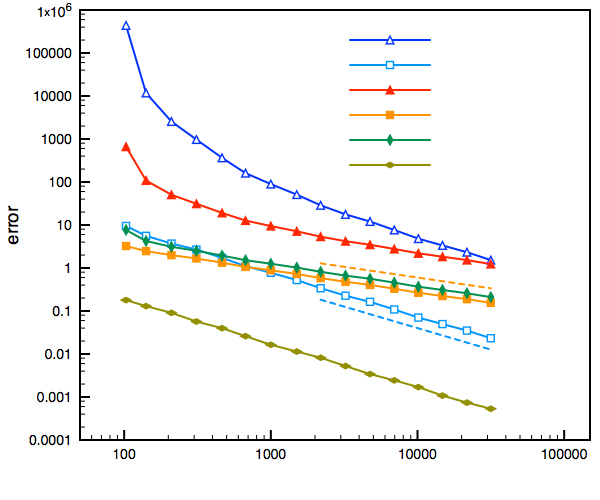
	\hspace{0cm}
   	\def\svgwidth{0.4\textwidth}
   	\setlength{\unitlength}{\svgwidth}
    \includegraphics[trim = 30mm -15mm 35mm 10mm, clip,width=\unitlength]{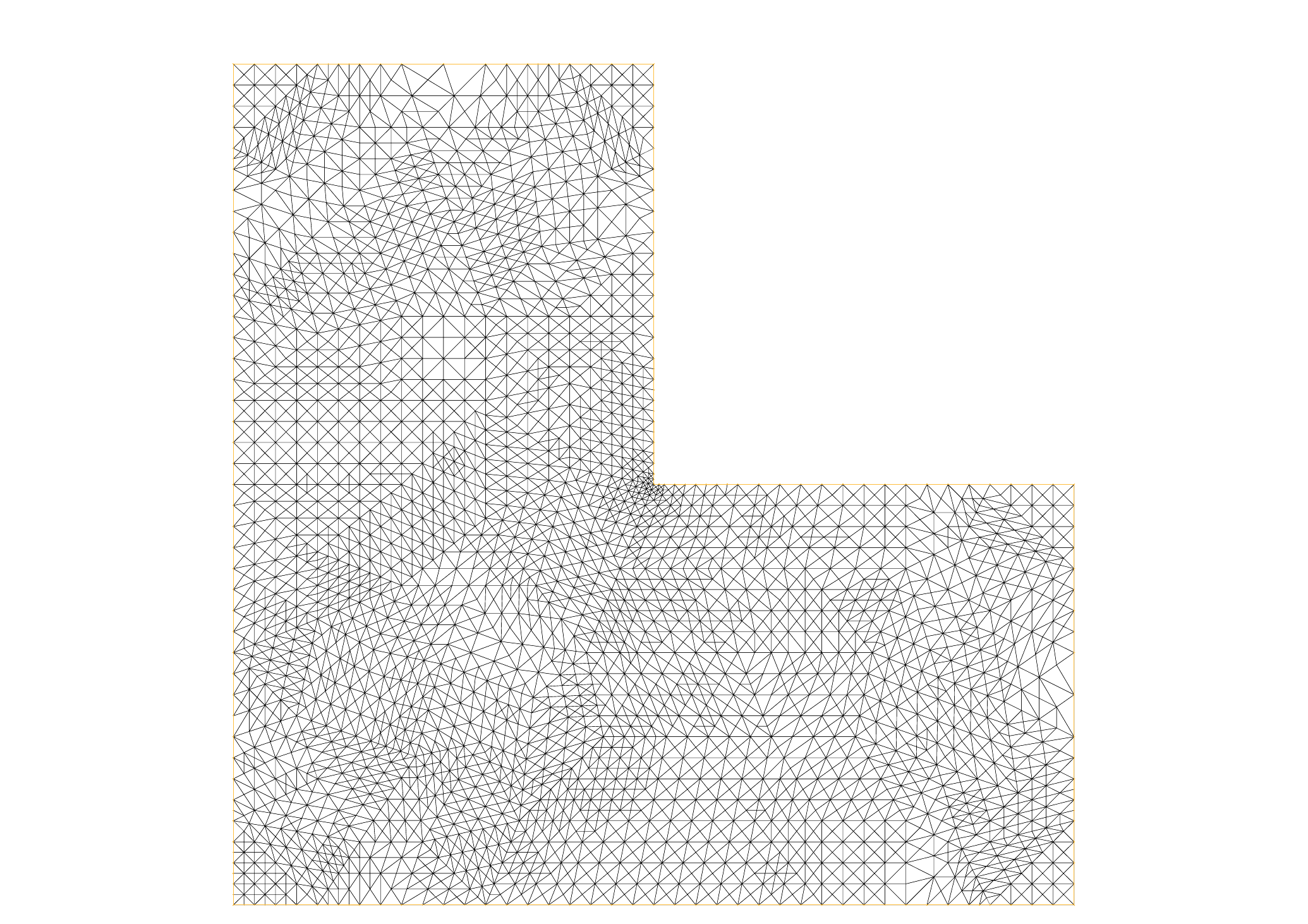}
	\caption{
	Convergence of various measures of the error and their upper bounds for finite elements on the L-shaped domain with $m=3$, $M=5$ using an adaptive refinement strategy (left).
	The mesh at the 10th iteration (right).
	}
	\label{fig:LshapeAdaptivConv}
\end{figure}


\subsection{A Schr\"odinger operator discretized with planewaves}

We consider in this section a Schr\"odinger operator of the form $-\Delta+V$ on $L^2_\#((0,2\pi)^d)$, where $V \in L^\infty_\#((0,2\pi)^d)$, $V \ge 1$, $d=1,2$. The problem is discretized with planewaves, and falls into the setting presented in Section~\ref{sec:PW_theory}. Using the notation of Section~\ref{sec:PW_theory}, $L=2\pi$. The potential $V$ is defined by its Fourier coefficients $\hat{V}_\kk$, $\kk\in \Z^d$,
which are of the form
\begin{equation}\label{eq:Fourier_coeff}
   \forall \kk\in \Z^d\backslash\{0\}, \quad \hat{V}_\kk = \frac{\alpha}{|\kk|^2},
\end{equation}
with $\alpha>0$ given and $\hat{V}_0$ such that $\min_{\xx\in (0,2\pi)^d} V(\xx) = 1$.

For the implementation of the bounds, note that the eigenvalues of the operator $-\Delta+1$, which are explicitly known, are lower bounds for the eigenvalues of $-\Delta+V$, since $V\ge 1$.
Moreover, the eigenvalues computed in the basis with planewave cutoff $N$ are upper bounds of the exact eigenvalues, due to the variational principle. The constant $c_N$ defined in~\eqref{eq_cin} can therefore be computed with these bounds of the eigenvalues.

The notation used here is similar to the notation of Section~\ref{sec:num_lap}, see~\eqref{eq:notations}. In particular, the effectivity indices are defined by
\[
	I^{{\tt eff}}_{\lambda} \eq \frac{\eta^2}{{\tt Err}_{\lambda}},
	\qquad
	I^{{\tt eff}}_{H^1} \eq \frac{\eta}{{\tt Err}_{H^1}},
	\qquad
	I^{{\tt eff}}_{L^2} \eq \frac{\etaltwo}{{\tt Err}_{L^2}},
\]
where $\eta$ and $\etaltwo$ are resp. defined in~\eqref{eq_eig_down_PW} and~\eqref{eq:gam_err_PW}.

\subsubsection{One-dimensional simulations}

In the following simulations, we take $d=1$ and $\alpha=1$.
For this potential, we compute reference eigenvectors and eigenvalues taking $N=600$. We then compute approximate eigenvectors for different values of the discretization parameter $N$ varying from 10 to 130. For all the chosen eigenvalue clusters, the assumptions required for Theorems~\ref{thm_bound_eigs_PW} and~\ref{thm_bound_eigvecs_PW} are already satisfied for $N=10$.

We first assess the quality of the estimators for $m=2, M=3$.
Figure~\ref{fig:Conv_PW} illustrates the convergence of the error quantities $	{\tt Err}_{\lambda}$, ${\tt Err}_{H^1}$, and ${\tt Err}_{L^2} $ as well as the corresponding upper bounds $\eta^2$, $\eta$, $\etaltwo$ and Table~\ref{tab:size3} (top) reports the corresponding effectivity indices.
We observe that the estimators $\eta^2$ and $\eta$ are numerically asymptotically exact.

We then consider clusters of increasing indices and increasing size. Namely, we take $m=10$, $M=11$ and $m=16$, $M=17$, as well as $m=1$, $M=9$ and $m=1$, $M=17$. The results presented in Table~\ref{tab:size3} confirm excellent efficiency and robustness of the bounds in all the considered situations.

\begin{table}[t]
	\footnotesize
	\setlength{\tabcolsep}{4pt}
	\centering
		\vspace{5pt}
	\begin{tabular}{lrrrrrrrrrrrr}
		\hline
		\multicolumn{1}{c}{} &
		\multicolumn{1}{c}{$N$} &
		\multicolumn{1}{c}{{\tt ndof}} &
		\multicolumn{1}{c}{${\tt Err}_{\lambda}$} &
		\multicolumn{1}{c}{$\eta^2$} &
		\multicolumn{1}{c}{$I^{{\tt eff}}_{\lambda}$} &
		\multicolumn{1}{c}{${\tt Err}_{H^1}$} &
		\multicolumn{1}{c}{$\eta$} &
		\multicolumn{1}{c}{$I^{{\tt eff}}_{H^1}$} &
		\multicolumn{1}{c}{${\tt Err}_{L^2}$} &
		\multicolumn{1}{c}{$\etaltwo$} &
		\multicolumn{1}{c}{$I^{{\tt eff}}_{L^2}$}
		\Ttab\Btab
	 	\\
		\hline
    $m=2$ & 10 &21 &  2.81e-06&  2.71e-05&  9.65&
            1.72e-03&  5.21e-03&  3.03&
            1.89e-04&  1.68e-03&  8.89 	  \\
    $m=3$ & 50 &101 &  1.18e-09&  1.59e-09&  1.35&
            3.44e-05&  3.99e-05&  1.16&
            8.12e-07&  6.91e-06&  8.50 	  \\
          & 90 &181 &  6.39e-11&  7.08e-11&  1.11&
          8.00e-06&  8.41e-06&  1.05&
          1.06e-07&  8.94e-07&  8.46 	\\
          & 130 & 261 &   1.02e-11&   1.08e-11&   1.05&
          3.20e-06&   3.28e-06&   1.03&
          2.93e-08&   2.48e-07&   8.44 	  	\\
		\hline
    $m=10$  & 10  & 21  &   2.74e-05 &   4.20e-04 &   15.3 &
              6.35e-03 &   2.05e-02 &   3.23 &
              6.83e-04 &   2.68e-03 	&   3.93   \\
    $M=11$  & 50  & 101  &   2.78e-09 &   4.70e-09 &   1.69 &   5.32e-05 &   6.85e-05 &   1.29 &   1.26e-06 &  5.91e-06 	&   4.70 	\\
        & 90  & 181  &   1.44e-10 &   1.75e-10 &   1.21 &   1.20e-05 &   1.32e-05 &   1.10 &   1.59e-07 &   7.48e-07 &	   4.71 \\
       &  130  & 261  &   2.29e-11 &   2.51e-11 &   1.10 &   4.78e-06 &   5.01e-06 &   1.05 &   4.38e-08 &   2.06e-07 	&   4.71 \\
		\hline
    $m=16$ &  10  & 21  &   4.41e-04 &   1.74e-02 &   39.5 &   3.67e-02 &   1.32e-01 &   3.59 &   3.69e-03 &   1.14e-02 	&   3.08 \\
      $M=17$  & 50  & 101  &   3.00e-09 &   1.19e-08 &   3.97 &   5.59e-05 &   1.09e-04 &   1.96 &   1.32e-06 &   8.21e-06 	&   6.22 \\
    & 90  & 181  &   1.43e-10 &   2.76e-10 &   1.93 &   1.20e-05 &   1.66e-05 &   1.38 &   1.59e-07 &   1.00e-06 &	   6.31 	 \\
      & 130 & 261 &   2.24e-11&   3.22e-11&   1.44&   4.73e-06&   5.68e-06&   1.20&   4.34e-08&   2.74e-07 	&   6.32 	\\
    \hline
    $m=1$ &  10  & 21  &   3.88e-05 &   3.82e-04 &   9.83 &   6.80e-03 &   1.95e-02 &   2.88 &   7.43e-04 &   3.04e-03 	&   4.09 	   \\
    $M=9$ & 50&	 101&	 1.02e-08 	& 1.41e-08 	& 1.38 	& 1.01e-04 	& 1.19e-04 	& 1.17 	& 2.40e-06 	& 1.02e-05 	&   4.25 	   \\
       &  90  & 181  &   5.46e-10 &   6.10e-10 &   1.12 &   2.34e-05 &   2.47e-05 &   1.06 &   3.09e-07 &   1.31e-06 &	   4.24 	   \\
      &  130  & 261  &   8.72e-11 &   9.22e-11 &   1.06 &   9.35e-06 &   9.60e-06 &   1.03 &   8.57e-08 &   3.62e-07 	&   4.23 	   \\
      \hline
      $m=1$ &  10  &21  &  3.18e-03 &  1.49e-01 &  46.9 &  1.18e-01 &  3.86e-01 &  3.28 &  1.15e-02 &  1.53e-02 	&   3.53 	   \\
      $M=17$ & 50 &101 &  2.49e-08&  1.36e-07&  5.49&  1.59e-04&  3.69e-04&  2.32&  3.76e-06&  2.22e-05 	&   6.32 	   \\
        &90 &181 &  1.26e-09&  3.02e-09&  2.39&  3.57e-05&  5.50e-05&  1.54&  4.71e-07&  2.81e-06 	&   6.34 	   \\
        &130 &  261 &    2.00e-10&    3.33e-10&    1.67&    1.41e-05&    1.82e-05&    1.29&    1.30e-07&    7.74e-07 	 &  6.34 	   \\
        \hline
	\end{tabular}
	\caption{[Planewaves, one-dimensional case, Schr\"odinger operator with $\alpha=1$ in~\eqref{eq:Fourier_coeff}]
	Errors, estimates, and effectivity indices for different clusters of eigenvalues.
	The values of $m$ and $M$ are indicated on the far left.}
  \label{tab:size3}
\end{table}

\begin{figure}
	\centering
\def\svgwidth{0.70\textwidth}
    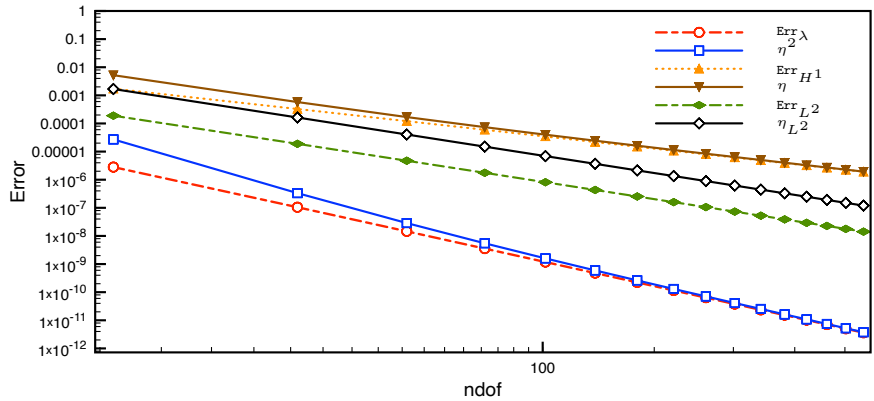
	\caption{Convergence of the errors and their upper bounds for a 1D Schr\"odinger operator with periodic boundary conditions with $m=2$, $M=3$.}
	\label{fig:Conv_PW}
\end{figure}

\subsubsection{Two-dimensional simulations}

We now take $d=2$ and first use $\alpha = 0.1$. For this potential, we compute reference eigenvectors and eigenvalues taking $N=50$, the number of degrees of freedom being $(2N+1)^2$. We then compute approximate eigenvectors and eigenvalues for different $N$ varying from 5 to 25. We compute the error bounds as well as the effectivity indices for different clusters of eigenvalues, namely $\m = 1$, $\M = 5$, then $\m = 6$, $\M = 9$, and finally $\m = 10$, $\M = 13$. The eigenvalue clusters are chosen such that the gaps between the cluster and the surrounding eigenvalues are rather large, in practice $> 0.87$. The results, presented in Table~\ref{tab:pw_d2_1}, confirm excellent accuracy of the bounds in this case as well.

\begin{table}[t]
	\footnotesize
	\setlength{\tabcolsep}{4pt}
	\centering
		\vspace{5pt}
	\begin{tabular}{lrrrrrrrrrrrr}
		\hline
		\multicolumn{1}{c}{} &
		\multicolumn{1}{c}{$N$} &
		\multicolumn{1}{c}{{\tt ndof}} &
		\multicolumn{1}{c}{${\tt Err}_{\lambda}$} &
		\multicolumn{1}{c}{$\eta^2$} &
		\multicolumn{1}{c}{$I^{{\tt eff}}_{\lambda}$} &
		\multicolumn{1}{c}{${\tt Err}_{H^1}$} &
		\multicolumn{1}{c}{$\eta$} &
		\multicolumn{1}{c}{$I^{{\tt eff}}_{H^1}$} &
		\multicolumn{1}{c}{${\tt Err}_{L^2}$} &
		\multicolumn{1}{c}{$\etaltwo$} &
		\multicolumn{1}{c}{$I^{{\tt eff}}_{L^2}$}
		\Ttab\Btab
	 	\\
		\hline
    $m=1$ & 5 & 121 &   2.62e-05&   2.02e-04&   7.70&   5.32e-03&   1.42e-02&   2.67&   9.94e-04&   6.18e-03 	&   6.22 	   \\
    $M=5$ & 15 & 961 &   4.12e-07&   7.31e-07&   1.77&   6.45e-04&   8.55e-04&   1.32&   4.47e-05&   2.62e-04 	&   5.85 	   \\
    & 25 & 2601 &   5.32e-08&   7.22e-08&   1.36&   2.31e-04&   2.69e-04&   1.16&   9.99e-06&   5.80e-05 	&   5.81	   \\
    \hline
    $m=6$ & 5 & 121 &   5.12e-05&   2.80e-04&   5.47&   7.60e-03&   1.67e-02&   2.20&   1.41e-03&   5.90e-03 	&   4.17	   \\
    $M=9$ &15 &961 &  7.51e-07&  1.15e-06&  1.53&  8.73e-04&  1.07e-03&  1.23&  6.05e-05&  2.43e-04 	&   4.02 	   \\
    & 25 &2601 &  9.63e-08&  1.22e-07&  1.26&  3.11e-04&  3.49e-04&  1.12&  1.35e-05&  5.38e-05 	&   4.00	   \\
   \hline
   $m=10$ & 5 &121 &  3.81e-05&  1.79e-03&  46.9 &  6.83e-03&  4.23e-02&  6.19&  1.28e-03&  1.30e-02 	  & 10.1   \\
    $M=13$ & 15 &961 &  4.47e-07&  2.93e-06&  6.55&  6.77e-04&  1.71e-03&  2.53&  4.69e-05&  4.87e-04 	 &  10.4 	   \\
   & 25 & 2601 &   5.64e-08&   1.80e-07&   3.18&   2.39e-04&   4.24e-04&   1.78&   1.03e-05&   1.07e-04 	&   10.4 	   \\
   \hline
	\end{tabular}
	\caption{[Planewaves, two-dimensional case, Schr\"odinger operator with $\alpha = 0.1$ in~\eqref{eq:Fourier_coeff}]
	Errors, estimates, and effectivity indices for different clusters of eigenvalues.
	The values of $m$ and $M$ are indicated on the far left.}
  \label{tab:pw_d2_1}
\end{table}

We, however, note that the parameter $\alpha$, which determines the amplitude of the potential, has a large influence on the efficiency of the bounds. In table~\ref{tab:pw_d2_2}, we present the error bounds and the effectivity indices in the setting $\alpha = 0.5$ for two clusters $\m = 6$, $\M = 9$ and  $\m = 10$, $\M = 13$. The efficiency is here reduced by one order of magnitude, though the assumptions required for the bounds to be valid are still satisfied from $N=5$ onwards.

\begin{table}[t]
	\footnotesize
	\setlength{\tabcolsep}{4pt}
	\centering
		\vspace{5pt}
	\begin{tabular}{lrrrrrrrrrrrr}
		\hline
		\multicolumn{1}{c}{} &
		\multicolumn{1}{c}{$N$} &
		\multicolumn{1}{c}{{\tt ndof}} &
		\multicolumn{1}{c}{${\tt Err}_{\lambda}$} &
		\multicolumn{1}{c}{$\eta^2$} &
		\multicolumn{1}{c}{$I^{{\tt eff}}_{\lambda}$} &
		\multicolumn{1}{c}{${\tt Err}_{H^1}$} &
		\multicolumn{1}{c}{$\eta$} &
		\multicolumn{1}{c}{$I^{{\tt eff}}_{H^1}$} &
		\multicolumn{1}{c}{${\tt Err}_{L^2}$} &
		\multicolumn{1}{c}{$\etaltwo$} &
		\multicolumn{1}{c}{$I^{{\tt eff}}_{L^2}$}
		\Ttab\Btab
	 	\\
		\hline
    $m=6$ &5 &121 &  2.33e-04&  9.87e-02&  424&  1.65e-02&  3.14e-01&  19.0&  2.90e-03&  1.01e-01 	&   35.0 	   \\
    $M=9$ &15 & 961 &   4.23e-06&   2.03e-04&   47.9&   2.08e-03&   1.42e-02&   6.84&   1.42e-04&   4.55e-03 	 &  32.0   \\
    & 25 &2601 &  5.57e-07&  1.05e-05&  18.8&  7.50e-04&  3.24e-03&  4.32&  3.22e-05&  1.02e-03 	&   31.6	   \\
        \hline
        $m=10$ &5 & 121 &   1.02e-04&   1.44e-01&   1410&   1.12e-02&   3.79e-01&   33.9&   2.00e-03&   1.78e-02&   8.89 	   \\
        $M=13$& 15 & 961 &   1.61e-06&   2.58e-04&   161&   1.29e-03&   1.61e-02&   12.5&   8.80e-05&   7.51e-04&   8.53 	   \\
        & 25 &2601 &  2.10e-07&  1.30e-05&  61.6&  4.61e-04&  3.60e-03&  7.81&  1.98e-05&  1.67e-04&  8.44 	   \\
        \hline
	\end{tabular}
	\caption{[Planewaves, two-dimensional case, Schr\"odinger operator with $\alpha = 0.5$ in~\eqref{eq:Fourier_coeff}]
	Errors, estimates, and effectivity indices for different clusters of eigenvalues.
	The values of $m$ and $M$ are indicated on the far left.}
  \label{tab:pw_d2_2}
\end{table}

\section{Conclusion} \label{sec:concl}

In this paper, we have introduced a new framework for error estimation in eigenvalue problems based on the density matrix formalism. This framework allows to deal with clusters of eigenvalues with possible degeneracies or near-degeneracies, as long as there is a gap between the considered eigenvalues and the rest of the spectrum. We propose {\it a posteriori} error estimates that are valid for conforming finite element and planewaves discretizations where in the first case, equilibrated flux reconstruction is used to bound the dual residual norms. The numerical results witness a very good quality of the derived methodology in a large set of test scenarios.

\appendix

\section*{Appendix}

We present the proof of~\eqref{eq:eqH1} from Lemma~\ref{lem:norm_eq2} in this appendix.

\section{Proof of \eqref{eq:eqH1} from Lemma~\ref{lem:norm_eq2}}
\label{app:proof_lemma}

\begin{proof}
To show~\eqref{eq:eqH1}, let us first express $\norm{\A^{1/2}(\Phie-\Phiho)}_{\sHH}^2$ and $\norm{\A^{1/2}(\gam-\gamh)}_{\HS}^2$.
From~\eqref{eq:Agamhgam} and~\eqref{eq:A12Phi}, there holds
\bse\begin{align}
   \norm{\A^{1/2}(\gam-\gamh)}_{\HS}^2
   = & \sum_{i=\m}^\M (\lah{i} - \la{i})
   + 2 \sum_{i=\m}^\M \la{i}
   \| (1-\gamh) \ue{i} \|_\sHH^2 \label{eq:Agamhgam_A}\\   \norm{\A^{1/2}(\Phie-\Phiho)}_{\sHH}^2 = &
   \sum_{i=\m}^\M (\lah{i} - \la{i}) + \sum_{i=\m}^\M \la{i} \norm{\ue{i} - \uho{i}}_\sHH^2. \label{eq:A12Phi_A}
\end{align}\ese
Since the projector $(1-\gamh)$ applied to any approximate eigenvector in the cluster is equal to zero, we obtain
\begin{equation}
   \label{eq:A1/2gam}
   \norm{\A^{1/2}(\gam-\gamh)}_{\HS}^2
 =  \sum_{i=\m}^\M (\lah{i} - \la{i})
 + 2 \sum_{i=\m}^\M \la{i}
 \| (1-\gamh) (\ue{i}-\uho{i}) \|_\sHH^2.
\end{equation}

To show the left inequality in~\eqref{eq:eqH1}, we use the fact that the operator norm of the projector $(1-\gamh)$ in $\mathcal{L}(\HH)$ is equal to 1 and~\eqref{eq:conf_approx} which, together with~\eqref{eq:A1/2gam}, yield
\[
   \norm{\A^{1/2}(\gam-\gamh)}_{\HS}^2 \le
   2 \sum_{i=\m}^\M (\lah{i} - \la{i}) + 2 \sum_{i=\m}^\M \la{i}
   \| \ue{i}-\uho{i} \|_\sHH^2
   = 2 \norm{\A^{1/2}(\Phie-\Phiho)}_{\sHH}^2.
\]

To show the right inequality in~\eqref{eq:eqH1}, we compute the difference
$\norm{\A^{1/2}(\Phie-\Phiho)}_{\sHH}^2 -  \norm{\A^{1/2}(\gam-\gamh)}_{\HS}^2$. Starting from~\eqref{eq:A12Phi_A} and~\eqref{eq:A1/2gam}, decomposing the identity as the sum of two orthogonal projectors $1 = \gamh + (1-\gamh)$, using \eqref{eq:gamma_v} from Lemma~\ref{lem_orth_proj}, we obtain
\begin{align*}
   \norm{\A^{1/2}(\Phie-\Phiho)}_{\sHH}^2 -  \norm{\A^{1/2}(\gam-\gamh)}_{\HS}^2 = {} &
    \sum_{i=\m}^\M \la{i} \norm{\ue{i} - \uho{i}}_\sHH^2 
    \\ & - 2 \sum_{i=\m}^\M \la{i}
   \| (1-\gamh) (\ue{i}-\uho{i}) \|_\sHH^2 \\
   \le {} &  \sum_{i=\m}^\M \la{i} \norm{\ue{i} - \uho{i}}_\sHH^2 
   \\ & -  \sum_{i=\m}^\M \la{i}
   \| (1-\gamh) (\ue{i}-\uho{i}) \|_\sHH^2 \\
   = {} & \sum_{i=\m}^\M \la{i} \norm{\gamh (\ue{i}-\uho{i})}_\sHH^2 \\
   = {} & \sum_{i,j=\m}^\M \la{i} \left| \left(\ue{i}-\uho{i},\uho{j}\right)_\sHH \right|^2.
\end{align*}
Further, applying $(\uho{j},\ue{i}-\uho{i})_\sHH = \frac{1}{2} (\uho{j}-\ue{j},\ue{i}-\uho{i})_\sHH$ which follows as~\eqref{eq:estimL2} and
using the Cauchy--Schwarz inequality together with~\eqref{eq:vecs}, we get
\begin{align*}
   \norm{\A^{1/2}(\Phie-\Phiho)}_{\sHH}^2 -  \norm{\A^{1/2}(\gam-\gamh)}_{\HS}^2 & \le \sum_{i,j=\m}^\M \frac{\la{i}}{4}  \|\ue{i}-\uho{i}\|^2_\sHH  \norm{\ue{j}-\uho{j}}_\sHH^2 \\ &
   \le  \frac{\la{\M}}{4}\norm{\Phie - \Phiho}_\sHH^4.
\end{align*}
Combining with~\eqref{eq:eqL2}, we obtain
\begin{align*}
  \norm{\A^{1/2}(\Phie-\Phiho)}_{\sHH}^2 -  \norm{\A^{1/2}(\gam-\gamh)}_{\HS}^2
  \le \frac{\la{\M}}{4} \norm{\gam - \gamh}_\HS^4.
\end{align*}
Also, using~\eqref{eq_A12ggh_dev} from the proof of Lemma~\ref{lem:2.6} together with $\norm{\ue{i}} = \norm{\uh{i}} = 1$ and using $(\gam)^2 = \gam$, $(\gamh)^2 = \gamh$, together with~\eqref{eq:conf_approx} and~\eqref{eq_Tr_2_scp}, \eqref{eq:Trgam2}  yields
\begin{align*}
  \norm{\A^{1/2}(\gam-\gamh)}_{\HS}^2 = {} & \sum_{i=\m}^\M \la{i}\left(  \ue{i},  \ue{i}   \right)_\sHH
  - 2 \sum_{i=\m}^\M  \la{i}\left(  \ue{i}, \gamh \ue{i}   \right)_\sHH
   +  \sum_{i=\m}^\M \lah{i}\left(  \uh{i}, \uh{i}   \right)_\sHH \nonumber \\
   \ge  {} & \sum_{i=\m}^\M \la{i} \left[ \left(  \ue{i},  \ue{i}   \right)_\sHH
   - 2  \left(  \ue{i}, \gamh \ue{i}   \right)_\sHH
    + \left(  \uh{i}, \uh{i}   \right)_\sHH \right] \nonumber \\
    \ge  {} & \la{\m}\sum_{i=\m}^\M  \left[ \left(  \ue{i},  \ue{i}   \right)_\sHH
    - 2  \left(  \ue{i}, \gamh \ue{i}   \right)_\sHH
     + \left(  \uh{i}, \uh{i}   \right)_\sHH \right] \nonumber \\
     =  {} & \la{\m} \left[ \Tr((\gam)^2) - 2\Tr(\gam\gamh) + \Tr((\gamh)^2)\right]\\
     =  {} & \la{\m} \left[ \norm{\gam}_{\HS}^2 -2 (\gam,\gamh)_{\HS}+ \norm{\gamh}_{\HS}^2 \right]\\
     =  {} & \la{\m} \norm{\gam-\gamh}_{\HS}^2,
\end{align*}
from which we deduce that
\begin{align*}
  \norm{\A^{1/2}(\Phie-\Phiho)}_{\sHH}^2 -  \norm{\A^{1/2}(\gam-\gamh)}_{\HS}^2
  \le \frac{\la{\M}}{4\la{\m}} \norm{\gam - \gamh}_\HS^2  \norm{\A^{1/2}(\gam-\gamh)}_\HS^2,
\end{align*}
which gives the right inequality of~\eqref{eq:eqH1}.
\end{proof}

\bibliographystyle{siam}
\bibliography{biblio5}

\end{document}

%% file: Square.pdf_tex
\begingroup%
  \makeatletter%
  \providecommand\color[2][]{%
    \errmessage{(Inkscape) Color is used for the text in Inkscape, but the package 'color.sty' is not loaded}%
    \renewcommand\color[2][]{}%
  }%
  \providecommand\transparent[1]{%
    \errmessage{(Inkscape) Transparency is used (non-zero) for the text in Inkscape, but the package 'transparent.sty' is not loaded}%
    \renewcommand\transparent[1]{}%
  }%
  \providecommand\rotatebox[2]{#2}%
  \ifx\svgwidth\undefined%
    \setlength{\unitlength}{570.12175bp}%
    \ifx\svgscale\undefined%
      \relax%
    \else%
      \setlength{\unitlength}{\unitlength * \real{\svgscale}}%
    \fi%
  \else%
    \setlength{\unitlength}{\svgwidth}%
  \fi%
  \global\let\svgwidth\undefined%
  \global\let\svgscale\undefined%
  \makeatother%
  \begin{picture}(1,1)%
    \put(0,0){\includegraphics[width=\unitlength]{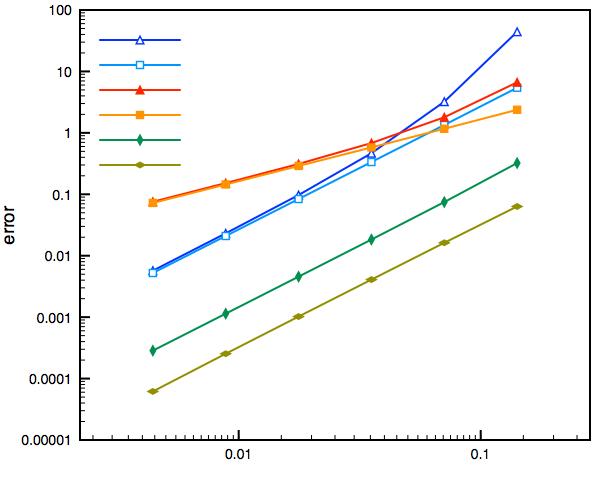}}%
      \put(0.55,0.03){\color[rgb]{0,0,0}\makebox(0,0)[lb]{\smash{\small $h$}}}%
      \put(0.33,0.765){\color[rgb]{0,0,0}\makebox(0,0)[lb]{\smash{\tiny $\eta^2$}}}%
      \put(0.33,0.725){\color[rgb]{0,0,0}\makebox(0,0)[lb]{\smash{\tiny ${\tt Err}_{\lambda}$}}}%
      \put(0.33,0.675){\color[rgb]{0,0,0}\makebox(0,0)[lb]{\smash{\tiny $\eta$}}}%
      \put(0.33,0.635){\color[rgb]{0,0,0}\makebox(0,0)[lb]{\smash{\tiny ${\tt Err}_{H^1}$}}}%
      \put(0.33,0.595){\color[rgb]{0,0,0}\makebox(0,0)[lb]{\smash{\tiny $\eta_{L^2}$}}}%
      \put(0.33,0.555){\color[rgb]{0,0,0}\makebox(0,0)[lb]{\smash{\tiny ${\tt Err}_{L^2}$}}}%
\end{picture}%
\endgroup%

%% file: Lshaped.pdf_tex
\begingroup%
  \makeatletter%
  \providecommand\color[2][]{%
    \errmessage{(Inkscape) Color is used for the text in Inkscape, but the package 'color.sty' is not loaded}%
    \renewcommand\color[2][]{}%
  }%
  \providecommand\transparent[1]{%
    \errmessage{(Inkscape) Transparency is used (non-zero) for the text in Inkscape, but the package 'transparent.sty' is not loaded}%
    \renewcommand\transparent[1]{}%
  }%
  \providecommand\rotatebox[2]{#2}%
  \ifx\svgwidth\undefined%
    \setlength{\unitlength}{570.12175bp}%
    \ifx\svgscale\undefined%
      \relax%
    \else%
      \setlength{\unitlength}{\unitlength * \real{\svgscale}}%
    \fi%
  \else%
    \setlength{\unitlength}{\svgwidth}%
  \fi%
  \global\let\svgwidth\undefined%
  \global\let\svgscale\undefined%
  \makeatother%
  \begin{picture}(1,1)%
    \put(0,0){\includegraphics[width=\unitlength]{Lshaped}}%
      \put(0.55,0.03){\color[rgb]{0,0,0}\makebox(0,0)[lb]{\smash{\small $h$}}}%
      \put(0.33,0.765){\color[rgb]{0,0,0}\makebox(0,0)[lb]{\smash{\tiny $\eta^2$}}}%
      \put(0.33,0.725){\color[rgb]{0,0,0}\makebox(0,0)[lb]{\smash{\tiny ${\tt Err}_{\lambda}$}}}%
      \put(0.33,0.675){\color[rgb]{0,0,0}\makebox(0,0)[lb]{\smash{\tiny $\eta$}}}%
      \put(0.33,0.635){\color[rgb]{0,0,0}\makebox(0,0)[lb]{\smash{\tiny ${\tt Err}_{H^1}$}}}%
      \put(0.33,0.595){\color[rgb]{0,0,0}\makebox(0,0)[lb]{\smash{\tiny $\eta_{L^2}$}}}%
      \put(0.33,0.555){\color[rgb]{0,0,0}\makebox(0,0)[lb]{\smash{\tiny ${\tt Err}_{L^2}$}}}%
\end{picture}%
\endgroup%

%% file: Lshaped_adaptive.pdf_tex
\begingroup%
  \makeatletter%
  \providecommand\color[2][]{%
    \errmessage{(Inkscape) Color is used for the text in Inkscape, but the package 'color.sty' is not loaded}%
    \renewcommand\color[2][]{}%
  }%
  \providecommand\transparent[1]{%
    \errmessage{(Inkscape) Transparency is used (non-zero) for the text in Inkscape, but the package 'transparent.sty' is not loaded}%
    \renewcommand\transparent[1]{}%
  }%
  \providecommand\rotatebox[2]{#2}%
  \ifx\svgwidth\undefined%
    \setlength{\unitlength}{570.12175bp}%
    \ifx\svgscale\undefined%
      \relax%
    \else%
      \setlength{\unitlength}{\unitlength * \real{\svgscale}}%
    \fi%
  \else%
    \setlength{\unitlength}{\svgwidth}%
  \fi%
  \global\let\svgwidth\undefined%
  \global\let\svgscale\undefined%
  \makeatother%
  \begin{picture}(1,1)%
    \put(0,0){\includegraphics[width=\unitlength]{Lshaped_adaptive}}%
      \put(0.55,0.03){\color[rgb]{0,0,0}\makebox(0,0)[lb]{\smash{\small ${\tt ndof}$}}}%
      \put(0.75,0.765){\color[rgb]{0,0,0}\makebox(0,0)[lb]{\smash{\tiny $\eta^2$}}}%
      \put(0.75,0.725){\color[rgb]{0,0,0}\makebox(0,0)[lb]{\smash{\tiny ${\tt Err}_{\lambda}$}}}%
      \put(0.75,0.675){\color[rgb]{0,0,0}\makebox(0,0)[lb]{\smash{\tiny $\eta$}}}%
      \put(0.75,0.635){\color[rgb]{0,0,0}\makebox(0,0)[lb]{\smash{\tiny ${\tt Err}_{H^1}$}}}%
      \put(0.75,0.595){\color[rgb]{0,0,0}\makebox(0,0)[lb]{\smash{\tiny $\eta_{L^2}$}}}%
      \put(0.75,0.555){\color[rgb]{0,0,0}\makebox(0,0)[lb]{\smash{\tiny ${\tt Err}_{L^2}$}}}%
     \put(0.81,0.36){\color[rgb]{0,0,0}\makebox(0,0)[lb]{\smash{\tiny $1/({\tt ndof})^{\tiny \frac12}$}}}%
      \put(0.81,0.22){\color[rgb]{0,0,0}\makebox(0,0)[lb]{\smash{\tiny $1/{\tt ndof}$}}}%
\end{picture}%
\endgroup%

%% file: Fourier.pdf_tex
\begingroup%
  \makeatletter%
  \providecommand\color[2][]{%
    \errmessage{(Inkscape) Color is used for the text in Inkscape, but the package 'color.sty' is not loaded}%
    \renewcommand\color[2][]{}%
  }%
  \providecommand\transparent[1]{%
    \errmessage{(Inkscape) Transparency is used (non-zero) for the text in Inkscape, but the package 'transparent.sty' is not loaded}%
    \renewcommand\transparent[1]{}%
  }%
  \providecommand\rotatebox[2]{#2}%
  \ifx\svgwidth\undefined%
    \setlength{\unitlength}{570.12175bp}%
    \ifx\svgscale\undefined%
      \relax%
    \else%
      \setlength{\unitlength}{\unitlength * \real{\svgscale}}%
    \fi%
  \else%
    \setlength{\unitlength}{\svgwidth}%
  \fi%
  \global\let\svgwidth\undefined%
  \global\let\svgscale\undefined%
  \makeatother%
  \begin{picture}(1,1)%
    \put(0,0){\includegraphics[width=\unitlength]{Fourier.pdf}}%
      \put(0.88,0.465){\color[rgb]{0,0,0}\makebox(0,0)[lb]{\smash{\tiny ${\tt Err}_{\lambda}$}}}%
      \put(0.88,0.445){\color[rgb]{0,0,0}\makebox(0,0)[lb]{\smash{\tiny $\eta^2$}}}%
      \put(0.88,0.42){\color[rgb]{0,0,0}\makebox(0,0)[lb]{\smash{\tiny ${\tt Err}_{H^1}$}}}%
      \put(0.88,0.403){\color[rgb]{0,0,0}\makebox(0,0)[lb]{\smash{\tiny $\eta$}}}%
      \put(0.88,0.38){\color[rgb]{0,0,0}\makebox(0,0)[lb]{\smash{\tiny ${\tt Err}_{L^2}$}}}%
      \put(0.88,0.36){\color[rgb]{0,0,0}\makebox(0,0)[lb]{\smash{\tiny $\eta_{L^2}$}}}%
\end{picture}%
\endgroup%